\documentclass[a4paper,11pt]{amsart}
\pdfoutput=1

\usepackage[a4paper, centering]{geometry}
\usepackage{enumitem}
\usepackage{amsmath}
\usepackage{amsfonts}
\usepackage{amsthm}
\usepackage{amssymb}
\usepackage{upref}
\usepackage{color}
\usepackage[svgnames]{xcolor}
\definecolor{light-salmon}{RGB}{255,140,120}

\usepackage[colorlinks]{hyperref}
\hypersetup{
	bookmarks=false,
	colorlinks=true,
	breaklinks=true,
	hyperindex=true,
	allbordercolors=.,
	citecolor=NavyBlue,
	linkcolor=DarkRed,
	urlcolor=NavyBlue
}
\usepackage{graphicx}
\usepackage{array}

\graphicspath{{./pictures/}}

\theoremstyle{plain}
\newtheorem{thm}{Theorem}[section]
\newtheorem{lemma}[thm]{Lemma}
\newtheorem{cor}[thm]{Corollary}
\newtheorem{prop}[thm]{Proposition}

\newtheorem{rem}[thm]{Remark}
\theoremstyle{definition}

\newcommand{\di}{\operatorname{div}}
\newcommand{\Pol}[1]{\mathcal P_{#1}}
\newcommand{\bo}[1]{{\bf#1}}
\newcommand{\Id}{\operatorname{\bo{Id}}}
\newcommand{\Hess}{\operatorname{Hess}}
\newcommand{\thetathree}{\vartheta}
\newcommand{\diam}{\operatorname{diam}}
\newcommand{\aaa}{a}
\newcommand{\bbb}{b}
\newcommand{\reg}{\text{reg}}
\newcommand{\sing}{\text{sing}}
\renewcommand{\Bbb}{\mathbb}

\usepackage{indentfirst} 
\newenvironment{ack}{{\bf Acknowledgements.}}

\newcommand{\ra}{\rightarrow}
\def\R{\mathbb{R}}

\def\N{\mathbb{N}}

\def\Z{\mathbb{Z}}

\def\Pp{{\mathbb{P}_n}}
\newcommand{\vps}{\varepsilon}
\newcommand{\UUU}{\color{red}}

\newcommand{\EEE}{\color{black}}
\newcommand{\Om}{\Omega}
\newcommand{\om}{\omega}
\newcommand{\vphi}{\varphi}	
\newcommand{\lb}{\lambda}

\newcommand{\sm}{\setminus}

\newcommand{\sq}{\subseteq}
\newcommand{\ov}{\overline}

\definecolor{darkred}{rgb}{0.55, 0.0, 0.0}

\title{On the Polygonal Faber-Krahn Inequality}
\author{Beniamin Bogosel, Dorin Bucur}

\address [Beniamin Bogosel]{CMAP UMR 7641 \'Ecole Polytechnique, CNRS, Route de Saclay, 91128 Palaiseau Cedex (France)}
\email {beniamin.bogosel@polytechnique.edu}

\address[Dorin Bucur]{
Laboratoire de Math\'ematiques UMR CNRS 5127 \\
Universit\'e  Savoie Mont Blanc,  Campus Scientifique \\
73376 Le-Bourget-Du-Lac (France)
}
\email{dorin.bucur@univ-savoie.fr}

\begin{document}

\begin{abstract}
It has been conjectured by P\'{o}lya and Szeg\" o seventy years ago that the planar set which minimizes the first eigenvalue of the Dirichlet-Laplace operator among polygons with $n$ sides and fixed area is the regular polygon. Despite its apparent simplicity, this result has only been proved for triangles and quadrilaterals. In this paper we prove that for each $n \ge 5$  the proof of the conjecture can be reduced to a finite number of certified numerical computations. Moreover, the local minimality of the regular polygon can be reduced to a single  numerical computation.  For $n=5, 6,7, 8$ we perform this computation and certify the numerical approximation by finite elements, up to machine errors. 
\end{abstract}

\maketitle

\tableofcontents
\section{Introduction}

 For every bounded, open set $\Omega \subset \Bbb{R}^2$ we consider the eigenvalue problem for the Laplace operator with Dirichlet boundary conditions
 \begin{equation}\left\{ \begin{array}{rcll}
   -\Delta u & =& \lambda u & \text{ in }\Omega,\\
   u &= &0 & \text{ on }\partial \Omega.  
 \end{array} \right.
 \label{eq:dir-eigenvalue}
 \end{equation}
 The spectrum consists only on eigenvalues, which can be ordered (counting the multiplicity),
 $$0< \lb_1(\Om) \le \lb_2(\Om) \le \dots \le \lb_k(\Om) \dots \ra +\infty.$$
 Lord Rayleigh conjectured in 1877 that the first eigenvalue is minimal on the disc, among all other planar domains of the same area. The proof was given in 1923 by  Faber in two dimensions and three years later extended by Krahn in any dimension of the Euclidean space (see \cite{Da11} for a description of the history of the problem and \cite{he17b,henroteigs} for a survey of the topic).  
 
 In their book of 1951, P\' olya  and Szeg\"o have conjectured a polygonal version of this inequality (see \cite[page 158]{PoSz51}).   Precisely,  denote by $\Pol{n}$ the family of  simple polygons with $n$ sides in $\Bbb{R}^2$ and for every $n \geq 3$ consider the problem 
\begin{equation}
 \min_{P \in \Pol n, |P| = \pi} \lambda_1(P).
 \label{eq:minimization-lamk}
\end{equation}
\medskip

\noindent {\bf P\' olya-Szeg\"o Conjecture (1951).} {\it The unique solution to problem \eqref{eq:minimization-lamk} is the regular polygon with $n$ sides and area $\pi$.
}
\medskip

This question, easy to state, has puzzeld many mathematicians in the last seventy years, but no significant progress has been made.
The conjecture holds true for $n=3$ and $n=4$. A proof can be found, for instance, in \cite{henroteigs}  as a straightforward application of the Steiner symmetrization principle (the original proof can be found in \cite{PoSz51}). However, Steiner symmetrization techniques do not allow the treatment of the case $n \geq 5$ since, performing this procedure, the number of vertices could possibly increase. We are not aware of further results regarding this conjecture. Neverteless, we mention a new approach, which applies only to triangles, proposed by Fragal\`{a} and Velichkov in \cite{FragalaVelichkov19},  establishing that equilateral triangles are the only critical points for the first eigenvalue. 

 A  question of the same nature, involving the logarithmic capacity, has been completely solved by Solynin and Zalgaler \cite{SoZa04} in 2004. The proof takes  full advantage from the specific structure of the problem, in particular from harmonicity of the capacitary functions; it can not be extended to eigenvalues. Minimization of variational energies in the class of polygons has been intensively investigated in the recent years (see the survey by Laugesen and Siudeja \cite{LaSi17} or  \cite{BuFr21} and references therein) but  the very specific polygonal version of the Faber-Krahn inequality remains unanswered. 
 
 It is quite straightforward to prove the existence of an optimal $n$-gon in the closure of the set of simple $n$-gons with respect to the Hausdorff distance of the complements, as shown in \cite[Chapter 3]{henroteigs}. It has  precisely $n$ edges, but it is possibly degenerate in the sense that a vertex could belong to another edge.  However, it is not even known that this polygon has to be convex! Meanwhile, many numerical experiments have been performed  for small values of $n$ (see for instance \cite{AFpolygons}, \cite[Chapter 1]{beniphd}) which all suggest the validity of the conjecture. 
 
 The purpose of this paper is twofold. A first objective is  to prove that local minimality of the regular polygon can be reduced to a single certified numerical computation. 
In fact, we prove  that the local minimality of the regular polygon is a consequence of the positivity of the eigenvalues of a $(2n-4) \times (2n-4)$ matrix related to the shape Hessian of the scale invariant functional ${\mathcal P}_n \ni P \ra|P|\lb_1(P)$. The dimension $2n-4$ reflects the number of degrees of freedom for $n-2$ vertices, onces two consecutive ones are fixed.  There are two challenges in this question: a theoretical one and a numerical one. First, one needs to prove that if the matrix is positive definite for the regular polygon then, for a neigbourhood of the regular polygon, the matrix remains positive definite. This question is itself not trivial  and requires  to take full advantage from the uniform $H^{2+s}$ regularity of the eigenfunctions for polygons which are small perturbations of the regular one. Secondly, in  the absence of theoretical results concerning the positivity of the eigenvalues of the Hessian matrix, one has to perform {\it certified} computations of the positive eigenvalues of the matrix, i.e. numerical computations with explicit error bounds that are sufficiently small. In our context the matrix coefficients depend on solutions of PDEs with singular right hand sides (in $H^{-1+ \gamma}$) involving the traces of the gradient of the first eigenfunction on the diameter of the polygon. We perform these computations for $n= 5, 6, 7, 8$ and certify the numerical approximation by finite elements, up to machine errors. In order to support the conjecture, we provide as well (uncertified) numerical computations for $n=9, \dots, 15$.

A second objective of our paper is to prove that for each $n \ge 5$ the complete proof of the conjecture can formally be reduced to a {\it finite} number of numerical computations. Roughly speaking, first, we analytically 
find a computable open neigbourhood of the regular polygon where the {local} minimality occurs. This requires a precise estimate of the modulus of continuity of the shape Hessian matrix obtained above, for small perturbations of the regular polygon. This is the most technical part of the paper. Second, we give a bound for the maximal possible diameter of the optimal polygon as well as for the minimal length edge and inradius, when its area is fixed. As a consequence,  {it remains} to prove that all polygons with free vertices in a (computable) compact set $K \sq \R^{2n-4}$ are not optimal. This can be done by performing a finite number of certified computations of first eigenvalues, areas and perimeters. Indeed, if a polygon has  vertices in the compact set $K$ and is not optimal, then either due to uniform estimates of the modulus of continuity of the eigenvalue and measure or to monotonicty of both these quantities to inclusions, non-optimality is certified in an open neigbourhood.  A finite number of such (open) neigbourdhoods will cover $K$. 
 \medskip
 
\noindent Le us  detail our strategy. 
\smallskip

 \noindent{\bf Step 1. (Formal computation of the shape Hessian matrix).} We interpret the first eigenvalue as a function depending on the coordinates of the vertices of the $n$-gon (obtaining a function defined on a subset of $\R^{2n}$) and choose an appropriate, equivalent scale invariant formulation for problem \eqref{eq:minimization-lamk}. Once the validity of the first order optimality condition on the regular polygon is established, we compute the analytic expression of the shape Hessian. For that purpose, we rely on the computations done by A. Laurain in \cite{laurain2ndDeriv} for the energy functional (we recall the corresponding result in  Remark \ref{bobu40}) and perform similar computations for the eigenvalue, following the same method.  
Taking perturbations of
 polygons with $n$ sides in the second shape derivative, we obtain the Hessian matrix (of size $2n\times 2n$) for the eigenvalue having the vertex coordinates as variables. 
  \medskip
 
 \noindent{\bf Step 2. (Numerical proof of the positivity of the shape Hessian matrix for the regular polygon, for a given $n$).} The shape Hessian matrix of the scale invariant functional has four eigenvalues equal to $0$, corresponding to the rigid motions and homotheties of the polygon.
 We use interval arithmetics and explicit error estimates for the finite element approximation to certify the positivity of the other eigenvalues of the shape Hessian matrix for the regular polygon with $n$ sides.  For $n=5,6,7,8$ and a suitable choice of an appropriate discretization, we certify, up to machine errors  appearing in the meshing, the assembly and the resolution of the linear systems in the finite element method, that the remaining $2n-4$ of the eigenvalues of the Hessian are strictly positive.  
 
 A fully certified (including machine errors aspects) positivity of the eigenvalues of the shape Hessian matrix is enough to prove the local minimality of the regular polygon, provided one knows that the coefficients of the matrix are continuous for small  geometric perturbations of the regular polygon.  This type of stability result is necessary to establish that the non zero eigenvalues remain positive in small neighborhood of the regular polygon. This is discussed in Step 3, below. By strict convexity, the regular polygon will be a minimizer in this neighborhood.
 
 \medskip
 
 \noindent{\bf Step 3. (Quantitative stability of the shape Hessian matrix coefficients).} 
  Our objective is to identify a computable neighborhood of the regular polygon where the eigenvalues of a $(2n-4)\times (2n-4)$ submatrix of the shape Hessian matrix remains positive. The most technical part is to give analytic, computable, estimates of the variation of the coefficients of the Hessian matrix, for perturbations of the regular polygon. The difficulty comes from the fact that the expression of the coefficients involve the solutions of some  (degenerate) elliptic PDEs with data in $H^{-1+\gamma}$, depending on traces of the gradient of the eigenfunctions on segments. The analysis requires quantitative estimates of the perturbation of the eigenfunction in $H^2$ which relies, via Gagliardo-Nirenberg interpolation inequalities, on control of their norm in $H^{2+s}$. 
  These estimates
  show that the unique, certified, computation of the Hessian matrix on the regular polygon is enough to obtain local minimality on a computable neignbourhood! 
  
  \medskip
 
 \noindent{\bf Step 4. (Analytic estimates of the maximal and minimal edge lengths of an optimal polygon). } We give a  computable  estimate of the maximal diameter of the optimal polygon,  provided its area is fixed. The estimate is inductively obtained for $n \ge 5$: if the diameter of an $n$-gon exceeds some (computable) value, then its eigenvalue is close to the one associated to a polygon with $n-1$ sides, so it can not be optimal in the class ${\mathcal P}_n$. Here we use surgery techniques inspired  from \cite{BuMa15}, but face the difficulty of keeping constant the number of sides within the surgery procedure.
 As well, we give an analytic estimate for the minimal length of an edge and of the minimal inradius.

  \medskip
 
 \noindent{\bf Step 5. (Formal proof of the conjecture).} We show how to give an inductive formal proof of the conjecture reducing it to a finite number of (certified) numerical computations for each value of $n$. Up to this point we have computed, for the scale invariant functional, a neighborhood of the regular polygon where its minimality occurs  and we have computed the maximal and minimal legnths of edges of an optimal polygon at prescribed area. Therefore, we are able to reduce the study of the conjecture to a family of polygons 
 with vertices belonging to a compact set. Any certified evaluation of the eigenvalue/area of such a polygon showing non optimality, would readily produce a small neigbourhood of non optimal polygons, the size of the neigbourhood being uniform and analitically computed. Monotonicity with respect to inclusions of both the eigenvalue and the area may be very useful from a practical point of view, but not necessary for a theoretical argument. Finally we get a ball covering  of a compact set which with known diameter, by balls of uniform size. This means that one can prove the conjecture after a finite number of numerical computations. We shall describe this procedure in Section \ref{bobu200}.

This type of numerical procedure has successfully been used in \cite{BF16} (to which we refer for a detailed description),  for a  different problem involving the same variational quantities but with only two degrees of freedom. The arguments transfer directly to our problem. 
 
 Although we prove that for a specific $n$ the proof of the  conjecture is reduced to a finite number computations, it is not our  purpose to perform these computations, for two reasons. On the one hand, all constants that we prove to exist should be optimized and effectively computed. On the other hand, even for $n=5$, in our procedure the number of degrees of freedom for the free vertices is 6 (see Section \ref{bobu200}). This demands huge computational capacities. In other words, before any computational tentative, some further, deep, analysis should be performed to dramatically reduce the  size of the computational tasks.

\medskip
 
   The structure of the paper is the following. Section \ref{bobu200.s2}
 is devoted to the computation of the shape Hessian of the  area and first eigenvalue functionals by a distributed formula. In particular, on polygons, we give the expression of the Hessian matrix of the eigenvalue as function of vertices coordinates. This section is inspired by the recent work of Laurain \cite{laurain2ndDeriv} for the energy functional. Section \ref{bobus3} contains a quantitative geometric stability result of the coefficients of the Hessian matrix with respect to vertex perturbations. This part is the key for the proof of the local minimality of the regular polygon and allows to estimate the size of the neighborhood of the regular polygon where minimality occurs. Sections  \ref{sec:eig-hessian} and \ref{bobu200.s5}  are devoted to the analysis of the shape Hessian matrix coefficients and to estimates regarding their numerical approximation.  Section \ref{bobu200.s6}  contains certified computation of the eigenvalues of the shape Hessian matrix on the regular polygon, justifying, up to machine errors, its local minimality for $n= 5, 6, 7, 8$. In Section \ref{bobu200} we give an estimate of the maximal diameter of an optimal polygon and show how the proof of the conjecture reduces, for every $n \in \N$,  to a finite number of numerical computations. As well, we make short comment about the polygonal Saint-Venant inequality for the torsional rigidity, which can be analyzed in a similar way. 

\section{First and Second order shape derivatives}\label{bobu200.s2}
  In this section we analyze the first  and the second order shape derivatives of the first Dirichlet eigenvalue,  for both general domains and for polygons. This section follows the strategy developed by Laurain in \cite{laurain2ndDeriv} for energy functionals (see Remark \ref{bobu40} for a brief summary of the corresponding results). Many proofs are very similar and we shall not reproduce them, referring to \cite{laurain2ndDeriv}, whenever necessary. Nevertheless, the formulae for the eigenvalues are different, so that we shall detail them.  The ultimate objective for polygons is to get an expression of the Hessian in distributed form involving sums over the two dimensional domains and remove any boundary integral expression. This is somehow contrary to what usually one does in shape optimization,  the main motivation being that the distributed expression of the second shape derivative requires less regularity hypotheses than the boundary expressions. This is  particularly useful for polygons. Finally, when restricted to the class of polygons with $n$ sides, we shall describe the shape Hessian of the eigenvalue by a square symmetric matrix of size $2n\times 2n$. 

In the literature one can find detailed descriptions of the shape gradients and shape Hessians  of the eigenvalue on a smooth set (see for instance \cite{henrot-pierre-english, HPR04, LNM16,DL19}).  The case of polygons is more delicate, since the boundary expression of the shape Hessian fails to have sense, due to the lack of regularity of the boundary.

In order to simplify the reading and the interpretation of potential connections between the results of this section and \cite{laurain2ndDeriv},  we use the same notations and, when the computations are similar, we prefer not to reproduce them and refer precisely to various sections in \cite{laurain2ndDeriv}.
\subsection{General domains}
 For vectors $a,b \in \Bbb{R}^d$ and matrices $\bo S, \bo T\in  \Bbb{R}^{d\times d}$ define the following:
\begin{itemize}[noitemsep,topsep=0pt]
   \item $\Id$ denotes the identity matrix
   \item $a \otimes b$ is the second order tensor of two vectors $(a\otimes b)_{ij} = a_ib_j$
   \item $a\odot b = \frac{1}{2} (a\otimes b+b\otimes a)$ is the 	 symmetric outer product.
   \item $a \cdot b$ is the usual scalar product 
   \item  $\bo S : \bo T = \sum_{i,j=1}^n S_{ij}T_{ij}$ is the matrix dot product.
\end{itemize}
It is immediate to notice that $(a\otimes b)c = (c\cdot b)a$ and $\bo S: (a\otimes b) = a \cdot \bo S b$.
\smallskip

Given a shape functional $\Omega \to J(\Omega)$ and a vector field $\theta \in W^{1,\infty}(\Bbb{R}^2,\Bbb{R}^2)$ the shape derivative of $J$ at $\Omega$, denoted by $J'(\Omega)\in \mathcal L(W^{1,\infty}(\Bbb{R}^2,\Bbb{R}^2), \R)$ is the Fr\'echet derivative of the application $\theta \mapsto J((I+\theta)(\Omega))$ and verifies
\[ J((I+\theta)(\Omega)) = J(\Omega)+J'(\Omega)(\theta) + o(\|\theta\|_{W^{1,\infty}}).\]
As discussed in \cite[Section 9.1]{laurain2ndDeriv}, when computing second order shape derivatives, several approaches are possible. The one detailed in \cite{laurain2ndDeriv} uses the Eulerian derivative in order to compute the Fr\'echet derivative. However, the Eulerian derivative requires more regularity on one of the perturbation fields than  $ W^{1,\infty}(\Bbb{R}^2,\Bbb{R}^2)$, while perturbations of polygons are precisely in $ W^{1,\infty}(\Bbb{R}^2,\Bbb{R}^2)$. 

For a given vector field $\theta \in W^{1,\infty}(\Bbb{R}^d,\Bbb{R}^d)$ consider the domain $\Omega_\theta = (I+\theta)(\Omega)$. It is well known that for $\|\theta \|_{W^{1,\infty}}<1$ this transformation is an invertible diffeomorphism. In the following, when dealing with boundary value problems, we use subscripts to denote functions $\varphi_\theta  \in H_0^1(\Omega_\theta)$ and superscripts to denote the functions $\varphi^\theta = \varphi_\theta \circ (I+\theta)\in H_0^1(\Omega)$.

The objective in the following is to have distributed expressions which require less regularity than the generally well known boundary expressions for the shape derivative of the eigenvalue (\cite{HPR04}, \cite{henrot-pierre-english}). Following the strategy of Laurain for the energy functional, we state below analogue results for the first and second Fr\'echet shape derivatives for the simple eigenvalues of the Dirichlet-Laplace problem \eqref{eq:dir-eigenvalue}. While some of these facts are  standard (for instance the expression of the first derivative), the expression of the Fr\'echet second derivative and the matrix representation in the case of polygons seem to be new.

 In the following we suppose $\theta$ is small enough such that $\lambda(\Omega_\theta)$ is still a simple eigenvalue. For simplicity, we do not write its index, which remains constant along the perturbation. Let $u_\theta \in H_0^1(\Omega_\theta)$ be the solution of
\begin{equation}
\int_{\Omega_\theta} \nabla u_\theta \cdot \nabla v_\theta \, dx= \lambda (\Omega_\theta) \int_{\Omega_\theta} u_\theta v_\theta \,dx, \ \forall v_\theta  \in H_0^1(\Omega_\theta)
\label{eq:eig-problem-moved}
\end{equation}
with the normalization $\int_{\Omega_\theta} (u_\theta)^2\, dx = 1$. Let  $u^\theta = u_\theta\circ (I+\theta)\in H_0^1(\Omega)$, so that $u_\theta = u^\theta\circ(I+\theta)^{-1}$. Then
\begin{equation}
\nabla u_\theta = [(I+D\theta^T)^{-1}\nabla u^\theta ] \circ (I+\theta)^{-1}
\label{eq:change-var-grad}
\end{equation}
and a change of variables leads to
\begin{eqnarray}
\int_\Omega A(\theta) \nabla u^\theta \cdot \nabla v \,dx = \lambda(\Omega_\theta) \int_\Omega u^\theta v \det(I+D\theta) \,dx, \text{ for all } v \in H_0^1(\Omega),
\label{eq:eig-problem-fixed}
\end{eqnarray} 
with the notation $A(\theta) = \det(I+D\theta)(I+D\theta)^{-1}(I+D\theta^T)^{-1}$.

Following  \cite[Theorem 5.7.4 ]{henrot-pierre-english}, the mapping 
\[ \theta \in W^{1,\infty} \mapsto (u^\theta, \lambda_k(\theta)) \in H_0^1(\Omega)\times \Bbb{R}\]
is of class $C^\infty$ on a neighborhood of $0$, without any smoothness requirement for $\Omega$. We differentiate \eqref{eq:eig-problem-fixed} at $0$ and denoting $\dot u(\theta) \in H_0^1(\Omega)$ the  material derivative, we obtain  for all $v \in H_0^1(\Omega)$
\begin{equation*}
\int_\Omega A'(0)(\theta)\nabla u \cdot \nabla v \, dx +\int_\Omega \nabla \dot u(\theta)\cdot \nabla v \, dx = \lambda'(\Omega)(\theta) \int_\Omega uv\, dx +\lambda(\Omega)\int_\Omega [\dot u(\theta)v +uv\di \theta]\, dx,
\end{equation*}
for all $v \in H_0^1(\Omega)$, where $A'(0)(\theta) = \di \theta \Id - D\theta-D\theta^T$. Regrouping terms gives
\begin{multline}
 \int_\Omega  \left(\nabla \dot u(\theta)\cdot \nabla v -\lambda(\Omega) \dot u(\theta)v\right)dx \\
 = \int_\Omega\left(- A'(0)(\theta)\nabla u \cdot \nabla v + \lambda'(\Omega)(\theta)  uv+\lambda(\Omega) uv\di \theta\right) dx,
\label{eq:material-fv-eig}
\end{multline}
for every $v \in H_0^1(\Omega)$. 
 Note that problem \eqref{eq:material-fv-eig} does not have a unique solution. Indeed, adding to $\dot u(\theta)$ any eigenfunction for problem \eqref{eq:dir-eigenvalue} associated to the eigenvalue $\lambda(\Omega)$ gives another solution. Uniqueness is a consequence of the normalization condition $\int_\Omega (u^\theta)^2 \det(I+D\theta)\, dx=1$. The corresponding derivative evaluated at zero is
\begin{equation}
\int_\Omega 2u\dot u(\theta)+u^2 \di \theta \, dx = 0.
\label{eq:deriv-norm-eig}
\end{equation}
When dealing with a simple eigenvalue, the additional condition \eqref{eq:deriv-norm-eig} is sufficient to uniquely identify $\dot u(\theta)$. For multiple eigenvalues, all eigenfunctions in the associated eigenspace should be used in \eqref{eq:deriv-norm-eig}.

With these notations we are ready to state the following result.

\begin{thm} Let $\Omega\subset \R^d$ be a bounded Lipschitz domain and $\theta,\xi \in W^{1,\infty}(\Bbb{R}^d,\Bbb{R}^d)$. Let $\lambda$ be a simple eigenvalue of the Dirichlet Laplacian and $u$ an associated $L^2$-normalized eigenfunction. Then
\begin{itemize}
\item [(i)] The distributed shape derivative of $\lambda$  is given by 
$$\lambda'(\Omega)(\theta) = \int_\Omega \bo S_1^\lambda : D\theta \, dx $$
 with $\bo S_1^\lambda = (|\nabla u|^2 -\lambda(\Omega) u^2)\Id -2\nabla u \otimes \nabla u$. If, in addition, $u \in H^2(\Omega)$, the corresponding boundary expression is 
 $$\lambda'(\Omega)(\theta) = -\int_{\partial \Omega} |\nabla u|^2\theta\cdot \bo n\, ds.$$
	
\item [(ii)] The second order distributed Fr\'echet derivative is given by
	$$ \lambda''(\Omega)(\theta,\xi)= \int_\Omega \mathcal K^\lambda (\theta,\xi)$$
	with 
	\begin{align*}
	\mathcal K^\lambda(\theta,\xi) & = -2\nabla \dot u(\theta)\cdot \nabla \dot u(\xi)+2\lambda(\Omega)\dot u(\theta) \dot u(\xi)+ \bo S_1^\lambda: (D\theta\di \xi+D\xi \di \theta) \\
	& +\left(-|\nabla u|^2+\lambda u^2\right) (\di \xi \di \theta+D\theta^T:D\xi)\\
	& +2 (D\theta D\xi+D\xi D\theta+D\xi D\theta^T)\nabla u\cdot \nabla u\\
	& -\big[\lambda'(\Omega)(\theta)\di \xi+\lambda'(\Omega)(\xi) \di \theta \big]u^2 .
	\end{align*}
	where $\dot u(\theta)$ and $\dot u(\xi)$ are the material derivatives in directions $\theta, \xi$, respectively.\label{thm:sh-deriv-eig}
\end{itemize}
\end{thm}

 The first point is standard and may be found in many classical references, for instance \cite{henrot-pierre-english}. Some formulae for the second derivative are also available in the literature, see \cite{henrot-pierre-english}, \cite{HPR04}.The key point is that the distributed expression shown above is valid for Lipschitz domains and Lipschitz perturbations. Moreover, being written in symmetric form its expression  helps in the computation of the Hessian matrix in the case of polygons. 
\medskip

\noindent {\it Proof of Theorem \ref{thm:sh-deriv-eig}}. 
The first application of formula \eqref{eq:material-fv-eig} is the expression of the first shape derivative. This computation is a classical result, but we present it here for the sake of completeness, since it illustrates well the techniques used when computing shape derivatives. Take $v=u$ in \eqref{eq:material-fv-eig} and note that, since $u$ is the eigenfunction associated to $\lambda(\Omega)$,
\[ \int_\Omega \nabla u \cdot \nabla \dot u(\theta) \, dx = \lambda(\Omega) \int_\Omega u \dot u(\theta)\, dx.\]
Using $\int_\Omega u^2\, dx=1$, we obtain 
\[ \int_\Omega \left( \di \theta |\nabla u|^2 -2 \nabla u\otimes \nabla u : D\theta \right) dx=\lambda'(\Omega)(\theta)+\int_\Omega \lambda(\Omega)u^2 \di \theta\, dx .\]
A direct computation leads to
\begin{equation}
\lambda'(\Omega)(\theta) = \int_\Omega [(|\nabla u|^2-\lambda(\Omega) u^2)\Id-2\nabla u\otimes \nabla u] : D\theta\, dx= \int_\Omega \bo S_1^\lambda : D\theta\, dx .
\label{eq:first-deriv}
\end{equation}
Now we choose $\xi \in W^{1,\infty}$ and we redo the same procedure to differentiate the first shape derivative \eqref{eq:first-deriv}. Denote $\Omega_\xi = (I+\xi)(\Omega)$ and suppose  that $\xi$ is small enough such that $\lambda(\Omega_\xi)$ is still a simple eigenvalue. Denote with $u_\xi \in H_0^1(\Omega_\xi)$ the eigenfunction associated to the simple eigenvalue $\lambda(\Omega_\xi)$. We have
\begin{equation}
\lambda'(\Omega_\xi)(\theta) = \int_{\Omega_\xi}  \big[(|\nabla u_\xi |^2-\lambda(\Omega_\xi) u_\xi^2)\Id-2\nabla u_\xi \otimes \nabla u_\xi \big] : D(\theta\circ (I+\xi)^{-1})\, dx.
\label{eq:first-deriv-moved}
\end{equation}
We also have the following elementary computation: $D(\theta \circ (I+\xi)^{-1}) = D\theta \circ (I+\xi)^{-1} D(I+\xi)^{-1}$.
As before, via a change of variables we write $\lambda'(\Omega_\xi)(\theta)$ as an integral on $\Omega$ defining $u^\xi = u_\xi\circ  (I+\xi) \in H_0^1(\Omega)$. Using \eqref{eq:change-var-grad} and performing a change of variables, we obtain
\begin{align*} \lambda'(\Omega_\xi)(\theta) &= \int_\Omega \big[\left((I+D\xi)^{-1}(I+D\xi)^{-T} \nabla u^\xi\cdot \nabla u^\xi -\lambda(\Omega_\xi) (u^\xi)^2\right)\Id \\
&-2 (I+D\xi)^{-T} \nabla u^\xi \otimes (I+D\xi)^{-T} \nabla u^\xi\big] :  D\theta D(I+\xi)^{-1} \det(I+D\xi)
\end{align*}

Now we are ready to compute the second Fr\'echet derivative of $\lambda(\Omega)$ by differentiating the previous expression w.r.t. $\xi$ at $0$ and denoting the derivative of $u^\xi$ at $0$ by $\dot u(\xi)$. We use the product rule, differentiating the first term, the term $D((I+\xi)^{-1})$ and finally $\det(I+D\xi)$. In particular, we have 
\begin{itemize}[topsep=0pt,noitemsep]
	\item $D((I+\zeta)^{-1})'_\zeta(0)(\xi) = - D\xi$.
	\item $\det(I+D\zeta)'_\zeta(0)(\xi) = \di(\xi)$.
\end{itemize}

We obtain the following initial formula for the second shape derivative:
\begin{align*}
\lambda''(\Omega)(\theta,\xi) & = \int_\Omega \bo S_1^\lambda : D\theta \di \xi - \int_\Omega \bo S_1^\lambda : D\theta D\xi  \, dx \\
&+ \int_\Omega [(-D\xi -D\xi^T) \nabla u \cdot \nabla u  + 2 \nabla \dot u(\xi)\cdot \nabla u] \di \theta\, dx \\
&-\int_\Omega [\lambda'(\Omega)(\xi) u^2 + \lambda(\Omega) 2 u \dot u(\xi)
]\di \theta \, dx \\ 
&+\int_\Omega [4 D\xi^T \nabla u \odot \nabla u:D\theta  -4 \nabla \dot u(\xi) \odot \nabla u:D\theta ]\, dx
\end{align*}
Following \cite[pag 25]{laurain2ndDeriv}, we have 
\[ -4(\nabla \dot u(\xi) \odot \nabla u):D\theta +2\nabla \dot u(\xi)\cdot \nabla u \di \theta = 2A'(0)(\theta)\nabla u \cdot \nabla \dot u(\xi)\]
and the material derivative \eqref{eq:material-fv-eig} gives
\begin{align*} 2\int_\Omega A'(0)(\theta) \nabla u \cdot \nabla  \dot u(\xi)\, dx &= -2\int_\Omega \nabla \dot u(\theta)\cdot \nabla \dot u(\xi) \, dx+2 \lambda'(\Omega)(\theta) \int_\Omega u\dot u(\xi)\, dx\\
&+2\lambda(\Omega)\int_\Omega \dot u(\theta)\dot u(\xi)\, dx +2 \lambda(\Omega) \int_\Omega u\dot u(\xi)\di \theta\, dx.
\end{align*}
The derivative of the normalization condition gives
\[ \int_\Omega 2u\dot u(\xi) \, dx= - \int_\Omega u^2 \di \xi\, dx.\]
We also have
\[\bo S_1^\lambda : D\theta D \xi = (|\nabla u|^2-\lambda u^2) D\theta^T : D\xi - 2D\theta D\xi \nabla u \cdot \nabla u,\]
since $\text{tr} (D\theta D\xi) = D\theta^T: D\xi$.
Combining all these expressions we obtain
\begin{align*}
\lambda''(\Omega)(\theta,\xi) & = -2\int_\Omega\left( \nabla \dot u(\theta)\cdot \nabla \dot u(\xi) - \lambda(\Omega)\dot u(\theta)\dot u(\xi)\right)dx + \int_\Omega \bo S_1^\lambda :D\theta \di \xi dx\\
&-\int_\Omega 2D\xi \nabla u \cdot \nabla u  \di \theta\, dx\\ 
&+\int_\Omega 4 D\xi^T \nabla u \odot \nabla u :D\theta\, dx \\
& - \int_\Omega [\lambda'(\Omega)(\theta) \di \xi +\lambda'(\Omega)(\xi)\di \theta] u^2\, dx\\
& +\int_\Omega (-|\nabla u|^2+\lambda u^2) D\theta ^T: D\xi +2 D\theta D\xi \nabla u \cdot \nabla u\, dx .
\end{align*}
We have
\[ 2D\xi^T \nabla u \odot \nabla u :D\theta = (D\xi D\theta + D\xi D\theta^T)\nabla u \cdot \nabla u.\]
Which gives
\begin{align*}
\lambda''(\Omega)(\theta,\xi) & = -2\int_\Omega \left(\nabla \dot u(\theta)\cdot \nabla \dot u(\xi) -\lambda(\Omega)\dot u(\theta)\dot u(\xi)\right)\, dx + \int_\Omega \bo S_1^\lambda :(D\theta \di \xi+D\xi \di \theta)\, dx \\
&+\int_\Omega (-|\nabla u|^2+\lambda u^2)(\di \theta \di \xi+D\theta^T:D\xi)\, dx\\ 
&+2\int_\Omega (D\theta D\xi +D\xi D\theta + D\xi D\theta^T)\nabla u \cdot \nabla u\, dx\\
& - \int_\Omega [\lambda'(\Omega)(\theta) \di \xi +\lambda'(\Omega)(\xi)\di \theta] u^2\, dx.
\end{align*}
This finishes the proof of the theorem. \hfill $\square$

\subsection{Polygons} 
In order to exploit the expression of Theorem \ref{thm:sh-deriv-eig}  in the case when $\Omega$ is a polygon, we follow again the strategy of Laurain \cite{laurain2ndDeriv} to extend a geometric perturbation of vertices to a global perturbation of the polygon. 
\medskip

 \noindent {\bf Vertex perturbation versus global perturbation.} Suppose $\Omega$ is a $n$-gon. Starting from a perturbation of the vertices, the perturbation field $\theta \in W^{1,\infty}(\R^2)$ will be built as follows. Denote the vertices of the polygon by $\bo a_i \in \Bbb{R}^2,\ i=0,...,n-1$ and for each vertex consider the vector perturbation $\theta_i \in \Bbb{R}^2,\ i=0,...,n-1$. Whenever necessary, we suppose that the indices are considered modulo $n$. Consider a triangulation $\mathcal T$ of $\Omega$ such that the edges of the polygon are complete edges of some triangles in this triangulation.  Moreover, consider the following globally Lipschitz functions  $\varphi_i$ for $0\leq i\leq n-1$ that are piecewise affine on each triangle of $\mathcal T$ and satisfy
	\begin{equation}\label{def:phi}
	 \varphi_i(\bo a_j) = \delta_{ij} = \begin{cases}
	1 & \text{ if } i=j \\
	0 & \text{ if } i\neq j
	\end{cases}
	\end{equation}
	Several choices are possible, as the two examples  of Figure \ref{fig:simple-triangulations} show, their extension outside the polygon being irrelevant. 
	Then, we build a global perturbation of $\R^2$ given by
	\begin{equation}\label{def:phi2}\theta = \sum_{i=0}^{n-1} \theta_i \varphi_i \in W^{1,\infty}(\R^2).
		\end{equation}

\begin{figure}
	\centering 
	\includegraphics[height=0.3\textwidth]{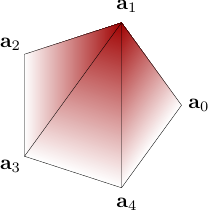}\quad
	\includegraphics[height=0.3\textwidth]{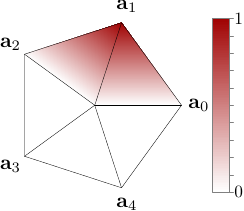}
	\caption{Examples of admissible triangulations used for defining perturbations on a polygon and graphical view of the function $\varphi_2$.}
	\label{fig:simple-triangulations}
\end{figure}

\noindent {\bf Gradient and Hessian of the area functional.} The shape derivatives for the area functional are classical and are widely studied in the literature (see \cite{henrot-pierre-english},\cite{laurain2ndDeriv}, etc.). The expression of the shape derivative of the area is
\begin{equation}
 |\Omega|'(\theta) = \int_{\partial\Omega} \theta \cdot \bo n.
\label{eq:sh-deriv-area}
\end{equation} However, in the particular case of $n$-gons the situation is much simpler, since explicit formulae exist in terms of the coordinates of the vertices of the polygon. For a non degenerate polygon whose coordinates of the vertices are denoted by $(x_i,y_i)$ and whose edges are oriented in the counter-clockwise order the area is given by
\[ \mathcal A(\bo x) = \frac{1}{2} \sum_{i=0}^{n-1} (x_iy_{i+1}-x_{i+1}y_i).\]
 The coordinates are regrouped in the vector by concatenating the coordinates of the vertices $\bo a_i$
\begin{equation}
 \bo x = (\bo a_0,\bo a_1,...,\bo a_{n-1}) = (x_0,y_0,...,x_{n-1},y_{n-1})\in \Bbb{R}^{2n},
 \label{eq:coordinates}
\end{equation}
which will always be the case in the following, when parametrizing polygons. 
The gradient of the area in terms of the coordinates verifies:
\[ \frac{\partial \mathcal A}{\partial x_i}(\bo x) = \frac{1}{2}(y_{i+1}-y_{i-1}), \quad 
\frac{\partial \mathcal A}{\partial y_i}(\bo x) = \frac{1}{2}(-x_{i+1}+x_{i-1}).\]
We denote by $\mathcal R_{\bo c,\alpha}$ the rotation around $\bo c \in \Bbb{R}^2$ with angle $\alpha$ (in the trigonometric sense), hence the gradient of the area has the    geometric expression
\begin{equation}
 \nabla \mathcal A(\bo x) = \frac{1}{2}\begin{pmatrix}
 \mathcal R_{\bo a_i,-\pi/2}(\overrightarrow{\bo a_{i-1}\bo a_{i+1}})
 \end{pmatrix}_{i=0,...,n-1}.
 \label{eq:grad-area}
\end{equation}
This is natural, since the area of the polygon when moving a vertex $\bo a_i$ only varies when moving the vertex $\bo a_i$ in the normal direction to the closest diagonal.

Another expression of the gradient of the area, using the functions $\varphi_i$ defined earlier, can be found following the results of \cite{laurain2ndDeriv} and is given by
\begin{equation}
\nabla \mathcal A(\bo x) = \left(\int_\Omega \nabla \varphi_i\right)_{i=0,...,n-1}.
\label{eq:grad-area-phi}
\end{equation}

Since the expression of the gradient of the area is linear in terms of the coordinates, the Hessian matrix of the area of the polygon is the constant $2n\times 2n$ block matrix
\begin{equation}
D^2 \mathcal A(\bo x) = \begin{pmatrix}
\bo B_{ij}
\end{pmatrix}_{0\leq i,j\leq n-1} 
\label{eq:Hess_area}
\end{equation}
where the non-zero $2\times 2$ blocks are given by $\bo B_{ij} = \begin{pmatrix}0& \frac12 \\ -\frac12 & 0 \end{pmatrix}$ if $j = i+1$ and $\bo B_{ij} = \begin{pmatrix}0& -\frac12 \\ \frac12 & 0 \end{pmatrix}$ if $i = j+1$. 

Following the results in \cite{laurain2ndDeriv} we find that the formula for the Hessian of the area in terms of the functions $\varphi_i$ can also be expressed using the following block structure
\begin{equation}
\bo B_{ij} = \int_\Omega [ \nabla \varphi_i \otimes \nabla \varphi_j - \nabla \varphi_j \otimes \nabla \varphi_i] 
\label{eq:hess_area_phi}.
\end{equation}
In particular the Hessian of the area can be written as a tensorial product (Kronecker product) between the matrices
\[ \begin{pmatrix}
0 & 1 & 0 & ... & 0 & -1 \\
-1 & 0 & 1 & ... & 0 & 0\\
\vdots & \vdots & \vdots & \ddots & \vdots & 
\vdots \\
1 & 0 & 0 & ... & -1 & 0 
\end{pmatrix} \text{ and } \begin{pmatrix}
0 & 0.5 \\ -0.5 & 0 
\end{pmatrix}\]
Therefore, the corresponding eigenvalues and eigenvectors can be found explicitly.

\medskip

\noindent {\bf Gradient of the eigenvalue.} Below we compute the gradient of the eigenvalue \eqref{eq:dir-eigenvalue} as function of the vertices, i.e. the partial derivatives of these functionals with respect to the coordinates of the vertices of the polygons. The expression of these gradients can be used  to prove that the regular polygon is a critical point under an area constraint and  are useful for  numerical computations.

The expression of the gradient of the eigenvalue with respect to the coordinates is a consequence of  the shape derivative formulae recalled in the previous section. It is enough to  use the distributed expression of the shape derivative, valid in general, with the perturbation field $\theta$ introduced in \eqref{def:phi2}. An example is given in Figure \ref{fig:vert-perturbation} for $\theta = \theta_i \varphi_i$.
\begin{figure}
	\centering
	\includegraphics[width=0.8\textwidth]{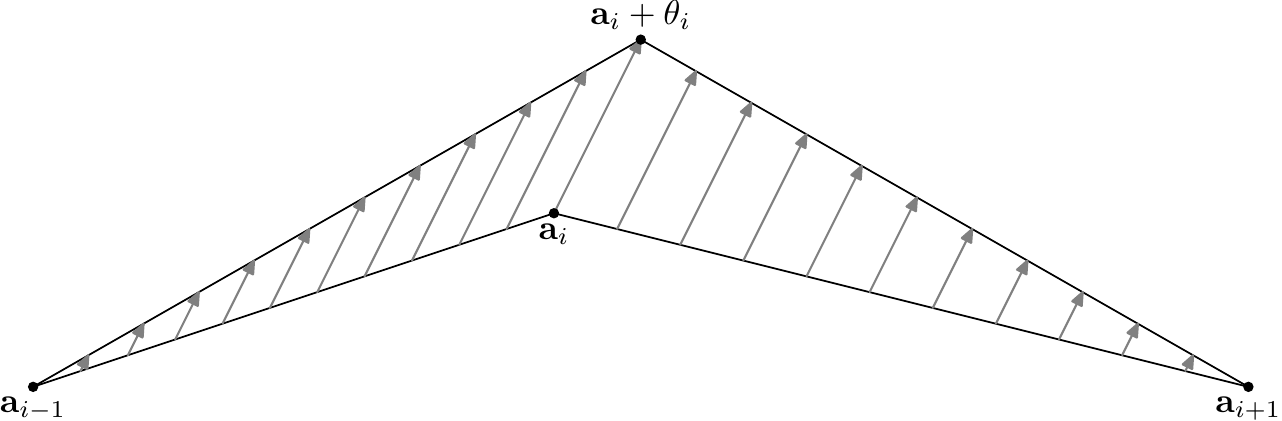}
	\caption{Boundary perturbation induced when perturbing a vertex}
	\label{fig:vert-perturbation}
\end{figure}
The proof is similar to the case of the torsion energy \cite{laurain2ndDeriv}. We choose to detail here only the boundary expression, with a slightly different argument than the one used in \cite{laurain2ndDeriv}. 
\begin{thm}
	The gradient of a simple Dirichlet-Laplace eigenvalue \eqref{eq:dir-eigenvalue} when $\Omega$ is a polygon with coordinates $\bo x$ as in \eqref{eq:coordinates} is given by 
	$$\nabla \lambda(\bo x) = \Big (\int_\Omega \bo S_1^\lambda \nabla \varphi_i \, dx\Big )_{i=0,...,n-1}= -\Big (\int_{\partial \Omega} |\nabla u|^2 \varphi_i \bo n \, ds\Big)_{i=0,...,n-1},$$
	  where $\bo n$ is the outer unit normal vector. 
	\label{thm:grad-eig}
\end{thm}
Notice that the boundary expression is  always valid, even though the eigenfunction itself does not belong to $H^2(\Om)$. This is a consequence of the fact that in an arbitrarry polygon (typically non convex), the eigenfunction enjoys a local $H^{2+\delta}$ regularity far from the corners, while at corners the singular part has a very specific structure, albeit good enough to make the boundary expression of the gradient valid. We recall from \cite{BDLN92} that
$$u=u_{\reg}+u_{\sing},$$
where $u_{\reg} \in H^{2+\delta}(\Om)$ for some $\delta >0$ and 
$$u_{\sing} =\sum_{i=0}^{n-1} C_i \psi _i r^\frac{\pi}{\om_i} \sin (\frac{\pi}{\om_i}\theta),$$
where
$C_i$ are constants, $\om_i$ are the angles, $\psi_i$ is cutoff function equal to $1$ in a neighborhood of the vertex $\bo a_i$ and $(r, \theta)$ are the polar coordinates around the angle $i$. 
\begin{proof}
The expression
$\nabla \lambda(\bo x) = \Big (\int_\Omega \bo S_1^\lambda \nabla \varphi_i\, dx\Big )_{i=0,...,n-1}$ is valid. It remains to prove the equality
$$\Big (\int_\Omega \bo S_1^\lambda \nabla \varphi_i\Big )_{i=0,...,n-1}= -\Big (\int_{\partial \Omega} |\nabla u|^2 \varphi_i \bo n\Big)_{i=0,...,n-1}.$$
First, note that the gradient of $u$ is point-wise defined on $\partial \Om$, except at the vertices, in a classical way. We fix a vertex $i$ and define 
$$\Om_\vps= \Om \setminus (\ov B(\bo a_{i-1} , \vps) \cup \ov B(\bo  a_{i} , \vps)\cup \ov B(\bo  a_{i+1} , \vps)),$$
$$\Gamma_\vps= \Om \cap  (\partial B(\bo a_{i-1} , \vps) \cup \partial B(\bo a_{i} , \vps)\cup \partial B(\bo a_{i+1} , \vps)).$$
Since $u|_{\Om_
\vps} \in H^2(\Om_\vps)$, a direct computation shows that $\di \bo S_1^\lambda=0$ on $\Omega_\varepsilon$, the divergence being applied on lines. Moreover, since $u=0$ on $\partial \Omega$, the gradient $\nabla u$ is colinear with the normal vector $\bo n$ on $\partial \Omega$. In particular, $(\nabla u \otimes \nabla u) \bo n = (\bo n \cdot \nabla u) \nabla u = |\nabla u |^2 \bo n$. As a consequence $\bo S_1^\lambda \bo n = -|\nabla u|^2 \bo n$. Therefore we obtain
\begin{align*}
\int_{\Omega_\varepsilon} \bo S_1^\lambda \nabla \varphi_i dx & = -\int_{\Omega_\varepsilon} \di (\bo S_1^\lambda) \varphi_i dx +\int_{\partial \Omega_\varepsilon} \bo S_1^\lambda \bo n \varphi_i= -\int_{\partial \Omega_\varepsilon \setminus \Gamma_\varepsilon} |\nabla u|^2 \varphi_i \bo n + \int_{\Gamma_\varepsilon} \bo S_1^\lambda \bo n \varphi_i.
\end{align*}
We conclude by noticing that 
$$\int_{ \Gamma_\vps} \bo S_1^\lambda \bo n \varphi_i \to  0, \mbox{ for } \vps \ra 0,$$
which is a consequence of the decomposition $u=u_{\reg}+u_{\sing}$. We know that $u_{\reg} \in H^{2+\delta} (\Om)$ and 
$H^{2+\delta} (\Om)$ is embedded in $W^{1, \infty} (\Om)$, so that the gradient of $u_{\reg} $ is bounded. 

At the same time, $|\nabla u_{\sing}|\le C r^{\frac{\pi}{\om_i} -1}$ for some constant $C$ independent on $\vps$. Both these observations lead to
$$\int_{ \Gamma_\vps} |\nabla u_{\reg}|^2  + |\nabla u_{\sing}|^2 \ra 0, \mbox{ for } \vps \ra 0.$$

To conclude notice that
\begin{align*}
\Big (\int_\Omega \bo S_1^\lambda \nabla \varphi_i\Big )_{i=0,...,n-1}&= \lim_{\vps \ra 0} \Big (\int_{\Omega_\vps} \bo S_1^\lambda \nabla \varphi_i\Big )_{i=0,...,n-1}\\
&= - \lim_{\vps \ra 0}  \int_{\partial \Om_\vps\sm \Gamma_\vps} |\nabla u|^2 \vphi_i n=  -\int_{\partial \Om} |\nabla u|^2 \vphi_i n.
\end{align*}
\end{proof}

\begin{rem}{\rm
	It is possible to note that the integrals which come into play in the boundary expression of the gradient only need to be computed on two adjacent sides to vertex $\bo a_i$, which gives
	\[
	\nabla \lambda (\bo x) = \begin{pmatrix}
	-\int_{\bo a_{i}\bo a_{i-1}}|\nabla u|^2\varphi_i \bo n_x-\int_{\bo a_{i}\bo a_{i+1}}|\nabla u|^2\varphi_i \bo n_x\\
	-\int_{\bo a_{i}\bo a_{i-1}}|\nabla u|^2\varphi_i \bo n_y-\int_{\bo a_{i}\bo a_{i+1}}|\nabla u|^2\varphi_i \bo n_y
	\end{pmatrix}_{i=0,...,n-1}.
	\]
	\label{rem:grad-explicit}
	}
\end{rem}

In the following we make the convention that the Jacobian matrix of a vector function contains gradients of the components on every line. 

\noindent {\bf  Hessian matrix of the eigenvalue.}   Following the notation of \cite{laurain2ndDeriv}, we introduce the functions $\bo U_i \in H_0^1(\Omega,\Bbb{R}^2),\ i=0,...,n-1$  such that $\dot u(\theta)  = \sum_{i=0}^{n-1} \theta_i \cdot \bo U_i$. Using \eqref{eq:material-fv-eig} we get the set of two PDEs: $\bo U_i \in H_0^1(\Omega,\Bbb{R}^2)$, 
\begin{align}\int_\Omega (D \bo U_i \nabla v-\lambda(\Omega) \bo U_iv) \, dx& = \int_\Omega \left[-(\nabla \varphi_i\otimes \nabla u)\nabla v+2(\nabla u \odot \nabla v)\nabla \varphi_i\right] \notag \, dx\\
& +\int_\Omega \bo S_1^\lambda \nabla \varphi_i  \int_\Omega uv\, dx + \lambda(\Omega) \int_\Omega uv \nabla \varphi_i\, dx,
\label{eq:material-eig-decomposed}
\end{align}
for every $v\in H_0^1(\Omega)$. The normalization condition \eqref{eq:deriv-norm-eig} gives
\begin{equation}
\int_\Omega ( 2u \bo U_i + u^2 \nabla \varphi_i)\, dx = 0,
\label{eq:normalization-decomposed}
\end{equation}
so that the system of equations \eqref{eq:material-eig-decomposed} - \eqref{eq:normalization-decomposed}   has a unique solution $\bo U_i$.

\begin{thm} The Hessian matrix $\bo N^\lambda \in \Bbb{R}^{2n\times 2n}$ of a simple Dirichlet-Laplace eigenvalue \eqref{eq:dir-eigenvalue} with respect to the coordinates of the $n$-gon is given by the following $n\times n$ block matrix
	\[ \bo N^\lambda = (\bo N_{ij}^\lambda)_{0\leq i,j\leq n-1}\]
	where the $2\times 2$ blocks are given by 
	\begin{align}
	\bo N_{ij}^\lambda & = \int_\Omega (-2D \bo U_i D \bo U_j^T+ 2\lambda(\Omega) \bo  U_i \bo U_j^T  + \nabla \varphi_i \otimes \bo S_1^\lambda \nabla \varphi_j + \bo S_1^\lambda \nabla \varphi_i \otimes \nabla \varphi_j)\, dx  \notag \\
	& +\int_\Omega \left(-|\nabla u|^2+\lambda(\Omega) u^2\right) (2 \nabla \varphi_i \odot \nabla \varphi_j) \, dx \notag \\
	& +2\int_\Omega \left[(\nabla \varphi_i\cdot \nabla u)(\nabla \varphi_j \otimes \nabla u)+(\nabla \varphi_j \cdot \nabla u)(\nabla u \otimes \nabla \varphi_i)+(\nabla \varphi_i \cdot \nabla \varphi_j)(\nabla u \otimes \nabla u)\right] \, dx\notag \\
	& - \int_\Omega u^2\left[\nabla \varphi_i \otimes \left(\int_\Omega \bo S_1^\lambda \nabla \varphi_j \, dx \right)+
	 \left(\int_\Omega \bo S_1^\lambda \nabla \varphi_i\, dx\right)\otimes \nabla \varphi_j \right]\, dx
	\label{eq:Hessian-lambda-block}
	\end{align}
	where $\bo U_i \in H^1(\Omega,\Bbb{R}^2),\ i=0,...,n-1$ are solutions of \eqref{eq:material-eig-decomposed}-\eqref{eq:normalization-decomposed}.
	\label{thm:hessian-eig}
\end{thm}

\noindent {\it Proof of Theorem \ref{thm:hessian-eig}:} The proof of this result, is computational in nature and is inspired by \cite[Proposition 14]{laurain2ndDeriv}. To obtain the Hessian matrix we use the formula for $\mathcal K^\lambda$ given in Theorem \ref{thm:sh-deriv-eig} for the Fr\'echet second shape derivative. There are several terms, already computed in \cite[Appendix A]{laurain2ndDeriv}, which also appear in the formula for the eigenvalue. We only present in detail the terms which are different. We point out that in order to obtain directly the Hessian matrix, the $2\times 2$ blocks should be multiplied by the variables $\xi_j$ below, which gives transposed $2\times 2$ blocks compared to \cite{laurain2ndDeriv}.

The first term is straightforward
\[-2 \int_\Omega (\nabla \dot u(\theta) \cdot \nabla \dot u(\xi) -\lambda(\Omega) \dot u(\theta) \dot u(\xi))\, dx = \sum_{i,j=0}^{n-1} \theta_i \cdot \left(\int_\Omega -2D\bo U_i D\bo U_j^T + \lambda(\Omega)\bo U_i \bo U_j^T\, dx\right)\xi_j. \]

The second term is treated in \cite{laurain2ndDeriv} (term $L_3$, pag. 38):
\[ \int_{\Omega} \bo S_1^\lambda :(D\theta \di \xi+D\xi \di \theta)\, dx = \sum_{i,j=0}^{n-1} \theta_i\cdot \left(\int_\Omega \left(\nabla \varphi_i \otimes \bo S_1^\lambda \nabla \varphi_j + \bo S_1^\lambda \nabla \varphi_i \otimes \nabla \varphi_j\right) \, dx\right)\xi_j. \]

The third term is similar to the term $L_4$ treated in \cite{laurain2ndDeriv} (pag. 39):

\begin{multline*}
\int_\Omega (-|\nabla u|^2+\lambda u^2)(\di \theta \di \xi+D\theta^T:D\xi)\, dx=\\ 
\sum_{i,j=0}^{n-1}  \theta_i \cdot \left( \int_\Omega \left(-|\nabla u|^2+\lambda(\Omega) u^2\right) (2 \nabla \varphi_i \odot \nabla \varphi_j)\, dx  \right) \xi_j\end{multline*}

The fourth term treated in \cite{laurain2ndDeriv} ($L_5$ pag. 39):
\begin{multline*}
2\int_\Omega (D\theta D\xi +D\xi D\theta + D\xi D\theta^T)\nabla u \cdot \nabla u\, dx = \sum_{i,j=0}^{n-1}  \theta_i \cdot \Bigg(2\int_\Omega\big[ (\nabla \varphi_j\cdot \nabla u)(\nabla u \otimes \nabla \varphi_i )\\ +(\nabla \varphi_i \cdot \nabla u)(\nabla \varphi_j\otimes \nabla u) +(\nabla \varphi_i \cdot \nabla \varphi_j)(\nabla u \otimes \nabla u)\big] \, dx\Bigg)\xi_j
\end{multline*}

The fifth term is new and will be computed below. Note that under the conventions $\theta = \sum_{i=0}^{n-1} \theta_i \varphi_i$ and $\xi = \sum_{i=0}^{n-1} \xi_i \varphi_i$ (see \eqref{def:phi}-\eqref{def:phi2} for the definition of $\varphi_i$) we have:
\begin{itemize}[topsep=0pt,noitemsep]
	\item $\di \theta = \sum_{i=0}^{n-1} \theta_i \cdot \nabla \varphi_i$
	\item for a $2\times 2$ matrix $A$, $A : D\theta = \sum_{i=0}^{n-1} \theta_i \cdot A\nabla \varphi_i$.
\end{itemize}
Using these relations we have
\begin{align*} &\int_\Omega [\lambda'(\Omega)(\theta) \di \xi +\lambda'(\Omega)(\xi)\di \theta] u^2\, dx\\
 = & \sum_{i,j=0}^{n-1} \int_\Omega u^2 \left[\left( \int_\Omega \theta_i \cdot \bo S_1^\lambda \nabla \varphi_i\, dx \right)(\xi_j \cdot \nabla \varphi_j)+\left( \int_\Omega \xi_j \cdot \bo S_1^\lambda \nabla \varphi_j\, dx \right)(\theta_i \cdot \nabla \varphi_i)\right]\, dx\\
 = & \sum_{i,j=0}^{n-1} \theta_i\cdot \left( \int_\Omega u^2\left[\nabla \varphi_i \otimes \left(\int_\Omega \bo S_1^\lambda \nabla \varphi_j\, dx \right)
 + \left(\int_\Omega \bo S_1^\lambda \nabla \varphi_i\, dx\right)\otimes \nabla \varphi_j \right]\, dx\right) \xi_j
\end{align*}
Regrouping all the above results finishes the proof of the theorem. \hfill $\square$
\begin{rem}{\rm 
	It is worth to notice that the matrix $\bo N^\lambda$ obtained in Theorem \ref{thm:hessian-eig} and the corresponding matrix obtained by Laurain in \cite[ Proposition 14]{laurain2ndDeriv} have similar structures (see Remark \ref{bobu40}). Moreover, the results resemble the structure of the tensor  $\bo S_1^\lambda$   corresponding to the first shape derivative in distributed form. The matrix $\bo N^\lambda$ has an additional term coming from the fact that the eigenvalue $\lambda(\Omega)$ is already present in $\bo S_1^\lambda$, and its derivative appears when computing the second shape derivative. 
	}
\end{rem}

\begin{rem}
	{\rm
	It can be noted that the Hessian matrix found in \eqref{eq:Hessian-lambda-block} does not depend on the normalization condition \eqref{eq:normalization-decomposed}. It is more convenient in the following to suppose that the functions $\bo U_i$ are normalized with the following condition 
	\begin{equation}
	\int_\Omega u \bo U_i\, dx = 0
	\label{eq:normalization-alternative}
	\end{equation}	
	where $u$ is the eigenfunction associated to the simple eigenvalue $\lambda(\Omega)$ of the Dirichlet-Laplacian.
	} 
\end{rem}

{\bf General properties of the Hessian matrix.} The formulas for the gradient and the Hessian matrix obtained previously do not depend on the choice of the perturbation given in \eqref{def:phi2}. As illustrated in Figure \ref{fig:simple-triangulations} multiple choices for the triangulations defining the functions $\varphi_i$ are possible. In particular:
\begin{itemize}[noitemsep,topsep=0pt]
	\item when the triangulation contains no inner vertices then $\sum_{i=1}^n \varphi_i=1$, which implies that $\sum_{i=1}^n \nabla \varphi_i = 0$. 
	\item for the regular polygon, considering a triangulation with an additional vertex at the center of the polygon provides additional symmetry properties. 
\end{itemize}
In the following we will switch between the two choices above in order to obtain further properties of the gradient and the Hessian matrix. In the following, define the two vectors $\bo t_x = (1,0,1,0,...,1,0)$ and $\bo t_y = (0,1,0,1,...,0,1) \in \Bbb{R}^{2n}$. 

\begin{prop}
	1. The sum of the components on of the gradient $\nabla \lambda(\bo x)$ on odd and even positions, respectively is zero. Equivalently we have $\nabla \lambda(\bo x) \cdot \bo t_x = \nabla \lambda(\bo x) \cdot \bo t_y = 0$.
	
	2. The vectors $\bo t_x, \bo t_y$ are eigenvectors of the matrix $\bo N^\lambda$ defined in \eqref{eq:Hessian-lambda-block}.
	\label{prop:translation-eigenvectors}
\end{prop}

\begin{proof} Let us note that by choosing $\varphi_i$ on a triangulation with no interior vertices we have $\sum_{i=1}^n \nabla \varphi_i = 0$. This already gives an answer to the first point above since 
\[ \sum_{i=0}^{n-1} \int_\Omega \bo S_1^\lambda \nabla \varphi_i\, dx = 0.\]

For the second point, let us note that with the same choice of the functions $\varphi_i$ the solutions $\bo U_i$ of \eqref{eq:material-eig-decomposed} with the normalization condition \eqref{eq:normalization-alternative} verify $\sum_{i=1}^n \bo U_i = 0$ since the sum of the right hand sides in \eqref{eq:material-eig-decomposed} is equal to zero. It is now straightforward to see that $\bo N^\lambda \bo t_x = \bo N^\lambda \bo t_y = 0$ which implies that the vectors $\bo t_x,\bo t_y$ are eigenvectors of $\bo N^\lambda$ corresponding to the zero eigenvalue.
\end{proof} 

Formula \eqref{eq:Hessian-lambda-block} respects the structure of the second shape derivative. It is possible to simplify the formula using the definition of $\bo S_1^\lambda$ and the property $(a\otimes b)(c\otimes d) = (b\cdot c)(a\otimes d)$. Regrouping terms we obtain
\begin{align*}
\bo N_{ij}^\lambda & = \int_\Omega (-2D \bo U_i D \bo U_j^T+ 2\lambda(\Omega) \bo  U_i \bo U_j^T)\, dx  \notag\\
& +\int_\Omega \left(|\nabla u|^2-\lambda(\Omega) u^2\right) ( \nabla \varphi_i \otimes \nabla \varphi_j-\nabla \varphi_j \otimes \nabla \varphi_i) \, dx \notag \\
&-2\int_\Omega \left(\nabla u \otimes \nabla u\right) ( \nabla \varphi_i \otimes \nabla \varphi_j-\nabla \varphi_j \otimes \nabla \varphi_i) \, dx \notag\\
&-2\int_\Omega  ( \nabla \varphi_i \otimes \nabla \varphi_j-\nabla \varphi_j \otimes \nabla \varphi_i)\left(\nabla u \otimes \nabla u\right) \, dx \notag\\
& +2\int_\Omega (\nabla \varphi_i \cdot \nabla \varphi_j)(\nabla u \otimes \nabla u) \, dx \notag \\
&  - \int_\Omega u^2\left[\nabla \varphi_i \otimes \left(\int_\Omega \bo S_1^\lambda \nabla \varphi_j \, dx\right)+
\left(\int_\Omega \bo S_1^\lambda \nabla \varphi_i \, dx\right)\otimes \nabla \varphi_j \right]\, dx
\end{align*}
It is immediate to see that 
\begin{multline*} \left(\nabla u \otimes \nabla u\right) ( \nabla \varphi_i \otimes \nabla \varphi_j-\nabla \varphi_j \otimes \nabla \varphi_i)+ ( \nabla \varphi_i \otimes \nabla \varphi_j-\nabla \varphi_j \otimes \nabla \varphi_i)\left(\nabla u \otimes \nabla u\right)=\\
|\nabla u|^2 ( \nabla \varphi_i \otimes \nabla \varphi_j-\nabla \varphi_j \otimes \nabla \varphi_i).
\end{multline*}
Therefore, the expression of the Hessian matrix simplifies to 
\begin{align}
\bo N_{ij}^\lambda & = \int_\Omega (-2D \bo U_i D \bo U_j^T+ 2\lambda(\Omega) \bo  U_i \bo U_j^T)\, dx  \notag\\
& +\int_\Omega \left(-|\nabla u|^2-\lambda(\Omega) u^2\right) ( \nabla \varphi_i \otimes \nabla \varphi_j-\nabla \varphi_j \otimes \nabla \varphi_i) \, dx  +2\int_\Omega (\nabla \varphi_i \cdot \nabla \varphi_j)(\nabla u \otimes \nabla u) \, dx \notag \\
&  - \int_\Omega u^2\left[\nabla \varphi_i \otimes \left(\int_\Omega \bo S_1^\lambda \nabla \varphi_j\, dx \right)+
\left(\int_\Omega \bo S_1^\lambda \nabla \varphi_i\, dx \right)\otimes \nabla \varphi_j \right]\, dx.
\label{eq:Hessian-lambda-exp}
\end{align}

From this point on, in the rest of the paper, we concentrate on the case of the {\bf first eigenvalue of the regular polygon} and we further simplify the expression of the Hessian.  By uniqueness arguments the first eigenfunction $u$ of the Dirichlet Laplace operator on the regular polygon has the same symmetries as the regular polygon.

In the following suppose that $\varphi_i$, $0 \leq i \leq n-1$ are associated to the particular triangulation $\mathcal T=(T_k)_{k=0}^{n-1}$ of the regular polygon made of congruent triangles with one vertex at the center (see Figure \ref{fig:simple-triangulations}). Thus, the triangulation $\mathcal T$ also respects the symmetry of the regular polygon. The symmetry of the first eigenfunction implies that $\int_{T_k}(|\nabla u_1|^2-\lambda_1(\Omega) u_1^2)\, dx=0$. Using this relation the gradient of $\lambda_1(\Omega)$ on the regular polygon becomes
\[ \int_\Omega \bo S_1^\lambda \nabla \varphi_i = \int_\Omega (|\nabla u_1|^2-\lambda_1(\Omega) u_1^2)\nabla \varphi_i - 2 (\nabla u_1 \otimes \nabla u_1) \nabla \varphi_i = -2  \int_\Omega (\nabla u_1 \otimes \nabla u_1) \nabla \varphi_i.\]

 Using the fact that $\nabla \varphi_i \otimes \nabla \varphi_j-\nabla \varphi_j \otimes \nabla \varphi_i$ is piece-wise constant on every triangle $T_k, k=0,...,n-1$, we find that
\[
\int_\Omega \left(-|\nabla u_1|^2-\lambda(\Omega) u_1^2\right) ( \nabla \varphi_i \otimes \nabla \varphi_j-\nabla \varphi_j \otimes \nabla \varphi_i) \, dx = \frac{\lambda_1(\Omega)}{|\Omega|} \bo B_{ij},
\]
where $\bo B_{ij}$ are the blocks of the Hessian of the area given in \eqref{eq:hess_area_phi}.

Recall that $\nabla \varphi_i$ is piecewise constant on the triangles $T_k$ and by symmetry $\int_{T_k} u_1^2\, dx = 1/n$ for $k=0,...,n-1$. Therefore $\int_{T_k} u_1^2 \nabla \varphi_i\, dx = \frac{1}{|\Omega|} \int_{T_k}\nabla \varphi_i\, dx=\frac{1}{|\Omega|} \nabla \mathcal A(\bo x)$, where $\mathcal A(\bo x)$ is the area of the polygon having vertices at coordinates given by $\bo x$, as recalled earlier. Therefore, the last term in $\bo N_{ij}^\lambda$ has the form 
\[ \int_\Omega u_1^2\left[\nabla \varphi_i \otimes \left(\int_\Omega \bo S_1^\lambda \nabla \varphi_j\, dx \right)+
\left(\int_\Omega \bo S_1^\lambda \nabla \varphi_i\, dx \right)\otimes \nabla \varphi_j \right]\, dx = \frac{2}{|\Omega|} (\nabla \mathcal A(\bo x) \odot \nabla \lambda_1(\bo x)) \]

Consider now the Hessian of the product $\lambda_1(\bo x)\mathcal A(\bo x)$ and note that we have
\[ \Hess (\lambda_1(\bo x)\mathcal A(\bo x)) = |\Omega|\Hess \lambda_1(\bo x)+\nabla \lambda_1( \bo x) \otimes \nabla \mathcal A(\bo x)+
\nabla \mathcal A(\bo x)\otimes \nabla \lambda_1( \bo x) + \lambda_1 (\bo x) \Hess \mathcal A(\bo x).  \]
In this formula the last term of the Hessian of $\lambda_1(\bo x)$ simplifies the tensorial products between the gradient of the area and the gradient of the eigenvalue.

Following the previous computations we arrive at the following significant simplification for the Hessian of the product of the area and the eigenvalue.

\begin{prop}
	In the case where $\Omega$ is a regular $n$-gon and the triangulation $\mathcal T$ defining $\varphi_i$ is symmetric the Hessian matrix of $\lambda_1(\Omega)|\Omega|= \mathcal A(\bo x) \lambda_1(\bo x)$ in terms of the coordinates of the polygon has the $2\times 2$ blocks $\bo M_{ij}^\lambda$, $0\leq i,j \leq n-1$ given by
	\begin{align}
	\bo M_{ij}^\lambda  & = |\Omega|\int_\Omega (-2D\bo U_i  D\bo U_j^T+ 2\lambda_1(\Omega)\bo U_i\bo U_j^T)   \notag \\
	&  -\lambda_1(\Omega)\int_\Omega [ \nabla \varphi_i \otimes \nabla \varphi_j - \nabla \varphi_j \otimes \nabla \varphi_i] \notag \\
	& +2|\Omega|\int_\Omega (\nabla \varphi_i \cdot \nabla \varphi_j)(\nabla u_1 \otimes \nabla u_1) .
	\label{eq:simplified-hessian}
	\end{align}	
\end{prop}

 The simplified formula \eqref{eq:simplified-hessian} for the Hessian of the product of the area and the first eigenvalue has three terms:
	\begin{itemize}
		\item The first one is related to the decomposition $\bo U_i$ of the material derivatives given in \eqref{eq:material-eig-decomposed}. Furthermore, the terms are related to the bilinear form from the variational formulations of $\bo U_i$, which will be essential in improving the estimates in the numerical simulations. This part of the Hessian is negative definite.
		\item The second term is related to the Hessian of the area given in \eqref{eq:hess_area_phi}. The associated blocks are non-zero only when $|i-j|=1$ (modulo $n$). This part has both positive and negative eigenvalues.
		\item The third term involves only the first eigenfunction $u_1$ and the functions $\varphi_i$ defined in \eqref{def:phi}. The associated blocks are non-zero only when $|i-j|\leq 1$. This part of the Hessian is positive definite.
	\end{itemize}

Although the expression of the Hessian given in \eqref{eq:simplified-hessian} is explicit, its positive definiteness is not obvious. The analysis of the eigenvalues of this matrix is continued in Section \ref{sec:eig-hessian}.

\section{Geometric stability of the shape Hessian matrix}\label{bobus3}
In this section we shall perform both a qualitative and quantitative analysis of the behavior of the coefficients of the Hessian matrix for local perturbations of the vertices of the regular  polygon  $\mathbb P_n$ inscribed in the unit circle with one vertex at $(1,0)$. Some of the results would extend naturally either to perturbations of general convex polygons or even to more general sets. Nevertheless, we focus on the perturbation of the regular $n$-gon and we shall not search generality. The two main technical aspects of this section are described below.
\begin{itemize}
\item {\bf Continuity of the  Hessian matrix coefficients for the geometric perturbation.} 
We prove the continuity of the shape Hessian matrix for a perturbation of the regular polygon. This question is itself non trivial because of the weak regularity of the right hand sides in the equations satisfied by the solutions $\bo U_i$ of \eqref{eq:material-eig-decomposed}. Stability results for the eigenfunctions in $H^2$ are required, whereas the classically known stability based on $\gamma$-convergence holds in $H^1$. The continuity of the coefficients will readily give {\it the local minimality of the regular polygon} provided the positive definiteness of the Hessian matrix is known on the regular polygon {\it only}.

\item {\bf Estimate of the modulus of continuity of the coefficients for the geometric perturbation.} This information is crucial to formally reduce the proof of the conjecture to a finite number of numerical computations. We compute the modulus of continuity of the coefficients, i.e. we find estimates of the variation of all coefficients of the Hessian matrix in terms of some power of Hausdorff distance between the perturbed polygon and the regular polygon. In other words, for every $\delta >0$ we identify a value $\vps >0$ such that all the coefficients of the Hessian matrix computed on polygons with $n$ sides in an $\vps$-neigbourhood of $\mathbb P_n$  stay in a $\delta$-neighborhood of the coefficients of the Hessian matrix associated to $\mathbb P_n$. 
\end{itemize} 
We split this section in three subsections, going from basic estimates for the variations of the eigenvalues and eigenfunctions to the estimates of the variation of the matrix coefficients. This last point is more delicate as it  involves solutions of \eqref{eq:material-eig-decomposed}-\eqref{eq:normalization-alternative}  with variable, singular, right hand sides that are  not in $L^2$. 

Throughout this section, we denote by $C, \theta$  two positive constants which may change from line to line.
The tracking of those constants is possible but, since we will not perform here numerical computations of an effective neighborhood of minimality, this is not immediately useful. Consequently, in order to avoid heavy calculations we choose to prove only the existence of those constants. In particular, we are not aimed here to optimize the constants, which in case of certified numerical computations of the neighborhood would be a priority. 

\subsection{Basic quantitative estimates along the perturbation} 
Let $\Om \sq \R^2$ be a bounded, simply connected, open Lipschitz set and  $f \in H^{-1} (\R^2)$.  We consider the problem
\begin{equation}\label{bobu02}
\left\{ \begin{array}{rcll}
   -\Delta v & =& f & \text{ in }\Omega,\\
   v &= &0 & \text{ on }\partial \Omega. 
 \end{array} \right.
 \end{equation}
 In the particular case in which $f=1$, we denote $w_{\Om}$ the the solution of \eqref{bobu02}, and call it torsion function. The torsion function is the unique minimizer of the torsion energy,
$$E(\Om):= \min_{u \in H^1_0(\Om)} \frac 12 \int_\Om |\nabla u(x) |^2 dx -\int_\Om u(x) dx.$$

Let now $\Om_\alpha$, $\alpha\in \{a,b\}$ be two such domains and denote by $v_\alpha$ the solution of \eqref{bobu02} on $\Om_\alpha $ for the right hand side $f_\alpha$ and by $u_{1,\alpha}$  the $L^2$-normalized, non-negative  eigenfunctions on $\Om_\alpha$ corresponding to the first eigenvalues $\lb_{1,\alpha}$, respectively. We denote by $d_H$ the Hausdorff distance. 

In a first step, we seek estimates of the form
\begin{equation}\label{bobu04}
\|v_a-v_b\|_{H^1(\R^2)}\le Cd_H^{\thetathree}(\partial \Om_a, \partial \Om_b)(\|f_a\|_{L^2(\R^2)}+\|f_b\|_{L^2(\R^2)}) + C\|f_a-f_b\|_{L^2(\R^2)},
 \end{equation}
\begin{equation}\label{bobu05}
|\lb_{1,a}-\lb_{1,b}| \le  Cd_H^{\thetathree}(\partial \Om_a, \partial \Om_b),
 \end{equation}
\begin{equation}\label{bobu06}
  \|u_{1,a}-u_{1,b}\|_{H^1(\R^2)}\le Cd_H^{\thetathree}(\partial \Om_a, \partial \Om_b),
   \end{equation}
for some computable $C, \thetathree>0$.

Above, all functions $u_{1,\alpha}, v_\alpha$ are assumed to be extended by $0$ on the complement of their definition domain, this extension being suitable for $H^1$-estimates. By abuse of notation, the extensions by $0$ are still denoted with the same symbols. 
The literature is quite rich for such type of  $H^1$-estimates, like \eqref{bobu04} and  \eqref{bobu06}. For instance, Savar\'e and Schimperna \cite{SaSc02} give estimates for 
solutions of \eqref{bobu02} in the class of sets satisfying a uniform cone condition while Burenkov and Lamberti \cite{BBL10}, Feleqi \cite{FE16} discuss the eigenfunctions. Concerning \eqref{bobu05}, we refer to \cite{Pa97} (see as well Section \ref{bobu200}) for sharp estimates with power $\vartheta =\frac 12$ and controlled constant.

Let us point out a relevant fact, which becomes important as soon as we search to identify all the constants in  \eqref{bobu04}- \eqref{bobu06}. The results referred above occur in the class of domains satisfying a uniform cone condition, while our setting is much more regular: we locally perturb the regular $n$-gon,  always obtaining a convex $n$-gon. This regular behavior will be exploited in the next subsection to get estimates in higher order norm even in the case of singular right hand sides and it dramatically simplifies the proofs of the $H^1$-estimates.

Below we shall only recall some results without proofs. The interested reader could easily recover the estimates in our regular setting in a more direct way. 
Assume that $\Om_a, \Om_b\sq \R^2$ satisfy a uniform $(\rho, \varepsilon)$-cone condition (see \cite[Definition 2.6]{SaSc02}). 
\begin{prop}[Savar\'e-Schimperna \cite{SaSc02}]\label{bobu7} If $f_a=f_b:=f$, there exists a constant depending only on the diameters such that
\begin{equation}\label{bobu03}
\|\nabla v_a- \nabla v_b\|_{L^2} \le C\|f\|_{L^2} ^\frac 12 \|f\|_{H^{-1} } ^\frac 12 \Big ( \frac{d_H(\Om_a, \Om_b)}{\rho \sin \varepsilon}\Big )^\frac 12.
\end{equation}
\begin{equation}\label{bobu08}
\|v_a- v_b\|_{L^2} \le C\|f\|_{L^2} ^\frac 12 \|f\|_{H^{-1} } ^\frac 12 \frac{d_H(\Om_a, \Om_b)}{\rho \sin \varepsilon}.
\end{equation}
\begin{equation}\label{bobu30}
\|v_a- v_b\|_{L^2} \le C \|f\|_{H^{-1} }  \Big ( \frac{d_H(\Om_a, \Om_b)}{\rho \sin \varepsilon}\Big )^\frac 12.
\end{equation}

\end{prop}
Note that the first two inequalites require $f\in L^2(\R^2)$. The result recalled in Proposition \ref{bobu7} together with  the Poincar\'e inequality readily gives inequality \eqref{bobu04}. Note as well that the Poincar\'e constants on the two domains equal the first Dirichlet eigenvalues. 

For a small perturbation of the regular $n$-gon, the values of  $\rho$ and $\vartheta $ can be computed explicitly. 
However, in this last case a more direct proof of the inequalities can be obtained as a consequence of the uniform bound of the $H^2$ norms of the solutions with an explicit value (maybe not optimal) of the constant $C$. 

Concerning the estimates \eqref{bobu05} and \eqref{bobu06}, we refer to the papers of Feleqi \cite{FE16} and Burenkov and Lamberti \cite{BuLa12}. Those esimates being less explicit, we give below a slef contained argument which takes advantage of the convexity of the sets. 

For now, assume that $\Om_a$ and $\Om_b$ are convex, in which case the level sets of the torsion function and of the first eigenfunctions are convex. Moreover,  $v_\alpha$ and the eigenfunction $u_{1,\alpha}$ belong to $H^2(\Om_\alpha)$ as we shall recall in the next subsection. We recall a  first regularity result in the class of convex sets, due to Grisvard \cite[Theorem 3.1.2.1]{grisvard}.

\begin{prop}[Grisvard]
Assume $\Om_\alpha$ is a bounded convex open set and $f_\alpha \in L^2(\Om_\alpha)$.
Let $v_\alpha$ solve \eqref{bobu02}. Then
$$\|D^2 v_\alpha\|_{L^2(\Om_\alpha)} \le \|f_\alpha\|_{L^2(\Om_\alpha)}.$$
\end{prop}
 For a $n$-gon which is a small perturbation of the regular $n$-gon $\mathbb P_n$, this inequality gives uniform bounds for the $H^2$-norms of the normalized eigenfunctions and of some $H^2$ extensions in $\R^2$. The bounds in $L^\infty$ are standard and the convexity of the polygon together with the barrier method provides $L^\infty$ estimates for the gradients.

\begin{lemma}\label{bb05}
 Assume that $f_a=f_b= f \in L^\infty (\R^2), f \ge 0$. Then
\begin{equation}\label{bb05.1}
\int_{\R^2} |\nabla v_{a}-\nabla v_b |^2 dx\le d_H(\partial \Om_{a}, \partial \Om_{b}) \|f\|_\infty ^2 \Big  (|\Om_a|\diam (\Om_a) +|\Om_b|\diam (\Om_b)\Big ).
\end{equation}
\end{lemma}
\begin{proof}
Let $\tilde \Om =\Om_a \cap \Om_b$. Then we have as well $d_H(\partial \tilde \Om, \partial \Om_\alpha) \le d_H(\partial \Om_{a}, \partial \Om_{b}) $ and $\tilde \Om \sq \Om_\alpha$ for $\alpha\in \{a,b\}$. Denoting   $\tilde v$ the solution of \eqref{bobu02} in $\tilde \Om$, we have
$$\int_{\Om_\alpha} |\nabla \tilde v -\nabla v_\alpha |^2 dx= \int_{\Om_\alpha} f(\tilde v-v_\alpha) dx\le\|f\|_\infty |\Om_\alpha|  \max_{x \in\Om_\alpha} \big (v_\alpha(x)-\tilde v (x)\big ).$$
We notice that the function $v_\alpha -\tilde v$ is harmonic on $\tilde \Om$, so its maximum on $\tilde \Om$ is attained on $\partial \tilde \Om$, where $\tilde v$ vanishes. Since $\partial \tilde \Om$ lies in a  neighborhood of $\partial \Om_\alpha$, denoting $\vps = d_H(\partial \Om_{a}, \partial \Om_{b})$    we have
$$\max_{x \in\Om_\alpha } \big (v_\alpha(x)-\tilde v (x) \big ) \le \max_{x \in \partial \Om_\alpha \oplus B_\vps}v_\alpha(x).$$
However, for every $x \in \partial \Om_\alpha \oplus B_\vps$ we have $v_\alpha(x)\le \|f\|_\infty w_\alpha(x)\le  \vps \|f\|_\infty ^2 \|\nabla w_\alpha \|_\infty$, where $w_\alpha$ is the torsion function.

In order to  bound $\|\nabla w_\alpha \|_\infty$ we take advantage that the level sets of $w_\alpha$ are convex and so we have a barrier given by the width. Indeed, in every point $x$ of the the level set, we can find an infinite strip containing the level set and having one boundary line passing through $x$.  Using the classical barrier method gives $ |\nabla w_{\alpha}(x)| \le W_x/2$, where $W_x$ is the width. This implies
$$\int_{\Om_\alpha} |\nabla \tilde v-\nabla v_\alpha |^2 dx \le  \vps \|f\|_\infty ^2 |\Om_\alpha| \frac{\diam (\Om_\alpha)}{ 2}.$$
Adding the estimates for $v_a$ and $v_b$ leads to the conclusion.
\end{proof}

\medskip
\noindent {\bf Perturbations of the regular polygon.} 
 For $n \ge 5$ we  denote $\mathbb P_n=\bo a_0^* \bo a_1^* \dots \bo a_{n-1}^*$ the regular polygon with $n$ sides inscribed in the unit circle with $a_0^*=(1,0)$. We denote $R_n, r_n$ the radii of the circumscribed, inscribed circles for $\mathbb P_n$ and $l_n$ the length of an edge. Denote the area of $\Bbb P_n$ by $\mathcal A_n$. The angles are equal to $\frac{n-2}{n}\pi$. 
An easy computation leads to 
$$R_n=1 ,r_n = \cos \frac{\pi}{n}, l_n = 2 \sin \frac{\pi}{n}, \mathcal A_n = n\sin (2\pi/n)/2.$$
Let $P$ denote generically a perturbation of $\mathbb P_n$, i.e. polygon $\bo a_0 \bo a_1\dots \bo a_{n-1}$ with $n$ sides such that for every $i=0, \dots, n-1$ we have $|\bo a_i\bo a^*_i|\le \vps$. The critical value of $\vps$ where convexity is lost is $\vps= \sin^2 \frac {\pi}{n}$. For instance, if 
 $$|\bo a_i\bo a_i^*|\le  \frac 14  \sin ^2 \frac{\pi}{n}:= \vps_0,$$
 the angles of the perturbed polygon do not exceed
$$\omega_{0} = \frac{(n-2)\pi}{n}  +2 \arcsin \left(\frac 14 \sin \frac{\pi}{n}\right)<\pi.$$

We can represent both the boundaries of $\Bbb P_n$ and $P$  using the same $n$ charts given by the graphs of the boundaries $\partial \Bbb P_n, \partial P$ over the segments
$$[\bo x_i \bo y_{i}]\mbox{ where } \bo x_i = \frac 34 \bo a_i^*+\frac 14 \bo a_{i+1}^*, \bo  y_i= \frac 34 \bo a_{i+2}^*+\frac 14 \bo a_{i+1}^*.$$
In each chart, the function representing the boundary of the polygons is piecewise affine with two slopes not exceeding $\tan (\frac{\pi}{n}+ \arctan (\frac 14 \sin \frac{\pi}{n}))$. For $n \ge 5$ an upper bound for this quantity is $ 0.73 $.

 We denote $\lb_k, \lb_k^*$ the  $k$-th eigenvalues and $u_k$ and $u_k^*$ the corresponding normalized eigenfunctions on $P$, $\Bbb P_n$, respectively.
 \begin{prop}\label{bb03}
Under the previous hypotheses
\begin{equation}\label{bb04}
|\lb_1 -\lb_1^*|\le \int_{\R^2} |\nabla u_1 -\nabla u_1^* |^2 dx \le 2(E_1+E_3),
\end{equation}
where 
$$E_1 = \vps (\lb_1^*) ^2 \|u_1^*\|_\infty ^2 \big (2\pi+2\pi(1+\varepsilon)^3
\big )
,$$
$$E_2=  \frac{\lb_1}{\lb_2-\lb_1}\frac{(r_n+\vps)^4-r_n^4}{r_n^4} + \frac{2\lb_2}{\lb_2-\lb_1}\left (\frac{E_1}{\lb_1(\Bbb P_n\oplus B_\vps)}\right )^\frac12,$$

 $$E_3=\frac{2 \lb_1 E_2}{1+\alpha_1}+\lb_1^*\left(1- \left(\frac{r_n}{r_n+\vps}\right)^2\right)+\lb_1^* \left(\frac{E_1}{\lb_1(\Bbb P_n \oplus B_\vps)}\right)^\frac12$$
\end{prop}

\begin{proof}
The inclusions
$\frac{r_n}{r_n+\vps} P \sq \Bbb P_n \sq \frac{r_n}{r_n-\vps} P$, imply that $\big (\frac{r_N+\vps} {r_N}\big )^2\lb_1\ge \lb_1^* \ge \big (\frac{r_N-\vps} {r_N}\big )^2\lb_1$.

We introduce the problem
$$\psi \in H^1_0(P), -\Delta \psi = \lb_1^* u_1^* \mbox{ in } {\mathcal D'}(P).$$
Using Lemma \ref{bb05} and the Poincar\'e inequality we have
$$\int_{\Bbb{R}^2} |\nabla \psi -\nabla u_1^*|^2 dx \le E_1 \text{ and } \int_{\Bbb{R}^2} | \psi - u_1^*|^2 dx\le \frac{E_1}{\lb_1( \Bbb P_n\oplus B_\vps)}.$$ 
Using the orthonormal Hilbert basis of eigenfunctions in $H^1_0(P)$ we consider the decomposition
$\psi=\sum_{i=1}^{+\infty} \alpha_i u_i$ which gives
$$\int_{\Bbb{R}^2} |\nabla \psi |^2 dx \le \int_{\Bbb{R}^2} |\nabla u_1^* |^2 dx + 2 \int_{\Bbb{R}^2} \nabla \psi (\nabla \psi -\nabla u_1^*) dx\le \lb_1^*+ 2 \|\nabla \psi\|_2 \|\nabla \psi-\nabla u_1^*\|_2.$$
We have
 $$  \int_{\Bbb{R}^2} |\nabla \psi |^2 dx = \lb_1^* \int_{\Bbb{R}^2}  \psi u^*\le \lb_1^* \|\psi \|_2\le \frac{\lb_1^*}{(\lb_1 )^\frac 12} \|\nabla \psi \|_2,$$
  which leads to
$$ \sum_i \alpha_i^2 \lb_i = \int_{\Bbb{R}^2} |\nabla \psi |^2 dx \le \frac{(\lb_1 ^*)^2}{\lb_1 }.$$
Consequently,
$\alpha_1 ^2 \lb_1  +  \lb_2  \sum _{i=2}^{+\infty} \alpha_i^2\le  \frac{(\lb_1^*)^2}{\lb_1 }\le\big (\frac{r_n+\vps} {r_n}\big )^4\lb_1$
so
$$\alpha_1 ^2 \lb_1 +  \lb_2 \left(\int_{\Bbb{R}^2} \psi ^2dx -\alpha_1^2\right) \le \left(\frac{r_n+\vps} {r_n}\right)^4\lb_1.$$
On the other hand,
$$\int_{\Bbb{R}^2} \psi ^2dx \ge \int_{\Bbb{R}^2} (u_1^*)^2 dx- 2 \int_{\Bbb{R}^2} u_1^*(u_1^*-\psi) dx \ge 1-2\|u_1^*-\psi\|_2 \ge  1-2\left(\frac{E_1}{\lb_1(\mathbb P_n\oplus B_\vps)}\right)^\frac12 ,$$
which, after elementary computations leads to
$$1-\alpha_1^2 \le \frac{\lb_1}{\lb_2-\lb_1}\frac{(r_n+\vps)^4-r_n^4}{r_n^4} + \frac{2\lb_2}{\lb_2-\lb_1}\left(\frac{E_1}{\lb_1(\mathbb P_n\oplus B_\vps)}\right)^\frac12:= E_2.$$
Finally,
$$\int_{\Bbb{R}^2}|\nabla \psi -\nabla u_1 |^2 dx = \lb_1^*\int_{\Bbb{R}^2} u_1^* \psi dx -2\lb_1\int_{\Bbb{R}^2} u_1 \psi dx+ \lb_1$$
$$=\lb_1^*+ \lb_1^*\int_{\Bbb{R}^2} u_1^* (\psi -u_1^*)dx-2\lb_1 \alpha_1+ \lb_1$$
$$= 2 \lb_1(1-\alpha_1)+ \lb_1^*-\lb_1+ \lb_1^* \|\psi -u_1^*\|_2$$
$$\le \frac{2 \lb_1 E_2}{1+\alpha_1}+\lb_1^*\left(1- \left(\frac{r_n}{r_n+\vps }\right)^2\right)+\lb_1^* \left(\frac{E_1}{\lb_1(\mathbb P_n\oplus B_\vps)}\right)^\frac12:=E_3.$$
By summation, the inequality follows. 
\end{proof}
\begin{rem}{\rm
In order to complete the estimates we recall that in simply connected domains $\|u_1\|_\infty \le \lb_1^\frac 12$ (see  Grebenkov   \cite[Formula (6.22)]{GrNg13}).  We also recall from \cite{AsBe92} that $\frac{\lb_2 }{ \lb _1} \le j_{1,1}^2/j_{0,1}^2$, where  $j_{0,1}, j_{1,1}$ denote the first  positive zero of the Bessel functions $J_0, J_1$  and that $\lb_2-\lb_1\ge \frac{3\pi^2}{\diam^2 (P)}$ from \cite{AnCl11}. 
As well, by inclusion and homogeneity, $\lb_1(\Bbb{P}_n\oplus B_\vps) \ge \Big ( \frac{1}{1+\vps} \Big )^2 \lb_1^*$. 

We can also give a direct estimate  for $\|\psi -u_1 \|_2$. Indeed, 
$$\int_{\Bbb{R}^2} (\psi -u_1 )^2 dx = (1-\alpha_1)^2+ \sum_{i=2}^{+\infty} \alpha_i^2= (1-\alpha_1)^2+ \int_{\Bbb{R}^2} \psi ^2 dx -\alpha_1^2\le$$
$$\le (1-\alpha_1)^2+(1+ \|\psi -u_1^*\|_2)^2 -\alpha_1^2$$
$$\le 2(1-\alpha_1)+ 2   \|\psi -u_1^*\|_2 +  \|\psi -u_1^*\|_2^2$$
$$\le \frac{2}{1+\alpha_1} E_2+ 2\left(\frac{E_1}{\lb_1(\Bbb{P}_n\oplus B_\vps)}\right)^\frac12+\frac{E_1}{\lb_1(\Bbb{P}_n\oplus B_\vps)}:=E_4.$$

}
\end{rem}

\begin{prop}\label{bobu16}
There exists a constant $C>0$ such that for all $0<\vps<\vps _0$
$$\|\nabla u_1\|_\infty \le C \text{ and }\|u_1-u_1^*\|_\infty \le C \|u_1-u_1^*\|_{H^1} ^\frac 13.$$
\end{prop}
\begin{proof}
The first inequality  is a consequence of the barrier method. The diameter and the inner ball control the size of the eigenvalue and of the $L^\infty$ norm of the the eigenfunctions, themself being controlled by $\vps_0$. 

The second inequality is a consequence the Gagliardo-Nirenberg inequality (see for instance \cite{Po20}) 
$$\|u_1-u_1^*\|_\infty \le C \|\nabla u_1-\nabla u_1^*\|_{L^3} ^\frac 23 \|\ u_1-u_1^*\|_{L^3} ^\frac 13.$$
Then we use first inequality and the continuous embedding $H^1(B_2) \sq L^3(B_2)$. 
\end{proof}
\subsection{Uniform $H^{2+s}$ regularity of the eigenfunctions}

In this section we recall some finer estimates of the regularity of the solutions $v_\alpha$ of \eqref{bobu02} in polygons which are small perturbations of the regular polygon. However, we need more regularity than $H^2$ in order to quantify the variation of the shape Hessian coefficients. These finer regularity results take full advantage from the very specific convex, polygonal geometry of the domains, size of angles and number of local charts of the boundary. We refer the reader to \cite{Da88} for detailed analysis of the regularity in polygonal domains.

We recall the following regularity  result from \cite[Theorem 9.8]{BDLN92} (see also \cite{Da88}).  
\begin{lemma}\label{bobu12}
Let $P$ be a perturbation of the regular polygon $\Bbb{P}_n$ as above. Let $0<\gamma\le \frac {\pi}{\omega_{0}} $.  Then, for every $f \in H^{-1+ \gamma }  (P)$ the solution of \eqref{bobu02} in $P$ satisfies 
$$\|v\|_{H^{1+\gamma}(P)} \le C \|f\|_{H^{-1+\gamma} (P)}.$$
The  constant $C$ depends on $\gamma$ but it is independent on $f$ and $P$.  
\end{lemma}
Above, the independence on $P$ comes precisely from the very specific perturbation we consider, which keeps constant the charts and controls the angles. Let us denote $s_0 = \frac {\pi}{\omega_{0}} - 1>0$ and let $0\le s \le s_0$. 
\begin{cor}
Under the previous hypotheses and notations we have 
$$u_1\in H^{2+s} (P),  \|u_1\|_{H^{2+s} (P)} \le C,$$
with $C$ depends on $s$ but is independent on the perturbation. 
\end{cor}
\begin{proof}
This is a consequence of Lemma \ref{bobu12} and of the fact that the right hand sides  $\lb_1 u_1$ of the equations solved by the eigenfunctions have an $H^1$-norm equal to $\lb_1 (1+ \lb_1)$ which is uniformly bounded in the class of perturbations we consider. 
\end{proof}

One has to pay particular attention to the extension of $u_1$ on the complement of  $P$. As far as we are concerned with $L^p, H^1$ properties of the extension, performing an extension by $0$ on $\R^2\sm P$ is enough. Neverhtless, such an extension does not belong to $H^2, H^{2+s}$, so we can not compare the extensions of $u_1$ and $u_1^*$ in those norms. 

Two choices can be done in order to compare solutions on different polygons in $H^2$. Either we extend them in $H^2$ and compare their extensions, or we locally compare on compact sets included in both domains. Below, we choose to compare their extensions. The extensions we seek rely on the Stein universal extension operator (see \cite{St70} and \cite{LaVi19,HLZ12}). We recall the following from from \cite{ St70}.
\begin{prop} \label{bobu13}
Assuming $P$ is a perturbation of the regular polygon as above, there exists an extension operator 
$$E_P: L^1(P) \to L^1(\R^2)$$
such that 
$$\forall q \ge 0, \qquad \| E_P(u) \|_{H^q (\R^2)} \le C\| u\|_{H^q(P)},$$
where the constant $C$ above depends on $q$ but not on $P$. 

\end{prop}
\begin{rem}{\rm
We point out that the extension of Stein relies mainly on the construction of a smoothed distance function. The choice of this function is not unique. Stein proposed a construction based on partition of the complement of $\ov P$ on squares belonging to the union of latices $(2^{-k} \Z^2)_{k \in \Z}$. In the sequel we shall use this argument and the freedom to build the smoothed distance function in order to be able to compare the extension operators on $P$ and $\UUU \mathbb P_n$. Using a cut off function, we will assume that all extensions $E_P(u)$ vanish outside the ball $B_2$. 
}
\end{rem}
We recall now the Gagliardo-Nirenberg inequality from \cite{BrMi18}.
\begin{prop}\label{bobu14}
There exists $C>0, \thetathree \in (0,1)$ such that for every $u \in H^{2+s} (\R^2)$
$$ \|u\|_{H^{2} (\R^2)}\le C \| u\|_{L^2(\R^2)} ^\thetathree  \|u\|_{H^{2+s} (\R^2)}^{1-\thetathree}.$$
\end{prop}
The key use of this result is related to the possible extensions of an eigenfunction outisde $P$. Indeed, from Proposition \ref{bb03} we control the norm $\|u_1-u_1^*\|_{H^1(\R^2)} $. However, this is true for the extensions by $0$ of the eigenfunctions not for the extensions given by the Stein operator. Proposition \ref{bobu14} together with Proposition \ref{bobu13} imply that we can control the norm of the difference in $H^2$ for the Stein extensions provided we control the norm in $L^2$. This is a consequence of the following Lemma. 
\begin{lemma}\label{bobu15}
By $E_{\Bbb P_n}$ we denote a (suitably chosen) Stein extension operator associated to $\Bbb P_n$. There exists a constant $C$ such that for every perturbation $P$  as above there exists a Stein  extension operator $E_P$ satisfying 
\begin{equation}\label{bobu21}
\|E_P(u_1)-E_{\Bbb P_n}(u_1^*)\|_{L^\infty(\R^2)}\le C (\|u_1-u_1^*\|_{L^\infty(\R^2)}+ d(\partial P, \partial \Bbb P_n)).
\end{equation}
\end{lemma}
\begin{proof}

We  rely on the construction of the operator by Stein using the averaging method (see \cite[Theorem 5, page 181]{St70}). The difficulty is that we deal with  extension operators corresponding to different domains and applied to different functions. We want to prove that the extended functions are close in $L^\infty$ provided that the non extended functions are close in $L^\infty$. Since each one is extended with its own operator, we have to detail the construction of the operators in order to be able to perform the comparison. 

\medskip
 \noindent{\bf Step 1. Localization.} 
 Since the boundary of $P$ is described in the same charts as the boundary of the regular polygon, we use the explicit formula of the extension operator. We refer the reader to \cite[Theorem 5, page 181]{St70} (see also \cite{LaVi19,HLZ12}), where the explicit construction is given.  
 
 There exists a smooth partition of unity consisting on $n+2$ functions $(\psi_j)_{j=0, \dots, n+1}$ such that for every vertex $\bo a_j$ of $\Bbb P_n$ there exists one function $\psi_j$  supported in $B(\bo a_j, \frac34 l_n)$, one of the functions is supported in  $\text{Int}( \Bbb P_n)$ and one is supported in $\text{Int}(\R^2 \setminus \Bbb P_n)$. 
In view of the smallness of the perturbation $P$ of the regular polygon, we can keep the same $n$ charts to describe the boundary of $\partial P$ and use the same  partition of unity as above, for the regular polygon.
The maps of the charts are built in a uniform way as piecewise affine functions having two controlled slopes. 

Moreover, instead of extending $u_1, u_1^*$ we shall extend each function $u_1\psi_j, u_1^*\psi_j$ relying on the special construction given by Stein in \cite[Theorem 5, page 181]{St70}, which takes advantage from the specific graph structure of the boundary. Finally, we use the generic comparison
$$\sum_{j=0}^{n-1} \| v_1\psi_j -v_2\psi_j\|_\infty \le n \| v_1 -v_2 \|_\infty \le n \sum_{j=0}^{n-1} \| v_1\psi_j -v_2\psi_j\|_\infty.$$

\medskip
 \noindent{\bf Step 2. Construction of the smoothed distance functions.} The expression of the Stein extension operator is explicit and relies on regularization of the distance functions to $P, \Bbb P_n$ respectively, say $\Delta_P, \Delta_{\Bbb P_n}$. The construction of these functions is quite delicate and we refer the reader to \cite[Theorem 2, page 171]{St70} for all the details. We have $\Delta_P \in C^\infty( \R^2 \setminus P)$, satisfying
\begin{equation}\label{bobu19}
c_1 d(x, P)\le \Delta _P (x) \le c_2  d(x, P) \text{ for every } x \in P^c
\end{equation}
\begin{equation}\label{bobu20}
\Big |\frac{\partial ^\alpha}{\partial x ^\alpha} \Delta_P (x)\Big |\le B_\alpha ( d(x, P))^{1-|\alpha|},
\end{equation}
 and similar inequalities for $\Delta_{\Bbb P_n}$. The constants $c_1,c_2, B_\alpha$ are independent on $P$. 
 
 In its construction, Stein gives a precise formula for $\Delta_P$, namely
 $$\Delta_P(x) = \sum_k \diam (Q_k) \phi_k(x),$$
 where $Q_k$ consists in a {\it suitable}  partition of $\R^2\setminus \overline P$ in squares and $\phi_k$ are $C^\infty$ functions equal to $1$ on $Q_k$ and vanishing outside a $\frac 98$-dilation of $Q_k$ by the center of $Q_k$.  The partition $(Q_k)_k$ is not arbitrary, the size of the squares being controlled by the distance of the square to the boundary of $P$. 

Assume now that $P$ is a perturbation of $\Bbb P_n$ as above such that $d_H(\partial P, \partial \Bbb P_n) =\vps$. Then,
\begin{equation}\label{bobu18} 
\forall x \in \R^2, \quad |d(x, P)-d(x, \Bbb P_n)| \le \vps.
\end{equation}
Our aim is to slightly modify the construction of the partition $(Q_k)$ for $P$ such that at distance larger than $16 \vps$ from the boundary of $P$, the partition coincides with the one associated to $\Bbb P_n$. This will entail that if $d(x, P) > 128 \vps$ then $\Delta_P(x) =\Delta_{\Bbb P_n}(x)$. This is done as follows.
\begin{itemize}
\item We first set the family grids $(2^{-k} \Z^2)_{k \in \Z}$ in $\R^2$ and choose a suitable partition for $\R^2\setminus \Bbb P_n$.
\item We select out from this partition all the squares which intersect the set
$$D^\vps_P=\{x\in \R^2 : d(x, P) \ge 16 \vps\}.$$
\item We use the Stein's method to fill the rest of the partition associated to $P$, namely to cover the open subset of $\R^2 \setminus \overline P$ not yet covered by the selected partition. 
\end{itemize}
Finally, the construction of the functions $\phi_k$ follows the same procedure as Stein. The only difference from the original Stein construction is only the alteration of the partition at distance larger than $16\vps$. In view of \eqref{bobu18},  properties \eqref{bobu19}-\eqref{bobu20} of $\Delta_P$ are preserved. 

The main consequence of this construction is that if $d(x, P) > 128 \vps$ then $\Delta_P(x) =\Delta_{\Bbb P_n}(x)$.

\medskip
 \noindent{\bf Step 3. Comparison of the extensions.} We recall that $u_1$ and $u_1^*$ are uniformly Lipschitz in $\R^2$, as a consequence of Proposition \ref{bobu16}. This plays a crucial role in estimate \eqref{bobu21}. Let us now recall from \cite{St70} how the Stein extension works. We shall simultaneously write the extension of $u_1$ with $E_P$ and the extension of $u_1^*$ with  $E_{\Bbb P_n}$.

Suppose $P, \Bbb P_n$ are above the graphs representing their boundaries on a segment $[m_j,M_j]$, which we suppose, without loss of generality, is contained in the horizontal coordinate axis. 

Let $\tau :[1, +\infty[$ be defined by 
$$\tau (s)= \frac{e}{\pi s} \text{Im} \Big [ \exp \Big (-(s-1)^\frac 14 \exp (-i \frac{\pi}{4})\Big) \Big ].$$
Then
$$\int_1^{+\infty} \tau (s) ds =1, \forall k=1, 2,\dots,\quad  \int_1^{+\infty} s^k \tau (s) ds =0,  \tau (s) \stackrel{s \rightarrow +\infty}{=}  O(s^{-k}).$$
Let $c>0$ be a constant such that
$$\forall (x,y) \in \R^2 \setminus  P, \quad  c \Delta_P (x,y) \ge \phi_j (x) -y,$$

$$\forall (x,y) \in \R^2 \setminus \Bbb P_n, \quad   c \Delta_{\Bbb P_n}  (x,y) \ge \phi^*_j (x) -y,$$
The extension operators are defined for $x \in [m_j,M_j]$ and $ y <\phi_j (x)$ and $ y < \phi_j^*(x)$, respectively,  by
$$ E_P(\psi_j u_1)(x,y)= \int_1^{+\infty} \psi_j(x, y+2 cs \Delta_P(x,y)) u_1(x, y+2 cs \Delta_P(x,y))\tau (s) ds, $$
$$E_{\Bbb P_n}(\psi_j u^*_1)(x,y)= \int_1^{+\infty} \psi_j (x, y+2 cs \Delta_{\Bbb P_n}(x,y)) u_1^*(x, b+y cs \Delta_{\Bbb P_n}(x,y))\tau (s) ds,$$
respectively.

Take a point $(x,y)$ such that $x \in [m_j, M_j]$ and $d((x,y), \partial P) \ge 128 \vps)$. Since $\Delta_P(x,y) =\Delta_{\Bbb P_n}(x,y)$ and $\|\psi_j\|_\infty \le 1$, we get by direct computation
$$|E_P(\psi_j u_1)(x,y)-E_{\Bbb P_n}(\psi_j u^*_1)(x,y)|\le \|u_1-u_1^*\|_{L^\infty(\R^2)}\int_1 ^{+\infty}|\tau(s) | ds= C \|u_1-u_1^*\|_{L^\infty(\R^2)}.$$
To complete the estimate, we evaluate both $E_P( u_1)(x,y)$ and $E_{\Bbb P_n}(u^*_1)(x,y)|$ for $(x,y)$ lying at distance not larger than $130 \vps$ from the boundary of $\Bbb P_n$. Here we take advantage from the fact that there exists $C$, independent on $P$ (see \cite[Theorem 5, page 181]{St70}) such that 
$$\|E_P( u_1)\|_{W^{1, \infty} (\R^2)} \le C  \|u_1\|_{W^{1, \infty} (P)},  \|E_ {\Bbb P_n}(u_1^*)\|_{W^{1, \infty} (\R^2)} \le C \|u_1^*\|_{W^{1, \infty} (\Bbb P_n)}.$$
Since $u_1, u_1^*$ vanish on $\partial P, \partial \mathbb P_n$, respectively, we get that for $(x,y)$ as above we have
$$E_P( u_1)(a,b) \le 130 \vps C  \|u_1\|_{W^{1, \infty} (P)} , E_ {\Bbb P_n}(u_1^*)(x,y) \le  130 \vps C  \|u_1^*\|_{W^{1, \infty} (\Bbb P_n)}.$$
This last inequality concludes the proof.
\end{proof}
As a consequence of the Proposition \ref{bobu14} and Lemma \ref{bobu15}, together with the uniform boundedness of the support of the extended functions, we get the following. 
\begin{cor}
There exist  constants $C$ and $\thetathree \in (0, 1)$ independent on the perturbation, such that 
$$ \|E_P(u_1)-E_{\Bbb P_n}(u_1^*)\|_{H^{2} (\R^2)}\le C (\|u_1-u_1^*\|_{L^\infty(\R^2)}+d(\partial P, \partial \Bbb P_n))^\thetathree.$$
\end{cor}

\subsection{Estimates of the Hessian coefficients along the perturbation}

 In the sequel we collect some $L^\infty$-estimates, necessary for estimates of the coefficients of the Hessian matrix.
Let $\vphi ^*: T^*\ra \R$, $ \vphi: T \ra \R$ be the functions defined in \eqref{def:phi} (the   second kind, in Figure \ref{fig:simple-triangulations}). We assume that $\forall i=0, \dots, n-1$ $|\bo a_i\bo a_i^*|\le \vps$ (which implies $d_H(\partial P, \partial \Bbb P_n) \le \vps$). Then
$$ \|\vphi^*\|_\infty \le 1,  \|\vphi\|_\infty \le 1, \|\nabla \vphi ^*\|_\infty \le \frac{1}{2 \sin \frac{ {2}\pi}{n}}, \|\nabla  \vphi\|_\infty \le \frac{1}{2\sin \frac{ {2}\pi}{n}-2\vps},$$
$$\forall x\in  T^*\cup T, |\vphi^*(x)-\vphi  (x)|\le 1_{T^*\Delta T} + \frac{\vps}{2 \sin \frac{ {2}\pi}{n}} 1_{T^*\cap T},$$
$$\forall x\in  T^*\cup T,  |\nabla \vphi(x)-\nabla \vphi_\vps (x)|\le \frac{2}{2\sin \frac{ {2}\pi}{n}-2\vps}1_{T^*\Delta T} + \frac{2\vps}{(2 \sin \frac{ {2}\pi}{n}-2\vps)^2}1_{T^*\cap T },$$
$$\|u_1^*\|_\infty \le( \lb_1^*)^\frac 12, \quad \|u_1 \|_\infty \le (\lb_1)^\frac 12,$$
$$\|\nabla u_1^*\|_\infty \le (\lb_1^*)^\frac 32, \quad \|\nabla u_1 \|_\infty \le (\lb_1)^\frac 32 (1+\vps).$$
The last inequality takes advantage  from the previous one  and from the fact that the level sets are convex, via the barrier method.

\begin{lemma}\label{bobu26}
Let $g \in H^1_0(B_2)$ and $S \sq \ov B_1$ a segment. We denote $\Phi \in H^{-1} (\R^2)$ defined by
$$H^1(\R^2) \ni \vphi \to \Phi (\vphi) = \int_S g \vphi ds.$$
Then, for every $s \in (0, \frac 12]$ there exists a constant $C_s$ depending only on $s$, such that
$$\|\Phi\|_{H^{-\frac 12-s}(\R^2)} \le C_s \|g\|_{H^1_0(B_2)}.$$
\end{lemma}
\begin{proof}
Indeed, we have
$$|\Phi (\vphi)| = \left|\int_S g \vphi ds\right|\le \|g\|_{L^2(S)} \|\vphi\|_{L^2(S)} $$
$$\le  C_s \|g\|_{H^1_0(B_2)} \|\vphi \|_{H^{ \frac 12+s}(\R^2)}.$$
In the last inequality, we used the classical trace inequality in $H^1(\R^2)$ and the fractional trace inequality in $H^{ \frac 12+s}(\R^2)$ (see \cite[Lemma 16.1]{Ta07}) together with the continuous embedding of $H^s(-1,1) \sq L^2(-1,1)$. 
\end{proof}

\begin{lemma}
Let $S_1= [0,1]\times \{0\}$ and $S_2 = [A_1A_2]$ be two segments of $\R^2$ such that $d_H(S_1, S_2) \le \vps$. 
Let $s_0\ge s >0$ and $g \in H^{1+s} (\R^2)$ with bounded support. There exists a constant $C>0$ such that 
$$\forall \vphi  \in H^1(\R^2), \; \left|\int_{S_1} g \vphi  ds - \int_{S_2} g \vphi ds\right| \le \vps ^{\frac s2} C \| \vphi \|_{H^1(\R^2)}.$$
\end{lemma}
\begin{proof}
We shall make an explicit computation. Let $\tilde S_2= [B_1B_2]$ be the segment on the same line as $S_2$ such that its vertical projection on the horizontal axis is precisely $S_1$. The $\|A_1B_1\| \le \vps$ and  $\|A_2B_2\| \le \vps$.
We have the following estimates.
$$\; \left|\int_{S_2} g \vphi ds - \int_{\tilde S_2} g \vphi  ds\right|\le \int _{[A_1B_1]} |g \vphi  |d \sigma + \int _{[A_2B_2]} |g \vphi  |d \sigma\le $$
$$\le \|g \|_\infty \vps ^\frac 12(\| \vphi \|_{L^2([A_1B_1])} + \| \vphi \|_{L^2([A_2B_2])} ) \le C\vps ^\frac12 \| \vphi \|_{H^1(\R^2)}.$$
Let us introduce the projector $\Pi_1: \tilde S^2 \ni (x,y) \to (x,0)\in S_1$. 
Then,
$$\left|\int_{\tilde S_2} g \vphi  ds -\int_{\tilde S_2} g\circ \Pi_1 \vphi  ds\right|\le \|g\|_{W^{\frac s2, \infty} (\R^2)}(2\vps )^\frac s2 \int _{\tilde S_2} | \vphi | ds.$$
Moreover,
\begin{align*}&\left|\int_{\tilde S_2} g\circ \Pi_1 \vphi  ds -\int_{ S_1} g  \vphi  ds\right|\\
&\le |(|\tilde S_2|-1)|\int_0^1 |g(x, 0)  \vphi (\Pi^{-1} (x,0))| dx+ \int_0^1 |g(x, 0)| \big | \vphi (\Pi^{-1} (x,0))- \vphi (x,0) \big |dx\\
& \le 2 \vps \|g\|_\infty \| \vphi \|_{L^1(\tilde S_2)} + \int_0^1 |g(x, 0) | \int_0^{\Pi^{-1} (x,0)} | \frac{\partial  \vphi }{\partial y} (x,y) |dxdy\\
& \le 2 \vps \|g\|_\infty \| \vphi \|_{L^1(\tilde S_2)} + \|g\|_{L^2(S_1)}  \Big [\int_0^1 \Big ( \int_0^{\Pi^{-1} (x,0)} | \frac{\partial  \vphi }{\partial y} (x,y) | dy \Big ) ^2 dx \Big ]^\frac 12\\
&\le  2 \vps \|g\|_\infty \| \vphi \|_{L^1(\tilde S_2)} + \|g\|_{L^2(S_1)} \Big [ 2\vps  \int_0^1  \int_0^1  \big (\frac{\partial  \vphi }{\partial y} (x,y)\big ) ^2 dx dy\Big ]^\frac 12\\
&= 2 \vps \|g\|_\infty \| \vphi \|_{L^1(\tilde S_2)}  + \|g\|_{L^2(S_1)} (2\vps )^\frac 12 \| \vphi \|_{H^1}.
\end{align*}
Adding all the previous estimates, we conclude the lemma.
\end{proof}
We turn our attention to $\bo U_i$, the solution of \eqref{eq:material-eig-decomposed} - \eqref{eq:normalization-alternative} in $P$.
Recall that the expression of the coefficients of $\bo N_{ij}$ in \eqref{eq:Hessian-lambda-block} does not change when a multiple of the eigenfunction $u_1$ is added to $\bo U_i$. In the following, whenever working with vectorial quantities, estimates are understood component by component.

We drop the index $i$ and we formally write
\begin{equation}\label{bobu28}
\left\{ \begin{array}{rcll}
   -\Delta \bo U - \lb \bo U& =& f& \text{ in }P\\
   \bo U&= &0 & \text{ on }\partial P\\
   \int_P u_1 \bo U dx&= &0
 \end{array} \right.
 \end{equation}
Here $f \in H^{-1} (P, \R^2)$ is defined in \eqref{eq:material-eig-decomposed} and involves the following type of terms (possibly multiplied by geometric quantities)
 $$\lb_1 u_11_T \nabla \vphi, \vphi D^2 u,  \nabla \vphi D^2 u,    \frac{ \partial \varphi }{\partial n} \nabla u {\mathcal H^1}\lfloor S, (\nabla \varphi \nabla u) \bo n {\mathcal H^1}\lfloor S$$
where $S$ is an edge of $T$ and $\bo n$ is the normal. Note that $u_1 \in H^{2+s} (P)$ and all these quantities are controlled for our perturbation, in a norm which is at least $H^{-1+s}$. 
\begin{lemma} \label{bobu27}
For every $s \in [0, \frac 12\wedge (\frac{\pi}{\omega_0}-1))$,  there exists a constant $C_s>0$ not depending on $P$, such that
$$\|\bo U\|_{H^{1+s} (P)} \le C_s.$$
\end{lemma} 
\begin{proof}
One readily gets
$$\left(1 -\frac{\lb_1}{\lb_2} \right) \int _P |\nabla \bo  U |^2 dx \le \|f\|_{H^{-1}(P)} \|  \bo U \|_{H^1_0(P)},$$
which gives, using the Poincar\'e inequality in the orthogonal of $u_1$,
$$\| \bo  U \|_{H^1_0(P)} \le \frac{\lb_2+1}{\lb_2-\lb_1}\|f\|_{H^{-1}(P)}.$$
Taking into account the Andrews-Clutterbuck result \cite{AnCl11} and the structure of $f$,  Lemma \ref{bobu12} gives the conclusion.
\end{proof}

In order  to estimate 
$$\int_{ {\Bbb{R}^2}} |\nabla \bo U^* - \nabla\bo U |^2 dx$$
we rely on the stability estimates for simultaneous domain and right hand side perturbations. Moreover, in view of the definitions of $\bo U^*, \bo U$, we have to  work in the orthogonal on $u, u^*$, and use a correction term built by projection.

 We have the following.
\begin{lemma} \label{bobu29}
There exist positive constants $C, \thetathree >0$,  such that for every admissible perturbation
$$\|\bo U-\bo U^*\|_{H^1_0(B_2, \R^2)} \le C \vps ^\thetathree.$$
\end{lemma}
\begin{proof}
Without restricting the generality we can assume that $P \sq \Bbb P_n$. Indeed, if this is not the case, we compare both $U$ and $U^*$ with the solution $U$ on the regular polygon $(1+\vps) \Bbb P_n$, which  contains both $P$ and $\Bbb P_n$.

We introduce the following auxiliary problem 
\begin{equation}\left\{ \begin{array}{rcll}
   -\Delta \bo V & =&\lb^*  \bo U^* + f^* & \text{ in }P\\
   \bo V &= &0 & \text{ on }\partial P 
 \end{array} \right.
 \end{equation}
which has a classical weak solution. In view of the result of Savar\'e-Schimperna \cite[Theorem 8.5]{SaSc02} 
\begin{equation}\label{bobu42}
 \|\bo V- \bo U^*\|_{L^2} \le C \| \lb^* \bo U^* + f^*\|_{H^{-1} (B_2)} \varepsilon^{\frac 12}. 
 \end{equation}
In the same time, both $\bo V$ and $\bo U^*$ belong to $H^{1+s}$ with controlled norm, so in particular they belong to $W^{\frac s2, \infty}$ with controlled norm. Using again the Gagliardo-Nirenberg inequality for the Stein extension of $\bo V$, we get
$$ \|E_P(\bo V)- \bo U^*\|_{H^1(\Bbb P_n)} \le C \vps ^\thetathree.$$
Note that $\bo U^* \in H^{1+s} (\Bbb P_n)$ and that $H^{1+s} (\Bbb P_n)$ continuosly embedes in $W^{1+\frac s2, 2+s} (\Bbb P_n)$. Consequently, from H\"{o}lder inequality
$$\int_{\Bbb P_n \sm P} |\nabla \bo  U ^* |^2 dx \le \left(\int_{\Bbb P_n \sm P} |\nabla  \bo U ^* |^{2+s} dx\right)^{\frac{2}{2+s}} | \Bbb P_n \sm P| ^ {\frac{s}{2+s}} \le C\vps ^{\frac{s}{2+s}}.$$
Finally,
$$ \|\bo V- \bo U^*\|_{H^1(\Bbb P_n)} \le C \vps ^\thetathree.$$
Let us now introduce the function
$\tilde{\bo{V}}= \bo V - (\int_P \bo Vu_1 dx )u_1\in H^1_0(P)$. Then 
\begin{align*} \|\bo V-\tilde{\bo{V}} \|_{H^1_0(P)} & = \|u_1\|_{H^1_0(P)} \int _P \bo Vu_1 dx \\
& = \|u_1\|_{H^1_0(P)}\left[\int_{\Bbb P_n} (\bo V- \bo U ^*)u_1^* dx + \int_{\mathbb P_n}  \bo V(u_1-u_1^*) dx\right]\le C \vps^\thetathree.
\end{align*}
At the same time,
$$ -\Delta \tilde{\bo{V}} - \lb \tilde{\bo{V}}= \lb_1^*   \bo U^* +f^*-\lb\bo V := \ov f \mbox{ in } {\mathcal D}'(  P)$$
and by straightforward computation 
$$\int_{P} |\nabla  \bo U -\nabla \tilde{\bo{V}} |^2 - \lb_1 ( \bo U -\tilde{\bo{V}})^2 dx = (f -\ov f,   \bo U -\tilde{\bo{V}})_{H^{-1}\times H_0^1}.$$ 
Since both $ \bo U, \tilde{\bo{V}}$ are $L^2$-orthogonal on $u_1$, we get 
$$\| \bo U -\tilde{\bo{V}} \|_{H^1_0(P)}  \le \frac{\lb_2}{\lb_2-\lb_1} \| f -\ov f\| _{H^{-1}}.$$

It remains to estimate $ \| f -\ov f\| _{H^{-1}}$. Since
$$f-\ov f= f-f^* + \lb \bo V-\lb^* \bo  U^*,$$
we can use the stability result \eqref{bobu42} to conclude that
$\| f -\ov f\| _{H^{-1}} \le C \vps^\thetathree$.

\end{proof}
We can now conclude with the following.
\begin{thm}\label{bobu43}
There exists $C, \thetathree >0$ such that for every polygon $P \in {\mathcal P}_n$ satsifying $\forall i=1, \dots, n$, $|\bo a_i\bo a_i^*| \le \vps \le \vps _0$ we have
$$\|\bo N^\lb _{ij} -(\bo N^\lb _{ij})^*\|_\infty \le C \vps ^\thetathree,$$
 
$$\forall k=1, \dots, 2n, \; \;  \;  |\lb_k(\bo N^\lb )-  \lb_k((\bo N^{\lb})^*)|\le    C \vps ^\thetathree.\hskip 3.5cm$$

 \end{thm}
 \begin{proof}
The first inequality is a direct consquence of Lemma \ref{bobu29}. The second one is a further consequence of the Weyl inequality on the stability of eigenvalues for perturbations of a symmetric matrix and on the equivalence of all norms over a finite dimensional space.
\end{proof}

\begin{rem}
	The Hessian matrix of the area of the polygon is constant. As a direct consequence, a similar estimate holds for the Hessian matrix $\bo M^\lambda$ of the scale invariant functional $P\mapsto |P|\lambda_1(P)$.
\end{rem}

\section{Eigenvalues of the Hessian matrix for the regular polygon}
\label{sec:eig-hessian}

We denote again $\Pp =[\bo a_0 \bo a_2 ...\bo a_{n-1}]$ the regular polygon with $n$-sides, centered at the origin, with the vertex $\bo a_0$ at the point $(1,0)$. As well, $\lb_1:= \lb_1(\Pp)$ denotes its first eigenvalue and $u_1:=u_1(\Pp)$ a positive, $L^2$-normalized eigenfunction. We also use the notation $\theta = 2\pi/n$.

As a consequence of the homogeneity of the eigenvalue to rescalings
\[ \lambda_1(tP) = \frac{1}{t^2} \lambda_1(P),\]
the proposition below establishes the equivalence between the original problem \eqref{eq:minimization-lamk} and some unconstrained versions. Its proof is standard and will not be recalled.
\begin{prop} 
	Let $c>0$. The three problems below 
	\begin{equation}
	(L_1): \; \min_{|P|=|\Pp|,\ P \in \mathcal P_n} \lambda_1(P), \quad \quad   (L_2): \;  \min_{P \in \mathcal P_n} |P| \lambda_1(P), \quad \quad  (L_3): \; \min_{P\in \mathcal P_n} \Big (\lambda_1(P) +c |P|\Big )
	\label{eq:constr}
	\end{equation}
	have the same solutions, up to  rescalings.
\end{prop}
For the convenience of the reader, 
we also collect below some well known facts. 

\begin{prop}
	Let $n \geq 3$. Then
	\begin{enumerate}
		\item The first eigenfunction on $\Pp$ has the symmetry of the $n$-gon. 
		\item 
		\begin{itemize}
			\item $\Pp$   is a critical point for  problem  $(L_1)$ above;
			\item any regular $n$-gon is a critical point for  problem  $(L_2)$ above (see Theorem \ref{thm:critical-point});
			\item   the regular n-gon  $\Big (\frac{\lambda_1(\Pp)}{|\Pp| c}\Big ) ^\frac14 \Pp$ is critical for  problem  $(L_3)$ above.
		\end{itemize}
		\item
		If moreover any of the regular $n$-gons above is a local minima for its own problem, then all the others are local minima for their own problems.
	\end{enumerate}
	\label{thm:criticality}
\end{prop}

 \begin{rem}[Symmetry of the first eigenfunction]{\rm 
		On  $\Pp$, the first eigenfunction enjoys the symmetry of the polygon. In particular on all triangles $\Delta O\bo a_i \bo a_{i+1}$ the eigenfunction has the same geometry, symmetric with respect to the bisector of the angle $\widehat {\bo a_iO \bo a_{i+1}}$. As well, the normal derivative of the eigenfunction vanishes on  the segments $[O\bo a_i]$, $[O\bo a_{i+1}]$.
	}
\label{rem:symm-eigenfunction}
\end{rem}
\EEE
\begin{rem}[Optimality conditions]{\rm 
		The existence of other critical polygons than the regular polygon is an open question for $n \geq 4$. In the case of triangles, results in \cite{FragalaVelichkov19} show that the equilateral one is the only possible critical point for the two functionals (first eigenvalue and torsional rigidity) studied here. 
	}
\end{rem}

\begin{prop}
	Let $\Pp$ be the regular polygon defined above. If the Hessian matrix $\bo M^\lambda$ of $P\mapsto |P|\lambda_1(P)$ evaluated at $\Pp$, given in \eqref{eq:simplified-hessian}, has $2n-4$ eigenvalues that are strictly positive then $\Pp$ is a local minimum.
	\label{prop:justification-2n-4}
\end{prop}

\begin{proof}  In the previous section in Theorem \ref{bobu43} it is shown that the coefficients of Hessian matrix are continuous for a local perturbation of the free vertices. Therefore, it would be enough to prove that the Hessian matrix associated to the free variables is positive definite. Fix the two consecutive vertices $\bo a_{n-2},\bo a_{n-1}$ and consider the associated matrix $\widetilde{\bo{M}}$ which is the $(2n-4)\times(2n-4)$ principal submatrix of $\bo M^\lambda$ obtained by removing the last four lines and columns. Then $\widetilde{\bo{M}}$ is the Hessian matrix of the same functional, with the last four variables removed.

First of all, we observe that $\bo M^\lambda$ has $4$ zero eigenvalues which correspond to translations, scalings and rotations which leave the objective function invariant. In Propositions \ref{prop:tx-ty}, \ref{prop:k=0,1} direct proofs are given showing that 
\begin{equation}\bo t_x= \begin{pmatrix} 1\\0\\1\\0\\ \dots\\0\end{pmatrix}, \bo t_y= \begin{pmatrix} 0\\1\\0\\1\\ \dots\\1 \end{pmatrix}, \bo s=\begin{pmatrix} 1\\ 0\\ \cos\frac {2\pi}{n} \\ \sin\frac {2\pi}{n}\\ \dots\\\sin \frac{2(n-1) \pi}{n} \end{pmatrix}, \bo  r=\begin{pmatrix} 0\\-1\\\sin \frac{2 \pi}{n} \\- \cos\frac {2\pi}{n} \\ \dots\\ -\cos \frac{2(n-1) \pi}{n} \end{pmatrix}.
\label{eq:eigenvectors-zero}
\end{equation}
are indeed eigenvectors of $\bo M^\lambda$ associated to the zero eigenvalue. 

Suppose that $\bo M^\lambda$ has $2n-4$ strictly positive eigenvalues (in addition to the four zero eigenvalues described above). The result stated in \cite[Theorem 4.3.28]{horn-matrix} shows that the eigenvalues of $\widetilde{\bo{M}}$ have lower bounds given by those of $\bo M^\lambda$, therefore they are non-negative. Suppose that $\widetilde{\bo{M}}$ has a zero eigenvalues with an eigenvector $\xi\in \Bbb{R}^{2n-4}$. Completing $\xi$ with zeros would give an eigenvector of $\bo M^\lambda$ associated to the zero eigenvalue. This is impossible since taking the last four components of the eigenvectors in \eqref{eq:eigenvectors-zero} gives four independent vectors in $\Bbb{R}^4$. Therefore $\widetilde{\bo{M}}$ is positive definite implying that $\Pp$ is indeed a local minimum for the functional $P \mapsto \lambda_1(P)|P|$. 
\end{proof}

The remaining part of this section is dedicated to the computation of the eigenvalues of $\bo M^\lambda$. In particular, we show that the eigenvalues of $\bo M^\lambda$ can be computed in terms of the first eigenfunction $u_1$ and the solutions $(U_0^1,U_0^2)$ of \eqref{eq:material-eig-decomposed} with the normalization condition $\int_{\Pp} U_0^i u_1 =0,\ i=1,2$. A numerical approach for proving that the matrix $\bo M^\lambda$ has $2n-4$ eigenvalues that are strictly positive is provided in the next section.

\begin{prop}
	1. The vectors $\bo t_x = (1,0,...,1,0) \in \Bbb{R}^{2n}, \bo t_y = (0,1,...,0,1) \in \Bbb{R}^{2n}$ are eigenvectors of $\bo M^\lambda$ associated to the zero eigenvalue.
	\label{prop:tx-ty}
\end{prop}

\begin{proof} The proof is immediate, following the expression of $\bo M^\lambda$ given in \eqref{eq:simplified-hessian}. Proposition \ref{prop:translation-eigenvectors} shows that $\bo t_x$ and $\bo t_y$ are in the kernel of the Hessian of the eigenvalue and are orthogonal to both the gradients of the eigenvalue and of the area. Moreover, they are also in the kernel of the area Hessian \eqref{eq:Hess_area}. Combining all these aspects finishes the proof.
\end{proof}

The following result recalls the symmetry properties of $u$ and $U_0^1,U_0^2$. For simplicity, we use the notation $a(u,v) = \int_{\Pp} \nabla u \cdot \nabla v - \lambda \int_{\Pp} uv$.
\begin{prop} The following holds.
\begin{itemize}
	\item[1.]  The functions   $\partial_x u_1 , \partial_x \varphi_0, \partial_x U_0^1, \partial_y U_0^2$ are even with respect to $y$	
	and the  functions $\partial_y u_1$, $\partial_y \varphi_0$, $\partial_y U_0^1$, $\partial_x U_0^2$  are odd with respect to $y$.
	
	\item[2.] The quantities
	\[ j  \mapsto a(U_0^1,U_j^1), j \mapsto a(U_0^2,U_j^2)\]
	are even with respect to $j$ (modulo $n$) 
	and the quantities
	\[ j  \mapsto a(U_0^1,U_j^2), j \mapsto a(U_0^2,U_j^1)\]
	are odd with respect to $j$ (modulo $n$).
	\end{itemize}
	\label{prop:parity-u}
\end{prop}
The proof is straightforward from the definitions. 

{\bf Change of basis.} In order to deduce more information about the structure of the Hessian matrix it is useful to perform a change of basis so that for each vertex the basis directions correspond to the radial and tangential directions (see Figure \ref{fig:def-triangles}). The Hessian matrix in the new basis is given by the formula $\bo H^\lambda = \bo P^T \bo M^\lambda \bo P$ where $\bo P = (\bo P_{ij})_{1\leq i,j\leq n}$ is a $2\times 2$ block matrix with $\bo P_{jj} = \begin{pmatrix}
\cos(j-1)\theta & -\sin(j-1)\theta \\
\sin(j-1)\theta & \cos(j-1)\theta
\end{pmatrix}$. Of course, $\bo M^\lambda$ and $\bo H^\lambda$ have the same eigenvalues. 

Moreover, the Hessian matrix $\bo H^\lambda$ in this particular basis has an additional property. Indeed, it can be seen that in this basis the matrix does not change when a circular perturbation is applied to the vertices. Therefore the resulting Hessian matrix $\bo H^\lambda$ is circulant with respect to its $2\times 2$ blocks:
\begin{equation}
\bo H^\lambda = \begin{pmatrix}
\bo H_0 & \bo H_1 & ... & \bo H_{n-1} \\ 
\bo H_{n-1} & \bo H_0 & ... & \bo H_{n-2} \\
\vdots & \vdots & \ddots & \vdots \\
\bo H_1 & \bo H_2 & ... & \bo H_0 
\end{pmatrix}
\label{eq:circulant-hessian}
\end{equation}

The spectrum of this block circulant matrix is made of the union of the spectra of the following $n$ matrices of size $2\times 2$
\begin{equation}
\bo B_{\rho_k} =  \bo H_0+\rho_k \bo H_1 + \rho_k^2 \bo H_2 + ... + \rho_k^{n-1} \bo H_{n-1},
\label{eq:2x2-circulant}
\end{equation}
where $\rho_k =\exp(ik\theta)$, $k=0,...,n-1$. Fore more details the reader can refer to \cite{block-circulant} and the references therein. One may note that the symmetry of $\bo H^\lambda$ implies that $\bo H_{n-k} = \bo H_k^T$. Moreover, the $2\times 2$ matrices described in \eqref{eq:2x2-circulant} are all Hermitian (and therefore have real eigenvalues). 

\begin{figure}
	\includegraphics[width=0.3\textwidth]{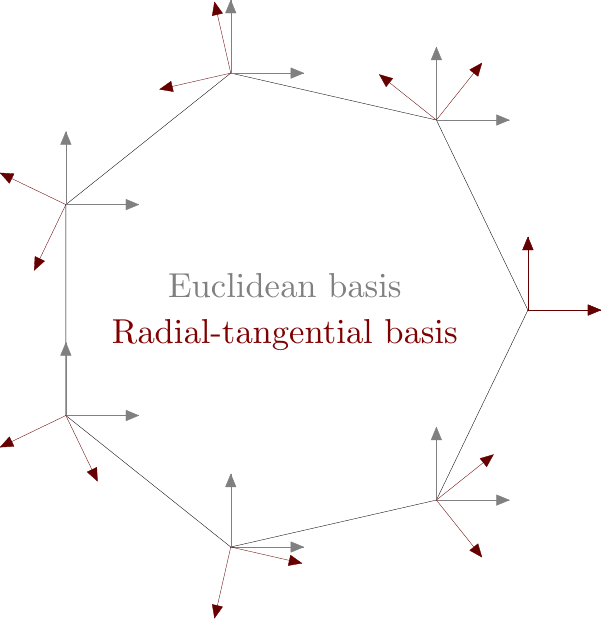}\hspace{0.5cm} 
	\includegraphics[width=0.3\textwidth]{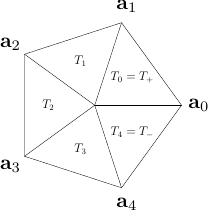}

	\caption{Change of basis to radial and tangential components (left). An example of symmetric triangulation defining $\varphi_j$ for the regular polygon (right).}.
		\label{fig:def-triangles}
\end{figure}

In the following we assume that the triangulation defining the functions $\varphi_j$ in \eqref{def:phi} is symmetric and is made of the triangles $T_j$ having vertices $(0,0),(\cos j\theta, \sin j\theta),(\cos(j+1)\theta, \sin(j+1)\theta)$, $0 \leq j \leq n-1$. For convenience we may use the notation $T_+ = T_0, T_- = T_{n-1}$ (see Figure \ref{fig:def-triangles}). With these notations it can be seen that for $0\leq j \leq n-1$ we have 
\begin{equation}
 \nabla \varphi_j = \frac{1}{\sin \theta}\left[\begin{pmatrix}
\sin (j+1)\theta \\ - \cos (j+1)\theta 
\end{pmatrix}1_{T_j} +\begin{pmatrix}
-\sin(j-1)\theta\\
\cos(j-1)\theta
\end{pmatrix} 1_{T_{j-1}} \right].
\label{eq:grad-phi}
\end{equation}
Furthermore, in view of the symmetry of the eigenfunction, a simple integration by parts shows that 
\begin{equation}
\int_{T_j} \nabla u_1 \cdot \nabla v = \lambda_1 \int_{T_j}u_1v, \ \ \ \forall v \in H_0^1(\Pp)
\label{eq:slice-eig-equation}
\end{equation}

Denote with $\bo M_0, \bo M_1, ..., \bo M_{n-1}$ the blocks of the first line in $\bo M^\lambda$. Then for $\rho_k=\exp(ik\theta)$ a root of unity of order $n$ we have
\[ \bo B_{\rho_k} = \bo M_0+\bo M_1 (\rho_k \bo R_\theta)+...+\bo M_{n-1} (\rho_k \bo R_\theta)^{n-1},\]
where $\bo R_\tau=\begin{pmatrix}
\cos \tau & -\sin \tau \\
\sin \tau & \cos \tau
\end{pmatrix}$ denotes the rotation matrix around the origin with the angle $\tau$ in the trigonometric sense. By abuse of notation we will use the same notation for the rotation of angle $\tau$ around the origin. Recalling the formula \eqref{eq:simplified-hessian} we decompose each one of the blocks $\bo M_j = \bo M_j^1+\bo M_j^2+\bo M_j^3$ with
\[ \bo M_j^1 = -2|\Pp| \begin{pmatrix}
a(U_0^1,U_j^1) & a(U_0^1,U_j^2) \\
a(U_0^2,U_j^1) & a(U_0^2,U_j^2)
\end{pmatrix},
\bo M_j^2 = -\lambda_1 \int_{\Pp} [\nabla \varphi_0 \otimes \nabla \varphi_j - \nabla \varphi_j \otimes \nabla \varphi_0],\]
\[\bo M_j^3 = 2|\Pp| \int_\Pp (\nabla \varphi_0 \cdot \nabla \varphi_j) (\nabla u_1 \otimes \nabla u_1) \]
In the following, we compute separately the matrices $\bo B_{\rho_k}^l = \sum_{j=0}^{n-1} \rho_k^j\bo M_j^l\bo R_{j\theta}$, for $l=1,2,3$. We denote by $\bo{Id}$ the identity matrix and $\bo J = \begin{pmatrix}0 & -1 \\ 1 & 0
\end{pmatrix}$. The area of $\Pp$ is $|\Pp|=0.5n\sin \theta$.

Note that the matrices $\bo M_j^2$ come from the Hessian of the area. Therefore, by straightforward computations we have $\bo M_1^2 = -\lambda_1 \begin{pmatrix}
0 & 0.5 \\
-0.5 & 0
\end{pmatrix}=0.5\lambda_1 \bo R_{\pi/2}$, $\bo M_{n-1}^2 = -\bo M_1^2$ and $\bo M_j^2 = \bo 0$ for $j\notin\{1,n-1\}$. Therefore
\[ \sum_{j=0}^{n-1} \rho_k^j\bo M_j^2 \bo R_{j\theta} = \frac{\lambda_1}{2} (\rho_k\bo R_{\pi/2+\theta}+\bar \rho_k \bo R_{-\pi/2-\theta})=\lambda_1 (-\cos(k\theta) \sin \theta \bo{Id}+ i \sin(k\theta) \cos \theta \bo J). \]

 Furthermore, let $A_{xx} = \int_{T_+} (\partial_x u_1)^2, A_{yy} = \int_{T_+} (\partial_y u_1)^2, A_{xy} = \int_{T_+} \partial_x u_1 \partial_y u_1$. Then we have by the symmetry of the eigenfunction that $A_{xx}+A_{yy} = \lambda_1/n$. The fact that the gradients undergo a rotation when transferred from $T_-$ to $T_+$ implies the matrix equality
\begin{equation} 
\bo R_\theta\begin{pmatrix}
A_{xx} & -A_{xy} \\
-A_{xy} & B_{yy} 
\end{pmatrix}
\bo R_\theta^T
= 
\begin{pmatrix}
A_{xx} & A_{xy} \\
A_{xy} & A_{yy}
\end{pmatrix}.
\label{eq:relation-rotation-gradients}
\end{equation}
We find that $-A_{xx}\sin \theta + A_{yy} \sin \theta + 2A_{xy} \cos \theta = 0$. With the notations above we have 
\[ \bo M_0^3 = 4 |\Pp| |\nabla \varphi_0|^2 \begin{pmatrix}
A_{xx} & 0 \\
0 & A_{yy}
\end{pmatrix},
\bo M_1^3 = 2 |\Pp| (\nabla \varphi_0 \cdot \nabla \varphi_1)_{T_+} \begin{pmatrix}
A_{xx} & A_{xy} \\
A_{xy} & A_{yy} 
\end{pmatrix},\]
\[
\bo M_{n-1}^3 = 2 |\Pp| (\nabla \varphi_0 \cdot \nabla \varphi_{n-1})_{T_-} \begin{pmatrix}
A_{xx} & -A_{xy} \\
-A_{xy} & A_{yy}
\end{pmatrix}
\]
It is immediate to see that $(\nabla \varphi_0\cdot \nabla \varphi_1)_{T_+} = (\nabla \varphi_0 \cdot \nabla \varphi_{n-1})_{T_-} = -\cos \theta |\nabla \varphi_0|^2$. Keeping in mind that $|\Pp| = 0.5 n \sin \theta$ and $|\nabla \varphi_0|_{T_\pm} = 1/\sin \theta$ we get
\[ \bo M_0^3 = \frac{2n}{\sin \theta} \begin{pmatrix}
A_{xx} & 0 \\
0 & A_{xy}
\end{pmatrix},
\bo M_1^3 = -\frac{n\cos \theta}{\sin \theta}\begin{pmatrix}
A_{xx} & A_{xy} \\
A_{xy} & A_{yy} 
\end{pmatrix},
\bo M_{n-1}^3 = -\frac{n\cos \theta}{\sin \theta} \begin{pmatrix}
A_{xx} & -A_{xy} \\
-A_{xy} & A_{yy} 
\end{pmatrix}
\]
Of course, the other blocks on the first line are all equal to zero. Therefore we obtain
\begin{align*}
&\sum_{j=0}^{n-1} \rho_k^j \bo M_j^3 \bo R_{j\theta}  = \bo M_0^3+ \rho_k \bo M_1^3\bo R_\theta+\bar \rho_k \bo M_{n-1}^3 \bo R_\theta^T = \frac{2n}{\sin \theta} \begin{pmatrix}
A_{xx} & 0 \\ 0 & A_{yy}
\end{pmatrix}\\
 +&\frac{2n}{\sin \theta}\begin{pmatrix}
-\cos(k\theta)(A_{xx} \cos^2 \theta+A_{xy} \cos \theta \sin \theta) & 0 \\
0 & -\cos(k\theta)(A_{yy} \cos^2 \theta - A_{xy} \cos \theta \sin \theta)
\end{pmatrix}\\
+ i&\frac{2n}{\sin \theta} 
\begin{pmatrix}
0 & -\sin(k\theta)(-A_{xx} \cos \theta\sin \theta+ A_{xy}\cos^2 \theta) \\
-\sin(k\theta)(A_{yy} \cos \theta\sin \theta+A_{xy}\cos^2 \theta) & 0 
\end{pmatrix}
\end{align*}

	It can be noted that since $A_{xx}+A_{yy}=\lambda_1/n$ and $-A_{xx}\sin \theta+A_{yy}\sin \theta+2A_{xy}\cos \theta=0$ we can deduce that 
	\begin{equation}
	\lambda_1 = 2n\left(A_{xy}-\frac{\cos\theta}{\sin \theta}A_{xy}\right)
	\label{eq:criticality-tplus}
	\end{equation} 
	Using these relations and the computations above we find that
	\[\sum_{j=0}^{n-1} \rho_k^j (\bo M_j^2+\bo M_j^3) \bo R_{j\theta}
	=  \frac{2n(1-\cos(k\theta))}{\sin(\theta)} \begin{pmatrix}
	A_{xx} & 0 \\
	0 & A_{yy}
	\end{pmatrix}
	\]	

It remains to compute the contribution of the terms $\bo M_j^1$. Let us recall that due to the symmetry of the triangulation defining $\varphi_i$ we have, denoting $\bo U_j = ( U_j^1,U_j^2)$ the solutions of \eqref{eq:material-eig-decomposed} with the normalization \eqref{eq:normalization-alternative}, that
\[ \bo U_j(x) = \bo R_{j\theta} \bo U_0(\bo R_{j\theta}^Tx).\]
Note that this implies that $\bo R_{j\theta}^T \bo U_j = \bo U_0 \circ \bo R_{j\theta}^T$.
For $0\leq j \leq n-1$ we have
\begin{align*}
\bo M_j^1 \bo R_{j\theta} = & \begin{pmatrix}
a(U_0^1,U_j^1) & a(U_0^1,U_j^2) \\
a(U_0^2,U_j^1) & a(U_0^2,U_j^2) 
\end{pmatrix}
\begin{pmatrix}
\cos (j\theta) & -\sin(j\theta) \\
\sin(j\theta) & \cos(j\theta)
\end{pmatrix}\\
= & \begin{pmatrix}
a(U_0^1,\cos(j\theta)U_j^1+\sin(j\theta)U_j^2) & a(U_0^1,-\sin(j\theta)U_j^1+\cos(j\theta)U_j^2) \\
a(U_0^2,\cos(j\theta)U_j^1+\sin(j\theta)U_j^2) & a(U_0^2,-\sin(j\theta)U_j^1+\cos(j\theta)U_j^2) 
\end{pmatrix} \\
= & \begin{pmatrix}
a(U_0^1,U_0^1\circ \bo R_{j\theta}^T) & a(U_0^1,U_0^2\circ \bo R_{j\theta}^T) \\
a(U_0^2,U_0^1\circ \bo  R_{j\theta}^T) & a(U_0^2,U_0^2\circ  \bo R_{j\theta}^T)
\end{pmatrix}.
\end{align*}

\begin{rem}{\rm
    For $0\leq j \leq n-1$ the sum of the elements which are not on the diagonal of $\bo M_j^1 \bo R_{j\theta}$ is zero. This is a consequence of the fact that $a(U_0^1,U_0^2\circ\bo  R_{k\theta}^T) = -a(U_0^2,U_0^1\circ\bo  R_{k\theta}^T)$ which simply comes from the change of variables $y = \bo R_{k\theta}^Tx$ and the fact that $U_0^2$ is odd with respect to $y$ and $U_0^1$ is even with respect to $y$ (see Proposition \ref{prop:parity-u}).
    }
    \label{rem:zerooffdiagsum-M1}
\end{rem}

The next result shows that the eigenvalues of $\bo B_{\rho_k}$, and as a consequence those of $\bo M^\lambda$, can be expressed in terms of $u_1,U_0^1,U_0^2$.
\begin{thm}
	For $0\leq k \leq n-1$ we have $\bo B_{\rho_k} = \begin{pmatrix}
	\alpha_k & i\gamma_k \\-i \gamma_k & \beta_k 
	\end{pmatrix}$
	with 
	\begin{align*}
	\alpha_k & = \frac{2n(1-\cos(k\theta))}{\sin \theta} \int_{T_0}(\partial_x u_1)^2-2|\Pp| a(U_0^1,\sum_{j=0}^{n-1} \cos(jk\theta)(\cos (j\theta) U_j^1+\sin (j\theta) U_j^2) ) \\
	\beta_k & = \frac{2n(1-\cos(k\theta))}{\sin \theta}\int_{T_0} (\partial_y u_1)^2 -2|\Pp| a(U_0^2,\sum_{j=0}^{n-1} \cos(jk\theta)(-\sin (j\theta) U_j^1+\cos (j\theta) U_j^2) )\\
	\gamma_k & = -2|\Pp|a(U_0^1,\sum_{j=0}^{n-1} \sin(jk\theta)(-\sin (j\theta) U_j^1+\cos (j\theta) U_j^2) )\\ & = 2|\Pp|a(U_0^2,\sum_{j=0}^{n-1} \sin(jk\theta)(\cos (j\theta) U_j^1+\sin (j\theta) U_j^2) )
	\end{align*}
	
	Moreover, the eigenvalues of $\bo B_{\rho_k}$ are given by
	\[  \mu_{2k} = 0.5(\alpha_k+\beta_k- \sqrt{(\alpha_k-\beta_k)^2+4\gamma_k^2}), \ \mu_{2k+1} = 0.5(\alpha_k+\beta_k+ \sqrt{(\alpha_k-\beta_k)^2+4\gamma_k^2}).\]
	
	As a consequence, the eigenvalues of the Hessian matrix $\bo M_\lambda$ given in \eqref{eq:simplified-hessian} are exactly $\mu_j$, $j=0,...,2n-1$. 
	\label{thm:explicit-eigenvalues}
\end{thm}

\begin{proof} In view of the previous computations we have
	{\small	\begin{align*}
		\bo B_{\rho_k} & = \frac{2n(1-\cos (k\theta))}{\sin \theta} \begin{pmatrix}
		A_{xx}& 0 \\
		0 & A_{yy}
		\end{pmatrix} \\
		-2|\Pp| &\sum_{j=0}^{n-1}\begin{pmatrix}
		a(U_0^1, \cos(jk\theta)(\cos (j\theta) U_j^1+\sin (j\theta) U_j^2) ) & a(U_0^1, \cos(jk\theta)(-\sin (j\theta) U_j^1+\cos (j\theta) U_j^2) ) \\
		a(U_0^2, \cos(jk\theta)(\cos (j\theta) U_j^1+\sin (j\theta) U_j^2) )  & 
		a(U_0^2, \cos(jk\theta)(-\sin (j\theta) U_j^1+\cos (j\theta) U_j^2) ) 
		\end{pmatrix}\\
		-2i|\Pp| & \sum_{j=0}^{n-1}\begin{pmatrix}
		a(U_0^1, \sin(jk\theta)(\cos (j\theta) U_j^1+\sin (j\theta) U_j^2) ) & a(U_0^1, \sin(jk\theta)(-\sin (j\theta) U_j^1+\cos (j\theta) U_j^2) ) \\
		a(U_0^2, \sin(jk\theta)(\cos (j\theta) U_j^1+\sin (j\theta) U_j^2) )  & 
		a(U_0^2, \sin(jk\theta)(-\sin (j\theta) U_j^1+\cos (j\theta) U_j^2) ) 
		\end{pmatrix}
		\end{align*}
	}
The formulas follow directly from Proposition \ref{prop:parity-u} and Remark \ref{rem:zerooffdiagsum-M1}.
\end{proof}

In the following, we continue the computation further by using the variational formulations for $(U_j^1,U_j^2)$, $j = 0,...,n-1$. Recall that $\bo R_{j\theta}^T \bo U_j = \bo U_0 \circ \bo R_{j\theta}^T$. We only develop the expressions that are non-zero from the above matrices.

\begin{prop}
	We have the following equalities:
	{\small \begin{align*}
	a(U_0^1,\sum_{j=0}^{n-1} \cos(jk\theta) U_0^1 \circ R_{j\theta}^T)
	 &= \sum_{j=0}^{n-1} (\cos (j+1)k\theta+\cos jk\theta) \int_{T_j}\nabla u_1 \cdot \nabla U_0^1\\ 
	&+ \sum_{j=0}^{n-1} \frac{\cos(j+1)k\theta-\cos jk\theta}{\sin \theta} \int_{T_j}\begin{pmatrix}
	-\sin (2j+1)\theta & \cos (2j+1) \theta \\
	\cos(2j+1)\theta & \sin (2j+1)\theta
	\end{pmatrix}\nabla u_1 \cdot \nabla U_0^1 
	\end{align*}
	\begin{align*}
	a(U_0^2,\sum_{j=0}^{n-1} \cos(jk\theta) U_0^2 \circ R_{j\theta}^T) &
	=\frac{\cos \theta}{\sin \theta}\sum_{j=0}^{n-1} (\cos(j+1)k\theta-\cos jk\theta) \int_{T_j} \nabla u_1 \cdot \nabla U_0^2\\
	&+ \sum_{j=0}^{n-1} \frac{\cos(j+1)k\theta-\cos jk\theta}{\sin \theta} \int_{T_j}
	\begin{pmatrix}
	-\cos(2j+1)\theta & -\sin(2j+1)\theta \\ 
	-\sin(2j+1)\theta & \cos(2j+1)\theta 
	\end{pmatrix}\nabla u_1 \cdot \nabla U_0^2
	\end{align*}
	\begin{align*}
	a(U_0^1,\sum_{j=0}^{n-1} \sin(jk\theta) U_0^2 \circ R_{j\theta}^T) & 
	= \frac{\cos \theta}{\sin \theta}\sum_{j=0}^{n-1} (\sin(j+1)k\theta-\sin jk\theta)  \int_{T_j} \nabla u_1\cdot  \nabla U_0^1\\
	&+ \sum_{j=0}^{n-1}\dfrac{\sin(j+1)k\theta-\sin jk\theta}{\sin \theta} \int _{T_j} \begin{pmatrix}
	-\cos (2j+1)\theta & -\sin (2j+1) \theta \\
	-\sin (2j+1)\theta & \cos (2j+1)\theta
	\end{pmatrix}\nabla u_1 \cdot \nabla U_0^1
	\end{align*}
	\begin{align*}
	a(U_0^2,\sum_{j=0}^{n-1} \sin(jk\theta) U_0^1 \circ R_{j\theta}^T) &  = \sum_{j=0}^{n-1}(\sin (j+1)k\theta  + \sin jk\theta) \int_{T_j} \nabla u_1 \cdot \nabla U_0^2 \\
	&+ \sum_{j=0}^{n-1}\dfrac{\sin(j+1)k\theta-\sin jk\theta}{\sin \theta} \int _{T_j} \begin{pmatrix}
	-\sin (2j+1)\theta & \cos (2j+1) \theta \\
	\cos (2j+1)\theta & \sin (2j+1)\theta
	\end{pmatrix}\nabla u_1 \cdot \nabla U_0^2 
	\end{align*}
	\par
}
	\label{prop:detailed-computations}
\end{prop}

The proof is computational in nature and is detailed in Appendix \ref{appendix:computations}.

\begin{rem}
	A direct consequence of Theorem \ref{thm:explicit-eigenvalues} and Proposition \ref{prop:detailed-computations} is the fact that the eigenvalues of the Hessian matrix $\bo M^\lambda$ of $\lambda_1(\bo x)\mathcal A(\bo x)$ can be expressed explicitly in terms of the first eigenfunction $u_1$ and the couple $(U_0^1,U_0^2)$.
\end{rem}

The previous results allow us to give more details in the particular cases $k\in \{0,1,n-1\}$
\begin{prop}
	If $k=0$ then $\bo B_{\rho_0}= \bo 0$ with associated eigenvalues $\mu_0=\mu_1=0$. This implies that the vectors $\bo s, \bo r \in \Bbb{R}^{2n}$ defined in \eqref{eq:eigenvectors-zero} are eigenvectors of $\bo M^\lambda$.
	
	For $k=1$ we have $\alpha_1=\beta_1=\gamma_1$ and $\bo B_{\rho_1} = \alpha_1 \begin{pmatrix} 1& i \\ -i & 1\end{pmatrix}$. In particular $\mu_{2} = 0, \mu_{3} = \alpha_1$. 
	
	For $k=n-1$ we have $\alpha_{n-1}=\beta_{n-1}=-\gamma_{n-1}$ and $\bo B_{\rho_1} = \alpha_{n-1} \begin{pmatrix} 1& -i \\ i & 1\end{pmatrix}$. In particular $\mu_{2n-2} = 0, \mu_{2n-1} = \alpha_{n-1}=\alpha_1$. 
	\label{prop:k=0,1}
\end{prop}

\emph{Proof:} When $k=0$ the computations in Proposition \ref{prop:detailed-computations} and the fact that $\int_\Omega \nabla U_0^{1,2} \cdot \nabla u_1 = \lambda \int_\Omega U_0^{1,2} u_1 = 0$ imply that $\bo B_{\rho_0} = \bo 0$.

As a consequence if $\bo v \in \Bbb{R}^2$ then $(\bo v,\bo R_\theta\bo v,...,\bo R_{(n-1)\theta}\bo v) \in \Bbb{R}^{2n}$ is an eigenvector of $\bo M^{\lambda}$ associated to the zero eigenvalue. Taking $\bo v = (1,0)$ gives $\bo s$ and taking $\bo v = (0,-1)$ gives $\bo r$.

When $k=1$ let us evaluate 
\begin{align*}
	a(U_0^1,\sum_{j=0}^{n-1} \left(\cos(j\theta)(\cos(j\theta) U_j^1+\sin(j\theta)U_j^2) - \sin(j\theta)(-\sin (j\theta)U_j^1+\cos(j\theta)U_j^2) \right)=\sum_{j=0}^{n-1} a(U_0^1,U_j^1)
\end{align*}

On the other hand, Proposition \ref{prop:tx-ty} shows that $\bo t_x=(1,0,...,1,0)$ is an eigenvector of $\bo M^\lambda$ given in \ref{eq:simplified-hessian} for a zero eigenvalue. Therefore, the scalar product of the first line of $\bo M^\lambda$ with $\bo t_x$ is zero and we obtain
\[-2|\Omega| \sum_{j=0}^{n-1} a(U_0^1,U_j^1) +\frac{2n(1-\cos \theta)}{\sin \theta} A_{xx} = 0.\]
Using the relations computed above we find that $\alpha_k-\gamma_k=0$. 

Using the second formula for $\gamma_k$ in Theorem \ref{thm:explicit-eigenvalues} and the fact that $\bo t_y= (0,1,...,0,1)$ is an eigenvector of $\bo M^\lambda$ from Proposition \ref{prop:tx-ty} we find that $\beta_k = \gamma_k$. The case $k=n-1$ follows from $\bo B_{\rho_{n-1}} = \overline{\bo B_{\rho_1}}$. \hfill $\square$

\begin{cor}
	We have $\bo B_{\rho_k} = \overline{\bo B_{\rho_{n-k}}}$ (with indices modulo $n$). Therefore:
	
	1. $\bo B_{\rho_k}$ and $\bo B_{\rho_{n-k}}$ have the same eigenvalues. 
	
	2. If $n$ is odd then the spectrum of $\bo M^\lambda$ consists of $4$ zero eigenvalues and $n-2$ double eigenvalues.
	
	3. If $n$ is even then $\bo B_{\rho_{n/2}}$ is diagonal and the spectrum of $\bo M^\lambda$ consists of $4$ zero eigenvalues, $n-4$ double eigenvalues and another two eigenvalues that can be found on the diagonal of $\bo B_{\rho_{n/2}}$.
\end{cor}

 For the sake of completeness, in the following we give a short proof that the regular polygon is a critical point for $P\mapsto |P|\lambda_1(P)$. This result is known and can be recovered, for instance, using ideas from \cite{FragalaVelichkov19} or \cite[Chapter 1]{beniphd}. The proof given below relies on the representation formulas for the gradient given in Theorem \ref{thm:grad-eig}.
	
	\begin{thm}
		The regular polygon is a critical point for $\bo x \mapsto \mathcal A(\bo x)\lambda_1(\bo x)$.
		\label{thm:critical-point}
	\end{thm}
	
	\begin{proof}
		Fix the regular polygon $\Bbb P_n$ inscribed in the unit circle with $\bo a_0 = (1,0)$ and denote by $\lambda_1$ its first eigenvalue.
		Consider the functions $\varphi_i$, $i=0,n-1$ defined in \eqref{def:phi} and suppose they are symmetric like in the right picture in Figure \ref{fig:simple-triangulations}. For $i \in \{0,...,n-1\}$, the components $2i, 2i+1$ of the gradient of the objective function are given by
		\[ \lambda_1 \int_{\Bbb P_n} \nabla \varphi_i + |\Bbb P_n| \int_{\Bbb P_n} \bo S_1^\lambda \nabla \varphi_i. \]
		In view of the symmetry of the polygon and of the first eigenfunction, it is enough to perform the computations for $i=0$. 
		
		We have $\varphi_0 = (1,-1/\tan\theta)1_{T_+}+(1,1/\tan\theta)1_{T_-}$. This already shows that 
		\begin{equation}\lambda_1 \int_{\Bbb P_n} \nabla \varphi_0 = \frac{2\lambda_1}{n} |\Bbb P_n| \begin{pmatrix}
		1\\ 0
		\end{pmatrix}.
		\label{eq:criticality-part1}
		\end{equation}
		Using the expression of $\nabla \lambda_1(\bo x)$ and \eqref{eq:slice-eig-equation} we find that
		\[|\Bbb P_n| \int_{\Bbb P_n} \bo S_1^\lambda \nabla \varphi_0 =  |\Bbb P_n| \int_{\Bbb P_n} -2(\nabla u_1\otimes \nabla u_1)\nabla \varphi_0 =   -4|\Bbb P_n| \begin{pmatrix}
		\int_{T_+} [(\partial_x u_1)^2-\frac{1}{\tan \theta} \partial_x u_1 \partial_y u_1]\\
		0
		\end{pmatrix}.
		\]
		Using \eqref{eq:criticality-tplus} we simplify the above expression to 
		\begin{equation}	|\Bbb P_n| \int_{\Bbb P_n} \bo S_1^\lambda \nabla \varphi_0  = -\frac{2\lambda_1}{n}|\Bbb P_n| \begin{pmatrix} 1\\ 0
		\end{pmatrix}.
		\label{eq:criticality-part2}
		\end{equation}
		Adding \eqref{eq:criticality-part1} and \eqref{eq:criticality-part2} we find that the first two components of the gradient of $\bo x \mapsto \mathcal A(\bo x)\lambda_1(\bo x)$ are zero. By symmetry, all the other components are zero and $\Bbb P_n$ is indeed a critical point.
		\end{proof}

\section{A priori error estimates for the coefficients of the Hessian matrix}\label{bobu200.s5}

The eigenvalues of $\bo M^\lambda$ are described analytically in the previous section, but the formulae do not allow us to prove that these eigenvalues are non-negative. In view of Proposition \ref{prop:justification-2n-4} proving that $\bo M^\lambda$ has $2n-4$ eigenvalues that are strictly positive is enough to infer the local minimality of the regular polygon. In this section we describe how we can certify numerically this fact. In order to achieve this we provide a priori error estimates concerning numerical approximations based on finite elements for $\alpha_k,\beta_k,\gamma_k$ given in Theorem \ref{thm:explicit-eigenvalues}.

First we refer to classical certified estimates for the approximation of the first eigenpair and of the  second eigenvalue on the regular polygon $\Pp$ using $\bf P_1$ finite elements. In a second step, we get certified estimates for the finite element approximation of the function $\bo U_i$. In the last step we get certified approximation results for the coefficients of the Hessian matrix. 

\subsection{Step 1. Certified approximation of the first eigenpair and of the second eigenvalue.} 
\label{sec:eig-estimates}

 In the literature one can find certified approximation for the first eigenvalue in regular polygons (see for instance \cite{Jo17}). We shortly recall of the results  of  \cite[Theorem 4.3]{LiOi13}. 

 Let us consider a triangulation ${\mathcal T}^h$ of $\Pp$. In each triangle $T_i\in {\mathcal T}^h$, the ratio between the smallest edge and the middle one $L_i$ is denoted $\alpha_i$ and the angle between these two edges is $\tau_i$. Then, we denote
$$C(T_i):= 0.493  L_i\frac{1+ \alpha_i^2+\sqrt{1+2 \alpha_i^2 \cos(2\tau_i) + \alpha_i^4}}{\sqrt{2\big (1+ \alpha_i^2-\sqrt{1+2 \alpha_i^2 \cos(2\tau_i) + \alpha_i^4}\Big)}}.$$
Following \cite[Section 2]{LiOi13}, we introduce the constant
$$C_1= \sup_h \frac{C(T_i)}{h},$$
where the parameter $h$ dictating the size of the mesh is the size of the median edge.
Let us denote $\mathcal V^h$ the finite element space associated to ${\mathcal T}^h$ with  $\bf P_1$ finite elements. Denote by $\lb_{k,h}, u_{k,h}$ the $k$-th eigenvalue of $\Pp$ and its associated eigenfunction approximated in $\mathcal V^h$, solving
\begin{equation}
u_{k,h} \in \mathcal V^h, \int_\Pp \nabla u_{k,h}\cdot \nabla v_h = \lambda_{k,h}\int_\Pp u_{k,h} v_h, \ \ \ \forall v_h \in \mathcal V^h.
\label{eq:eigenvalue-finite-elements}
\end{equation}
Results of \cite{verified_eigenvalue_evaluation} show that
$$ \forall k \ge 1, \;\; \lb_{k,h} > \lb_k> \frac{\lb_{k,h}}{1+ C_1^2 h^2 \lb_{k,h}^2}.$$
As a direct consequence we have
\begin{equation}
 |\lambda_k-\lambda_{k,h}| \leq \lambda_{k,h}^3 C_1^2/(1+C_1^2h^2\lambda_{k,h}^2)\ h^2.
 \label{eq:difflam}
\end{equation}

Denoting $\Pi_{1,h}$ the Lagrange interpolation operator on the vertices of triangles of ${\mathcal T}^h$, for functions  $u \in H^2(\Pp)$ we have
$$\|\nabla u-\nabla \Pi_{1,h} (u)) \|_{L^2} \le C_1 h \|D^2 u\|_{L^2}.$$
 For each $u \in H^1_0(\Pp)$ let us denote $P_h(u)$ the projection of $u$ onto the finite element space $\mathcal V^h$, namely the solution of 
\begin{equation}
P_h(u) \in \mathcal V^h, \int_{\Pp} (\nabla u- \nabla P_h(u), \nabla v_h)dx =0,\ \ \ \forall v_h \in \mathcal V^h.
\label{eq:projection-h}
\end{equation}
Then 
\begin{equation}\label{eq:u-phu}
\|\nabla u- \nabla P_h(u) \|_{L^2} \le C_1 h \|D^2 u\|_{L^2} \ \  \text{ and } \ \ \|u-P_h(u)\|_{L^2} \leq C_1h \|\nabla u- \nabla P_hu\|_{L^2}.
\end{equation}
In particular, for $u = u_1 \in H^2$, using $\|D^2u\|_{L^2} = \|\Delta u\|_{L^2}$ (\cite[Theorem 4.3.1.4]{grisvard}), we get 
\begin{equation}\label{eq:error_xuefeng}
\|\nabla u_1-\nabla P_h(u_1)) \|_{L^2} \le C_1 h \|D^2 u_1\|_{L^2}  = C_1 h\lambda_1 \ \  \text{ and } \ \ \|u_1-P_h(u_1)\|_{L^2} \leq C_1^2h^2 \lambda_1.
\end{equation}
In order to estimate the error for the eigenfunction, let $u_{1,h}$ be an $L^2$-normalized, finite element approximation of the first eigenfunction given by \eqref{eq:eigenvalue-finite-elements}.

Let us denote by $p= P_h(u_1)$ and decompose $p= \alpha u_{1,h}+ \overline p$, where $\int_ {\Pp} \overline p u_{1,h} dx =0$, $\alpha \in \R$. Note that changing the sign of $u_{1,h}$ still gives an $L^2$-normalized solution, therefore we may assume $\alpha>0$ in the previous decomposition.

As we know that 
$$\int_\Pp \nabla p \cdot \nabla v_h = \int_\Pp \nabla u \cdot \nabla v_h = \lambda_1\int_\Pp u_1v_h, \ \ \ \forall v_h \in \mathcal V^h,$$
we get
$$\int_\Pp \nabla \overline p \cdot \nabla v_h - \lambda_{1,h} \int_\Pp \overline p v_h  = \int_{\Pp} ( \lb_1 u_1-  \lb_{1, h}p)v_h, \ \ \ \forall \mathcal V^h.$$
Using the Poincar\'e inequality on the orthogonal of $u_{1,h}$ in $\mathcal V^h$, we get 
$$\frac{\lambda_{2,h}-\lambda_{1,h}}{\lambda_{2,h}}\int_{\Pp} |\nabla \overline p|^2 dx \le \|  \lb_1 u_1-  \lb_{1, h}p\|_{L^2} \frac{1}{\sqrt{\lambda_{2,h}}} \|\nabla \overline p\|_{L^2} dx,$$
or
\begin{equation}
 \lambda_{2,h}^{1/2}\|\overline p\|_{L^2} \leq \|\nabla \overline p \|_{L^2} \le \frac{\lb_{2, h}^\frac 12}{(\lb_{2, h}-\lb_{1, h})}\Big (|\lb_1-\lb_{1, h}|+ \lb_{1, h}\|u_1- p\|_{L^2}\Big).
 \label{eq:estimate-p-bar}
\end{equation}

We obtain the error estimate
\begin{equation}
\|\nabla u_1- \nabla u_{1,h}\|_{L^2} \le \|\nabla u_1-\nabla p\|_{L^2} + \frac{|1-\alpha|}{|\alpha |} \|\nabla p\|_{L^2} + \frac{1}{\alpha } \|\nabla \overline p \|_{L^2}.
\label{eq:diffgradu}
\end{equation}

We compute the following bounds for $\|p\|_{L^2}, \|\nabla p\|_{L^2}$, which are immediate from the definition of $p$ and the projection operator $P_h$:
\[ \| \nabla p\|_{L_2}^2 = \int_\Pp \nabla u_1 \cdot \nabla p = \lambda_1 \int_\Pp u_1p \leq \lambda_1 \|p\|_{L^2}\leq \lambda_1 (\|u_1\|_{L^2}+\|p-u_1\|_{L^2})\]

In order to conclude we need bounds for $\alpha$.  We have $\int_\Pp p^2=\alpha^2 + \int_\Pp \overline p^2$, which shows that
\[ |1-\alpha| \leq |1-\alpha^2| \leq \int_\Pp \overline p^2 +\int_\Pp (u^2-p^2) \leq \int_\Pp \overline p^2 + \|u_1-p\|_{L^2}(2+\|u_1-p\|_{L^2}).\]
This estimate can be written in a quantitative form using \eqref{eq:estimate-p-bar} and \eqref{eq:error_xuefeng}. Since $\alpha>0$, for $h$ small enough, an explicit lower bound for $\alpha$ can also be found.

In the same way we obtain the $L^2$ error estimate for the first eigenfunction
\begin{equation}
\|u_1-u_{1,h}\|_{L^2} \leq \|u_1-p\|_{L^2}+\frac{|1-\alpha|}{|\alpha|} \|p\|_{L^2}+\frac{1}{\alpha} \|\overline p\|_{L^2}.
\label{eq:uL2}
\end{equation}

It can be noted that the optimal rates of convergence are obtained in \eqref{eq:diffgradu} and \eqref{eq:uL2}. Moreover, the term of order $O(h)$ in \eqref{eq:diffgradu}, which dominates the estimates comes from the interpolation error bound for $\|\nabla u_1-\nabla p\|_{L^2}$ while the remaining terms are of higher order $O(h^2)$.

\medskip

\subsection{Step 2. Certified approximation of $\bo U_j$.} 
\label{sec:estimates-U}
We begin with some generic approximation results for solutions of the Laplace equation with Dirichlet boundary conditions with singular right hand sides.

\begin{lemma}\label{bobu36.l}
Let $\gamma \in (0, \frac 12)$ and $v$ the solution of \eqref{bobu02} on $\mathcal P_n$ with $f \in H^{-\frac 12 -\gamma} (\R^2)$. Then
\begin{equation}\label{bobu36}
\|\nabla v- \nabla P_h(v)\|_{L^2} \le  \|f\|_{H^{-\frac 12 -\gamma} (\R^2)} ( C_1 h )^{\frac 12 -\gamma}(1+\frac {1}{\lambda_1})^{\frac 12 +\gamma}. \end{equation} 
and
\begin{equation}\label{bobu36.2}
\|  v-   P_h(v)\|_{L^2} \le   \|f\|_{H^{-\frac 12 -\gamma} (\R^2)} ( C_1 h )^{\frac 32 -\gamma}(1+\frac {1}{\lambda_1})^{\frac 12 +\gamma}. \end{equation} 
\end{lemma}
\begin{proof}
By the Aubin-Nitsche argument we get
$$\|v-P_h(v)\|_{L^2} \le C_1h \|\nabla v-\nabla P_h(v)\|_{L^2}.$$
To prove that, it is enough to introduce 
$$\xi \in H^1_0(\mathcal P_n), -\Delta \xi = v-P_h(v) \mbox{ in } H^1_0(\mathcal P_n).$$
Then $\xi \in H^2(\mathcal P_n)$ and using \cite[Theorem 4.3.1.4]{grisvard} we get
$$\|D^2 \xi\|_{L^2}  = \|\Delta \xi\|_{L^2} =  \|v-P_h(v)\|_{L^2},$$
so that
\begin{align*}
\|v-P_h(v)\|_{L^2}^2 &= \int_{\mathcal P_n} \nabla \xi\cdot \nabla (v-P_h(v) ) dx = \int_{\mathcal P_n} \nabla (\xi-\Pi_{1,h} \xi)\cdot \nabla (v-P_h(v) ) dx\\
&\le C_1h \|D^2 \xi\|_{L^2} \|\nabla v-\nabla P_h(v) \|_{L^2}\le C_1h \|v-P_h(v) \|_{L^2} \|\nabla v-\nabla P_h(v) \|_{L^2}.
\end{align*}
Since $f \in H^{-\frac 12 -\gamma} (\R^2)$ then
$$\|\nabla v-\nabla P_h(v)\|_{L^2}^2= (f, v-P_h(v))_{H^{-1}\times H^1_0} \le \|f\|_{H^{-\frac 12 -\gamma} (\R^2)}\|v-P_h(v)\|_{H^{\frac 12 +\gamma }(\R^2)}.$$
From the Gagliardo-Nirenberg interpolation inequality (all norms in $\R^2$ are taken to be the Fourier transform ones), we get 
\begin{align*}\|\nabla v-\nabla P_h(v)\|_{L^2}^2 & \le  \|f\|_{H^{-\frac 12 -\gamma} (\R^2)}\|v-P_h(v)\|_{L^2}^{\frac 12 -\gamma}\|v-P_h(v) \|_{H^1(\R^2)}^{\frac 12 +\gamma}  \\
& \le   \|f\|_{H^{-\frac 12 -\gamma} (\R^2)} ( C_1 h )^{\frac 12 -\gamma}(1+\frac {1}{\lambda_1})^{\frac 12 +\gamma} \|\nabla v-\nabla P_h(v)\|_{L^2}.\end{align*}
Finally, we get the conclusion.
\end{proof}

\begin{thm}\label{thm:estimate-singular}
Let $U\in H^1_0(\Pp)$  be the solution of 
\begin{equation}\label{bobu28.1}
\left\{ \begin{array}{rcll}
   -\Delta U - \lb_1U& =&f& \text{ in }  \Pp\\
   U&= &0 & \text{ on }\partial   \Pp\\
   \int_{\Pp} u_1U dx&= &0
 \end{array} \right.
 \end{equation}
where $(f,u_1)_{H^{-1},H_0^1} = 0$, $f= f^{\reg}+{f^{\sing}} $ with   $f^{\reg} \in L^2(\Pp)$ and $f^{\sing} \in H^{-\frac 12 -\gamma} (\Pp)$. Assume $f_h$ is a numerical approximation in $H^{-1}$ of $f$ which verifies $(f_h,u_{1,h})_{H^{-1},H_0^1}=0$ and $(u_{1,h}, \lb_{1,h})$ a numerical approximation of $(u_1, \lb_1)$ in $H^1_0(\Pp)\times \R$. Denote
 $U_h $ the finite element solution in $\mathcal V^h$ for
\begin{align}\forall v \in \mathcal  V^h, \quad \int_{\Pp} (\nabla U_h \cdot \nabla v-\lambda_{1,h} U_hv) \, dx& =( f_h, v)_{H^{-1}\times H^1_0} \label{eq:material-eig-decomposed.h}
\end{align}
together with the normalization
\begin{equation}
\int_{\Pp} u_{1,h}U_h  \, dx = 0,
\label{eq:normalization.h}
\end{equation}
Then
\begin{align*}
\|\nabla U-\nabla U_h\|_{L^2} & \le  C_1 h \| \lb_1 U+ f^{\reg}\|_{L^2}+ \|f^{\sing}\|_{H^{-\frac 12 -\gamma}} ( C_1 h )^{\frac 12 -\gamma}(1+\frac {1}{\lambda_1})^{\frac 12 +\gamma}\\
&+   \lambda_{1,h}^{\frac{1}{2}}\Big( (C_1 h)^2 \| \lb_1 U+ f^{\reg}\|_{L^2}
	+ \|f^{\sing}\|_{H^{-\frac 12 -\gamma} (\R^2)} ( C_1 h )^{\frac 32 -\gamma}(1+\frac {1}{\lambda_1})^{\frac 12 +\gamma} \\
&	+\|V\|_{L^2}\|u_{1,h}-u\|_{L^2}\Big)\\
&+\frac{\lb_{2,h}^\frac 12}{\lb_{2,h}-\lb_{1,h}} \Big (|\lb_{1,h}-\lb_1| \|U\|_{L^2} + \lb_{1,h}\|U-P_h(U)\|_{L^2}\\
&\hfill + (1+ \lb_{2,h})^\frac 12\|f-f_h\|_{H^{-1}}\Big).
\end{align*}

\end{thm}
\begin{proof}
We denote $U_{\reg}, U_{\sing}$ the solutions of 
$$U_{\reg} \in H^1_0(\Pp), \;\; -\Delta U_{\reg}= \lb_1 U+ f^{\reg}, \quad U_{\sing}\in H^1_0(\Pp),\;\; -\Delta U_{\sing} =   f^{\sing}, $$
so that $U=U_{\reg}+U_{\sing}$.

 We introduce the auxiliary functions
$V_{\reg}, V_{\sing} \in \mathcal  V^h$, $V_{\reg}=P_h(U_{\reg}), V_{\sing}=P_h(U_{\sing})$ the finite element solutions of
$$V_{\reg}\in \mathcal  V^h, \quad -\Delta V_{\reg} = \lb_1 U +f^{\reg} \quad V_{\sing} \in \mathcal  V^h, \quad -\Delta V_{\sing}= f^{\sing}.$$
For $V_{\sing}$, the estimate \eqref{bobu36} from Lemma \ref{bobu36.l} holds, and gives   
$$\|\nabla U_{\sing}-\nabla V_{\sing}\|_{L^2} \le  \|f^{\sing}\|_{H^{-\frac 12 -\gamma}} ( C_1 h )^{\frac 12 -\gamma}(1+\frac {1}{\lambda_1})^{\frac 12 +\gamma},$$
while for $V_{\reg}$ the estimate from \eqref{eq:error_xuefeng} gives
$$\|\nabla U_{\reg}- \nabla V_{\reg}  \|_{L^2} \le C_1 h \| \lb_1 U+ f^{\reg}\|_{L^2}.$$
Let us denote $V=V_{\reg}+V_{\sing}$ and define $\tilde V= V-(\int_{\Pp} Vu_{1,h} dx ) u_{1,h}$. Then we have 
\begin{equation}\label{bobu80}
\|\nabla \tilde V-\nabla V\|_{L^2} = \lambda_{1,h}^{\frac{1}{2}} 
\left|\int_{\Pp} (Vu_{1,h} -Uu_1) \right| \leq \lambda_{1,h}^{\frac{1}{2}}(\|U-V\|_{L^2}+\|V\|_{L^2}\|u_{1,h}-u\|_{L^2}).
\end{equation}

We have that $\tilde V$ is the finite element solution of
$$\tilde V \in \mathcal  V^h, \quad -\Delta \tilde V-\lb_{1,h} \tilde V= \lb_1 U +f- \lb_{1,h} V,$$
which gives
$$\int_{\Pp} |\nabla \tilde V -\nabla U_h|^2 dx - \lb_{1,h}\int_{\Pp} | \tilde V -U_h|^2 dx= ( \lb_1 U +f- \lb_{1,h} V-f_h, \tilde V-U_h)_{H^{-1}\times H^1_0}.$$
By the Poincar\'e inequality in the orthogonal of $u_{1,h}$ we get
$$\left(1- \frac{\lb_{1,h}}{\lb_{2,h}} \right) \int_{\Pp} |\nabla \tilde V -\nabla U_h|^2 dx\le \|\lb_1 U-\lb_{1,h} V\|_{L^2}\| \tilde V-U_h\|_{L^2}+ \|f-f_h\|_{H^{-1}} \|  \tilde V -  U_h\|_{H^1}$$
$$\le \|\lb_1 U-\lb_{1,h} V\|_{L^2} \lambda_{2,h}^{-\frac{1}{2}} \| \nabla \tilde V-\nabla U_h\|_{L^2}+  \|f-f_h\|_{H^{-1}}\left(1+\frac{1}{\lb_{2,h}}\right)^\frac 12  \| \nabla \tilde V-\nabla U_h\|_{L^2}.$$
Finally,
\begin{equation}\label{bobu81}
 \| \nabla \tilde V-\nabla U_h\|_{L^2(\mathcal P_n)}\hskip 9cm
 \end{equation}
$$\le \frac{\lb_{2,h}^\frac 12}{\lb_{2,h}-\lb_{1,h}} \Big (|\lb_{1,h}-\lb_1| \|U\|_{L^2} + \lb_{1,h}\|U-V\|_{L^2}+ (1+ \lb_{2,h})^\frac 12\|f-f_h\|_{H^{-1}}\Big).$$
\end{proof}

\begin{thm}\label{thm:estimate-singular-L2}
With the notations of Theorem \ref{thm:estimate-singular}, the following estimate holds
$$ \|U-U_{h}\|_{L^2}\le 2C_1h \|\nabla U-\nabla V\|_{L^2} + \|V\|_{L^2}\|u_1-u_{1,h}\|_{L^2}+ \lb_{2,h} ^{-\frac 12} \| \nabla \tilde V-\nabla U_h\|_{L^2}.$$
\end{thm}
\begin{proof}
First, by the Aubin-Nitsche trick, we have $\|U-V\|_{L^2}\le C_1h \|\nabla U-\nabla V\|_{L^2}$.
Using the definition of $\tilde V$ we have
$$\|\tilde V-V\|_{L^2(\mathcal P_n)} = \left|\int_{\mathcal P_n} Vu_{1,h} dx\right| = \left|\int_{\mathcal P_n} Vu_{1,h}-U u_1 dx\right|\le \|U-V\|_{L^2}+ \|V\|_{L^2}\|u_1-u_{1,h}\|_{L^2}$$
Finally, we have
$$  \|\tilde V-U_h\|_{L^2(\mathcal P_n)} \le \lb_{2,h} ^{-\frac 12} \| \nabla \tilde V-\nabla U_h\|_{L^2(\mathcal P_n)}\hskip 9cm$$
$$\le \frac{1}{\lb_{2,h}-\lb_{1,h}} \Big (|\lb_{1,h}-\lb_1| \|U\|_{L^2(\mathcal P_n)} + \lb_{1,h}\|U-V\|_{L^2(\mathcal P_n)}+ (1+ \lb_{2,h})^\frac 12\|f-f_h\|_{H^{-1} (\mathcal P_n)}\Big).$$

\end{proof}

\begin{rem}
	\rm
	It can be seen that the estimates from Theorems \ref{thm:estimate-singular}, \ref{thm:estimate-singular-L2} become explicit as soon as $\|f\|_{H^{-1}}, \|f_{\reg}\|_{L^2}, \|f_{\sing}\|_{H^{-\frac{1}{2}-\gamma}}, \|f-f_h\|_{H^{-1}}$ are known. We present below some inequalities that help obtain upper bounds for all other quantities presented here.

    Using the fact that $U$ is orthogonal on the first eigenfunction $u_1$ we find 
		$$ \sqrt{\lambda_2}\|U\|_{L^2}\le  \|\nabla U\|_{L^2}\le \frac{\sqrt{\lb_2(\lb_2+1)}}{\lb_2-\lb_1}\|f\|_{H^{-1}}.$$

	Since $V$ is the projection of $U$ on $\mathcal V_h$ we have $\|\nabla V\|_{L^2} \leq \|\nabla U\|_{L^2}$. Secondly we have $\|V\|_{L^2} \leq \frac{1}{\sqrt{\lambda_1}} \|\nabla V\|_{H^1}$. Since $\tilde V$ is the projection of $V$ on the orthogonal of $u_{1,h}$ in $\mathcal V^h$ we immediately have $\|\tilde V\|_{L^2} \leq \|V\|_{L^2}$ and $\|\nabla \tilde V\|_{L^2} \leq \|\nabla V\|_{L^2}$.
	
	We obtain the following estimates for $U_h$: 
	\[ \|\nabla U_h\|^2_{L^2}-\lambda_{1,h} \|U_h\|_{L^2}^2 \leq \|f_h\|_{H^{-1}}\|U_h\|_{H^1}.\]
	Since $U_h$ is orthogonal on $u_{1,h}$ the first eigenfunction associated to $\lambda_{1,h}$ we have $\|\nabla U_h\|_{L^2}^2 \geq \lambda_{2,h} \|U_h\|_{L^2}^2$ which implies that
	\[ \left(1-\frac{\lambda_{1,h}}{\lambda_{2,h}}\right) \|\nabla U_h\|^2_{L^2} \leq \|f_h\|_{H^{-1}} \sqrt{1+\frac{1}{\lambda_{2,h}}} \|\nabla U_h\|_{L^2}.\]
	This implies
	\[ \sqrt{\lambda_{2,h}}\|U_h\|_{L^2} \leq \|\nabla U_h\|_{L^2} \leq \frac{\sqrt{\lambda_{2,h}(1+\lambda_{2,h})}}{\lambda_{2,h}-\lambda_{1,h}} \|f_h\|_{H^{-1}}.\]
	\label{rem:needed-estimates}
\end{rem}

\begin{rem}{\rm In practice, the singular right hand side that we consider is of the following type. Let $S= [0, 1]\times \{0\}$ and $g= \frac{\partial u_1}{\partial x} $.  We define $f_{\sing}\in H^{-1}(\R^2)$ by
$$\forall \vphi \in H^1(\R^2),\quad (f _{\sing}, \vphi)_{H^{-1}\times H^1} = \int_S g \vphi ds.$$
Then for every $\gamma \in (0, \frac 12)$ we have
$$f_{\sing} \in H^{-\frac 12 -\gamma} (\R^2), \quad  \|f_{\sing}\|_{H^{-\frac 12 -\gamma} (\R^2)} \le
 \left ( \frac{\Gamma(\gamma)}{2\pi^\frac 12 \Gamma(1/2+\gamma)}\right ) ^\frac 12 \|g\|_{L^2(S)}.$$
  Indeed,  for every $\vphi\in H^{\frac 12 +\gamma} (\R^2)$
$$(f_{\sing} , \vphi)_{H^{-\frac 12-\gamma}\times H^{\frac 12 +\gamma}}\le \|g\|_{L^2(S)} \|\vphi\|_{L^2(S)},$$
and use the trace theorem  for $\vphi$ from
$H^{\frac 12 +\gamma} (\R^2)$ onto $L^2(\R\times\{0\})$ with constant   $C_\gamma := \Big ( \frac{\Gamma(\gamma)}{2\pi^\frac 12 \Gamma(1/2+\gamma)}\Big ) ^\frac 12$ (see Pak and Park \cite{PaPa11}). 
}
\label{rem:c-gamma}
\end{rem}

\smallskip

\noindent{\bf Practical estimate.} In order to estimate $ \|g\|_{L^2(S)}$  above, we notice that
$$\int_S g^2dx  = - \int_S  u_1 \frac {\partial ^2 u_1} {\partial x^2} dx\le - \int_S  u_1 \Delta u_1dx= \lb_1 \int_S u_1^2dx.$$
Here, we have used that $  \frac{\partial u_1}{\partial x}\in H^1_0(S)$ for symmetry reasons and angular behavior, together with     $\frac {\partial ^2 u_1} {\partial y^2} \le 0$, from symmetry and  convexity of the level lines of $u_1$. In order to estimate $ \int_S u_1^2dx$, we can use the following
$$\|u_1\|_{L^2(S)} \le  \|u_{1,h}\|_{L^2(S)}+ \|u_1-u_{1,h}\|_{L^2(S)}\le \|u_{1,h}\|_{L^2(S)}+ \Big[\int_{0}^{1} \int_0^{1} \left(\frac{\partial u_1}{\partial y}-\frac{\partial u_{1,h}}{\partial y}\right)^2 dx dy\Big]^\frac 12.$$
If the approximation by finite elements has the symmetry of the $n$-gon, then
$$ \int_{0}^{1} \int_0^{1} \Big (\frac{\partial u_1}{\partial y}-\frac{\partial u_{1,h}}{\partial y}\Big)^2 dx dy\le \lfloor (n+1)/2\rfloor \frac{1}{2n}\int_{\Pp} |\nabla u_1-\nabla u_{1,h} |^2.$$

\noindent{\bf Analysis of the function $\bo U_0= (U_0^1, U_0^2)$.} As we have seen in Section \ref{sec:eig-hessian}, it is enough to concentrate on the function $\bo U_0$ defined in \eqref{eq:material-eig-decomposed} with the normalization condition \eqref{eq:normalization-alternative}. Recall the definition of $\varphi_i$  is given in in \eqref{def:phi} (see also Figure \ref{fig:simple-triangulations}). 

Denote by $\bo f_0\in H^{-1}(\Pp,\Bbb{R}^2)$ the right hand side from \eqref{eq:material-eig-decomposed} for $j=0$. Since the index $0$ is fixed, we shall drop it from $\bo U_0$, its right hand side $\bo f_0$ and $\vphi_0$. We denote  by $T_+=T_0$ and $ T_-=T_{n-1}$ the upper and lower triangles of the support of $\vphi$. Then 
$$\nabla \vphi=  1_{T_+} \left (1,-\frac{1}{\tan \theta }\right ) +   1_{T_-} \left (1,\frac{1}{\tan \theta}\right ) \mbox{ and } \forall x \in \Pp, \; \|\nabla \varphi (x) \|= \frac{1}{\sin \theta} 1_{T_+\cup T_-}(x).$$

The right  hand side $\bo f \in H^{-1} (\Pp, \R^2)$ is the distribution given by
\begin{align*}
\forall v \in C_c^\infty (\Pp), (\bo f,v)_{H^{-1}\times H^1} & =  \int_{\Pp} -(\nabla \varphi\otimes \nabla u_1)\nabla v+2(\nabla u_1 \odot \nabla v)\nabla \varphi \notag \, dx\\
  &+\int_{\Pp} \bo S_1^\lambda \nabla \varphi  \int_{\Pp} u_1v\, dx + \lambda_1\int_{\Pp} u_1v \nabla \varphi\, dx.
\end{align*}
Recall that  $\bo S_1^\lambda = (|\nabla u_1|^2 -\lambda_1 u_1^2)\Id -2\nabla u_1 \otimes \nabla u_1$. Using integration by parts, we notice that
$$ \int_{\Pp} -(\nabla \varphi\otimes \nabla u_1)\nabla vdx+ \lambda_1 \int_{\Pp} u_1v \nabla \varphi\, dx=0.$$
We observe that
in the expression of $\int_{\Pp} \bo S_1^\lambda \nabla \varphi $ the first term cancels for symmetry reasons. We may also use the fact that the regular polygon is critical for $\lambda_1(\Pp)|\Pp|$ so that 
$$\int_{\Pp} \bo S_1^\lambda \nabla \varphi = \int_{\Pp} -2(\nabla u_1 \otimes \nabla u_1) \nabla \varphi  = -\frac{\lambda_1}{|\Pp|} \int_{\Pp} \nabla \varphi = \begin{pmatrix}
	-\frac{2\lambda_1}{n} \\ 0
	\end{pmatrix} $$

Moreover 
$$\int_{\Pp}-2 (\nabla u_1 \otimes \nabla u_1)  \nabla \varphi dx = -4\begin{pmatrix} \displaystyle    \int_{T^+  } (\partial_x u_1)^2 -  \frac{1}{\tan \theta}  (\partial_x u_1)(\partial_y u_1)dx\\0 \end{pmatrix}:=  \begin{pmatrix} s_1^\lambda \\ 0\end{pmatrix}.$$
Therefore $s_1^\lambda = -2\lambda_1/n$.

On the other hand, if $Q(v) = v-\left(\int_{\Pp} u_1v \right)u_1$ is the $L^2$-projection of $v$ on the orthogonal of $u_1$, we may note that
\begin{align}\label{eq:f-with-Q}
(\bo f,v)_{H^{-1}\times H^1} & = \int_{\Bbb P_n}2(\nabla u_1\odot \nabla v)\nabla \varphi dx - 2 \int_{\Bbb P_n} (\nabla u_1\otimes \nabla u_1)\nabla \varphi dx \\
& = 2\int_{\Bbb P_n} (\nabla u_1 \odot \nabla Q(v) )\nabla \varphi dx \notag
\end{align} 

Working with a symmetric triangulation for $\varphi_j$ (Figure \ref{fig:def-triangles}) and a mesh that is exact on $T_j$ and respects the symmetries of the regular polygon (Figure \ref{fig:example-regular-meshes}) the uniqueness of the first discrete eigenfunction $u_{1,h}$ implies that \eqref{eq:relation-rotation-gradients} holds also for the discrete quantities. In particular $\int_{T_+}(\partial_x u_{1,h})^2+\int_{T_+}(\partial_y u_{1,h})^2 = \lambda_{1,h}/n$ and $-\sin\theta \int_{T_+}(\partial_x u_{1,h})^2+\sin \theta \int_{T_+}(\partial_y u_{1,h})^2+2\cos\theta \int_{T_+}\partial_x u_{1,h}\partial_y u_{1,h} = 0$.
We denote by
$$  \begin{pmatrix}s_{1,h}^\lambda \\ 0\end{pmatrix} = -4  \begin{pmatrix}    \displaystyle \int_{T^+  }\Big (\frac{\partial u_{1,h}}{\partial x} \Big )^2 -  \frac{1}{\tan \theta}  \frac{\partial u_{1,h} }{\partial x} \frac{\partial u_{1,h}}{\partial y}dx\\0 \end{pmatrix} $$
and using the previous relations we find that $s_{1,h}^\lambda = -2\lambda_{1,h}/n$. Therefore $|s_1^\lambda-s_{1,h}^\lambda|=\frac{2}{n}|\lambda_1-\lambda_{1,h}|$.

Let
$\bo f_h$ be the distribution given by: $\forall v \in C_c^\infty (\R^2)$
$$ (\bo f_h,v)_{H^{-1}\times H^1} =  \int_{\Pp}(\nabla v \otimes  \nabla u_{1,h})\nabla \varphi \notag \, dx + \int_{\Pp}(\nabla u_{1,h} \otimes \nabla v)\nabla \varphi \notag \, dx+s_{1,h}^\lambda   \int_{\Pp} u_{1,h}v\, dx.
$$

\begin{prop}
The following inequality occurs
$$\|\bo f-\bo f_h\|_{H^{-1}} \le \frac{2 \sqrt{2}}{\sqrt{n} \sin \theta }\|\nabla u_{1,h}-\nabla u_1\|_{L^2} + \frac{1}{\sqrt{1+\lb_1}}\Big (|s_1^\lambda- s_{1,h}^\lambda |+|s_1^\lambda| \|u_1-u_{1,h}\|_{L^2}\Big).$$

\end{prop}

\emph{Proof:}
The proof is straight forward by direct computation, taking into account the vector norm inequality $\|(a\otimes b)c\|\le \|a\| \|b \| \|c\|$ and the Poincar\'e inequality
$$\|v\|_{L^2} \le \frac{1}{\sqrt{1+ \lb_1}} \|v\|_{H^1_0}.$$

We also used the symmetry of the mesh, which gives
$$\|\nabla u_{1,h}-\nabla u_1\|_{L^2(T_+\cup T_-)}= \sqrt{\frac 2n} \|\nabla u_{1,h}-\nabla u_1\|_{L^2(\Pp)}.$$\hfill $\square$
\smallskip

\noindent{\bf Practical estimate.} We estimate below the quantities needed for the estimates in Theorem \ref{thm:estimate-singular}.
In order to estimate $\|f^i\|_{H^{-1}}$, $\|f^i_{\reg}\|_{L^2}$  and $\|f^i_{\sing}\|_{H^{-\frac 12 -\gamma}}$ we use the following notations:
$$(\bo f,v)_{H^{-1},H_0^1}=\begin{pmatrix} (f^1,v)_{H^{-1}\times H^1_0}\\  (f^2,v)_{H^{-1}\times H^1_0}\end{pmatrix}  = \int_{\Pp} 2(\nabla u_1 \odot \nabla v)\nabla \varphi \notag \, dx+  \int_{\Pp} u_1v\, dx \begin{pmatrix} s_1^\lambda\\ 0 \end{pmatrix}$$
\[ = \int_{\Pp} (\nabla \varphi \cdot \nabla Q(v))\nabla u_1 +\int_\Pp (\nabla \varphi \cdot \nabla u_1)\nabla Q(v)\]
$$=  \int_{\Pp} (\nabla \varphi  \cdot \nabla v)\nabla u_1 \notag \, dx+ \int_{\Pp} (\nabla \varphi \cdot \nabla u_1)\nabla v \notag \, dx+\int_{\Pp} u_1v\, dx \begin{pmatrix} s_1^\lambda\\ 0 \end{pmatrix} :=\begin{pmatrix} A^1\\ A^2 \end{pmatrix}+\begin{pmatrix} B^1\\ B^2 \end{pmatrix}+\begin{pmatrix} C^1\\ C^2 \end{pmatrix}.
$$
Recall that $Q(v)$ is the projection of $v$ on the orthogonal of $u_1$. Therefore $\|\nabla Q(v)\|_{L^2} \leq \|\nabla v\|_{L^2}$ and $\|Q(v)\|_{L^2} \leq \|v\|_{L^2}$. 

For the $H^{-1}$ estimate we work with the formula involving $Q(v)$ and we have
\[ (f^1,v)_{H^{-1},H^1} = \int_{T_+\cup T_-} (2\partial_x u_1 \partial_x Q(v) + \partial_y \varphi (\partial_x u_1 \partial_y Q(v) +  \partial_y u_1 \partial_x Q(v)))\]
Which implies that
\[ \|f^1\|_{H^{-1}} \leq 2\sqrt{2} \left(\int_{T_+}(\partial_x u_1)^2\right)^{1/2}+\frac{1}{\tan \theta} \sqrt{\frac{2\lambda_1}{n}}\]
A similar computation for $f^2$ leads to
\[ \|f^2\|_{H^{-1}} \leq \frac{2\sqrt{2}}{\tan \theta}\left(\int_{T_+}(\partial_y u_1)^2\right)^{1/2} +\sqrt{\frac{2\lambda_1}{n}}.\]

Let us denote $S_+, S_0, S_-$ the segments $[0,\exp (i \theta)], [0,1], [0,\exp (-i \theta)]$  in the complex plane, and $\bo n=(n_x,n_y)$ the outside normal of a domain. We have
$$ \begin{pmatrix} (A^1,v)\\ (A^2,v) \end{pmatrix} = \begin{pmatrix} - \int_{T_+\cup T_-} v \nabla \varphi \cdot \nabla  \frac{\partial u_1}{\partial x} + \int_{\partial T_+}  v \nabla \varphi \cdot \bo n  \frac{\partial u_1}{\partial x}  +\int_{\partial T_-}  v \nabla \varphi \cdot \bo n  \frac{\partial u_1}{\partial x}  \\ - \int_{T_+\cup T_-} v \nabla \varphi \cdot \nabla  \frac{\partial u_1}{\partial y} + \int_{\partial T_+}  v \nabla \varphi \cdot \bo n  \frac{\partial u_1}{\partial y}  +\int_{\partial T_-}  v \nabla \varphi \cdot \bo n  \frac{\partial u_1}{\partial y}  \end{pmatrix}.$$
We decompose each term in $A^i=A^i_{\reg}+A^i_{\sing}$, the regular part given by the first integral over $T_+\cup T_-$ and the singular part given by the sum of the last two integrals over the boundaries of $\partial T_+$ and $\partial T_-$. Following \cite[Lemmas 3.4.1.2-3]{grisvard} we find that $\|\nabla(\partial_x u_1)\|_{L^2(T_j)} = \sqrt{\lambda_1} \|\partial_x u_1\|_{L^2(T_j)}$ and $\|\nabla(\partial_y u_1)\|_{L^2(T_j)} = \sqrt{\lambda_1} \|\partial_y u_1\|_{L^2(T_j)}$. Therefore, for the regular parts we have
\[ \|A^1_{\reg}\|_{L^2} \leq \frac{\sqrt{2\lambda_1}}{\sin \theta}  \|\partial_x u_1\|_{L^2(T_+)} \ \ \ \|A^2_{\reg}\|_{L^2} \leq \frac{\sqrt{2\lambda_1}}{\sin \theta}\|\partial_y u_1\|_{L^2(T_+)} \]
  
  We explicit $\nabla \varphi \cdot \bo n$ on $S_0,S_+,S_-$ and we obtain 
\[
 \begin{pmatrix} (A^1_{\sing},v)\\ (A^2_{\sing},v) \end{pmatrix}  
  =\begin{pmatrix}-\int_{S_+}\frac{1}{\sin\theta} v \frac{\partial u_1}{\partial x} 
 +2\int_{S_0}\frac{1}{\tan \theta}  v \frac{\partial u_1}{\partial x}
 -\int_{S_-}\frac{1}{\sin\theta}  v   \frac{\partial u_1}{\partial x} \\ 
 -\int_{S_+}\frac{1}{\sin\theta}  v \frac{\partial u_1}{\partial y} 
 -\int_{S_-}\frac{1}{\sin\theta}  v   \frac{\partial u_1}{\partial y}  
  \end{pmatrix}
\]
We have $\|\partial_x u_1\|_{L^2(S_\pm)}  = \cos \theta \|\partial_x u_1\|_{L^2(S_0)}$ and $\|\partial_y u_1\|_{L^2(S_\pm)}  = \sin \theta \|\partial_x u_1\|_{L^2(S_0)}$ since the normal component of the gradient of $u_1$ is zero on these segments. Therefore, using Remark \ref{rem:c-gamma} we obtain
$$   \|A^1_{\sing}\|_{H^{-\frac 12-\gamma}} \le  2  \frac{1+ \cos \theta}{ \sin \theta} \sqrt{ \lb_1 } \|u_1\|_{L^2(S_0)}C_\gamma, \ \ \  \|A^2_{\sing}\|_{H^{-\frac 12-\gamma}} \le 2   \sqrt{ \lb_1  } \|u_1 \|_{L^2(S_0)}C_\gamma.$$
A similar computation leads to
$$  \begin{pmatrix} (B^1,v)\\ (B^2,v) \end{pmatrix} = \begin{pmatrix} - \int_{T_+\cup T_-} v \nabla \varphi \cdot  \nabla  \frac{\partial u_1 }{\partial x} + \int_{\partial T_+}  v \nabla \varphi \cdot \nabla u_1   n_x  +\int_{\partial T_-}  v \nabla \varphi \cdot \nabla u_1 n_x  \ \\ - \int_{T_+\cup T_-} v \nabla \varphi \cdot \nabla  \frac{\partial u_1 }{\partial y} +\int_{\partial T_+}  v \nabla \varphi \cdot \nabla u_1   n_y  +\int_{\partial T_-}  v \nabla \varphi \cdot \nabla u_1 n_y   \end{pmatrix},$$
and we observe that the boundary integrals vanish. For the regular parts we have the same estimates as before
\[ \|B^1_{\reg}\|_{L^2} \leq \frac{\sqrt{2\lambda_1}}{\sin \theta}  \|\partial_x u_1\|_{L^2(T_+)} \ \ \ \|B^2_{\reg}\|_{L^2} \leq \frac{\sqrt{2\lambda_1}}{\sin \theta}\|\partial_y u_1\|_{L^2(T_+)}.\]
It is straightforward to see that $C^1,C^2$ are $L^2$ distributions, $C^2 = 0$ and $\|C^1\|_{L^2} = |s_1^\lambda| = 2\lambda_1/n$.

Finally, we get
\[ \|f^1_{\reg}\|_{L^2} \leq 2\frac{\sqrt{2\lambda_1}}{\sin \theta}  \|\partial_x u_1\|_{L^2(T_+)}+\frac{2\lambda_1}{n} \ \ \ \|f^2_{\reg}\|_{L^2} \leq 2\frac{\sqrt{2\lambda_1}}{\sin \theta}\|\partial_y u_1\|_{L^2(T_+)}.\]
$$   \|f^1_{\sing}\|_{H^{-\frac 12-\gamma}} \le  2  \frac{1+ \cos \theta}{ \sin \theta} \sqrt{ \lb_1 } \|u_1\|_{L^2(S_0)}C_\gamma, \ \ \  \|f^2_{\sing}\|_{H^{-\frac 12-\gamma}} \le 2   \sqrt{ \lb_1  } \|u_1 \|_{L^2(S_0)}C_\gamma.$$

Using the fact that $\|\partial_x u_1\|_{T_+}^2+ \|\partial_y u_1\|_{T_+}^2 = \lambda_1/n$ we may also use the slightly weaker, but simpler bounds below:
\[\|f ^i\|_{H^{-1}} \leq \frac{2 }{\sin\theta } \sqrt{\frac{2\lambda_1}{n}},\ \ 
 \|f^1_{\reg}\|_{L^2} \leq \frac{2\lambda_1}{\sin \theta}\sqrt{\frac{2}{n}}+\frac{2\lambda_1}{n} \ \ \ \|f^2_{\reg}\|_{L^2} \leq \frac{2\lambda_1}{\sin \theta}\sqrt{\frac{2}{n}}.\]

\subsection{Step 3.  Estimates for the eigenvalues of $\bo M^\lb$.}
\label{sec:estimates-coeffs}
 As shown in Theorem \ref{thm:explicit-eigenvalues} and Proposition \ref{prop:detailed-computations} the eigenvalues of $\bo M^\lambda$ can be expressed in terms of $u_1$ and $(U_0^1,U_0^2)$. As we saw in the previous sections, the terms containing derivatives of $u_1$ can be well approximated using $\bf P_1$ finite elements using an estimate of order $O(h)$ with explicit constants. 

Results of the previous section show that the estimate of the computation error for $\bo U$ behaves like $h^{\frac 12 -\gamma}$. Trying to bound directly the error for the eigenvalues of $\bo M^\lambda$ will give estimates of the same order, which in practice are not fine enough to provide bounds that allow to certify that the non-zero eigenvalues of $\bo M^\lambda$ are positive. 

However, it turns out that the estimate of the coefficients of the shape Hessian matrix of the eigenvalue is better, namely in $h^{1-2\gamma}$, as a consequence of the particular structure of the coefficients. As shown in \cite[Section 5]{approximation-trick} defining and solving an auxiliary problem using the same bilinear form can double the speed of the convergence.

We use the notations of Theorem \ref{thm:estimate-singular} for two generic problems with solutions $U^\aaa, U^\bbb$ corresponding to the right hand sides $f^\aaa, f^\bbb$ (not necessarily those explicited in the previous section). As well, we use the associated notations $V^\aaa, V^\bbb$,  $\tilde V^\aaa, \tilde V^\bbb$, $U_h^\aaa, U_h^\bbb$, $f_h^\aaa, f_h^\bbb$. We denote the bilinear forms
$$a: H^1_0(\Bbb P_n) \times H^1_0(\Bbb P_n) \ra \R,\;\; a(u,v)= \int \nabla u \cdot \nabla v -\lb_1 \int uv,$$
$$a_h: {\mathcal V}^h \times{\mathcal V}^h \ra \R, \;\; a_h(u,v)= \int \nabla u \cdot \nabla v -\lb_{1,h} \int uv.$$

Our objective is to estimate error terms of the type
$$|a(U^\aaa,U^\bbb)-a_h(U_h^\aaa, U_h^\bbb)|,$$
in order to get an estimate of order $h^{1-2\gamma}$ for $\alpha_k,\beta_k,\gamma_k$ in Theorem \ref{thm:explicit-eigenvalues}.
We have
\begin{align}
\label{eq:estimate-a(U,V)}
|a(U^\aaa,U^\bbb)-a_h(U_h^\aaa, U_h^\bbb)|\le & \\ \notag 
|a(U^\aaa,U^\bbb)-a(V^\aaa, V^\bbb)|+& |a(V^\aaa,V^\bbb)-a_h(\tilde V^\aaa, \tilde V^\bbb)|+ |a_h(\tilde V^\aaa, \tilde V^\bbb)-a_h(U_h^\aaa, U_h^\bbb)|.
\end{align}
We estimate each term of the right hand side, separately, the most delicate being the first one.

\noindent{\bf First term.}
$$a(U^a,U^b)-a(V^a, V^b)= \int_\Pp \nabla (U^a-V^a)\cdot \nabla (U^\bbb-V^\bbb) - \lb_1 \int_\Pp (U^\aaa-V^\aaa) V^\bbb- \lb_1 \int_\Pp (U^\bbb-V^\bbb) U^\aaa, $$
so that
$$|a(U^\aaa,U^\bbb)-a(V^\aaa, V^\bbb)|\le \hskip 11cm $$
$$  \| \nabla (U^\aaa-V^\aaa)\|_{L^2} \| \nabla (U^\bbb-V^\bbb)\|_{L^2} +  \lb_1\|V^\bbb\|_{L^2} \|V^\aaa-U^\aaa\|_{L^2}+  \lb_1\|U^\aaa\|_{L^2} \|V^\bbb-U^\bbb\|_{L^2}.$$ 
As a consequence of Lemma \ref{bobu36.l} applied for the  $L^2$-norms of  both  the functions and their gradients we get a control in $h^{1-2\gamma}$. 

\noindent{\bf Second term.}
$$|a(V^\aaa,V^\bbb)-a_h(\tilde V^\aaa, \tilde V^\bbb)| \le \|\nabla V^\aaa\|_{L^2} \|\nabla V^\bbb-\nabla \tilde V^\bbb\|_{L^2}+ \|\nabla \tilde V^\bbb\|_{L^2} \|\nabla V^\aaa-\nabla \tilde V^\aaa\|_{L^2}+  $$
$$\hskip 2.5cm |\lb_{1,h}-\lb_1| \|\tilde V^\aaa\|_{L^2}\|\tilde V^\bbb\|_{L^2}+ \lb_1 \|\tilde V^\bbb\|_{L^2} \| V^\aaa- \tilde V^\aaa\|_{L^2}+\lb_1 \|V^\aaa\|_{L^2} \| V^\bbb-\tilde  V^\bbb\|_{L^2},$$
which, in view of inequality \eqref{bobu80}, leads to an approximation of order $h$.

\noindent{\bf Third term.}
$$|a_h(\tilde V^\aaa, \tilde V^\bbb)-a_h(U_h^\aaa, U_h^\bbb)|\le |a_h(\tilde V^\aaa, \tilde V^\bbb-U^\bbb_h)|+ |a_h(\tilde V^\aaa-U^\aaa_h, U_h^\bbb)|\le $$
$$ \|\nabla \tilde V^\aaa\|_{L^2} \|\nabla \tilde V^\bbb-\nabla U^\bbb_h\|_{L^2}+ \|\nabla  U^\bbb_h\|_{L^2} \|\nabla \tilde V^\aaa-\nabla U^\aaa_h\|_{L^2}.$$
The last inequality is a consequence of the fact that $a_h(\cdot, \cdot)$ is a scalar product on $\{u_{1,h}\}^\perp$ in ${\mathcal V}^h$ and of the Cauchy-Schwarz inequality together with the observation that $a_h(v,v) \le \int |\nabla v|^2$.
Using inequality \eqref{bobu81} we get an approximation of order $h$.

\begin{rem}\rm The problematic term in the previous estimates can be simplified when the two distributions and associated solutions have opposite parity properties. Indeed, suppose that $f^1 = f^\aaa_{\reg}+f^\aaa_{\sing}$ with $f^\aaa_{\reg} \in L^2, f^\aaa_{\sing} \in H^{-\frac{1}{2}-\gamma}$ such that $(f^\aaa_{\sing},U^\bbb-V^\bbb)_{H^{-\frac{1}{2}-\gamma},H^{\frac{1}{2}+\gamma}} = 0$. Then we have
\begin{align*}
&\int_\Pp \nabla (U^\aaa-V^\aaa)\cdot \nabla (U^\bbb-V^\bbb)  = \int_\Pp \nabla U^\aaa \cdot \nabla (U^\bbb-V^\bbb) \\
 = & (\lambda_1U^\aaa+f^\aaa_{\reg}+f^\aaa_{\sing},U^\bbb-V^\bbb)_{H^{-1},H^1} = ( (\lambda_1U^\aaa+f^\aaa_{\reg},U^\bbb-V^\bbb)_{L^2,L^2},
\end{align*}
leading to an estimate of order $h^{3/2-\gamma}$, for $\gamma \in (0,0.5)$.
\label{rem:zero-singular-part}
\end{rem}

Below we show how to choose the functions in the above estimates in order to obtain the desired bounds for the quantities described in Theorem \ref{thm:explicit-eigenvalues}. Since in the case $k=0$ we have $\alpha_0=\beta_0=\gamma_0=0$ we focus only on the cases $1\leq k \leq n-1$.

\begin{rem}
	\rm
	It can be noted that the error estimates above can already be applied for terms of the type $a(U_j^{1,2},U_l^{1,2})$ that appear in the expressions of $\bo M^\lambda$ and $\alpha_k,\beta_k,\gamma_k$. However, if multiple such terms are present in some expression, a direct error estimate will accumulate the errors and the final results will be unusable for reasonably large $h$. It is best to choose properly the functions $U^a,U^b$ beforehand and apply the error estimate only once. 
	\label{rem:why-choose-different-functions}
\end{rem}

{\bf The term $\alpha_k$.} Denote by 
\[W^{\alpha_k} = \sum_{j=0}^{n-1} \cos(jk\theta) U_0^1\circ \bo R_{j\theta}^T\]
so that $a(U_0^1,W^{\alpha_k})$ allows us to express $\alpha_k$ (see Proposition \ref{prop:detailed-computations}). The orthogonality of $U_0^1$ on $u_1$ implies that $\int_\Pp W^{\alpha_k}u_1=0$. Denote $f^{\alpha_k} \in H^{-1}$ the distribution
\begin{align*}
(f^{\alpha_k},v)_{H^{-1},H_0^1} 
&= \sum_{j=0}^{n-1} (\cos (j+1)k\theta+\cos jk\theta) \lambda_1\int_{T_j} u_1 v\\ 
&+ \sum_{j=0}^{n-1} \frac{\cos(j+1)k\theta-\cos jk\theta}{\sin \theta} \int_{T_j}\begin{pmatrix}
-\sin (2j+1)\theta & \cos (2j+1) \theta \\
\cos(2j+1)\theta & \sin (2j+1)\theta
\end{pmatrix}\nabla u_1 \cdot \nabla v. 
\end{align*}
Multiplying a vector with the matrix $\begin{pmatrix}
-\sin (2j+1)\theta & \cos (2j+1) \theta \\
\cos(2j+1)\theta & \sin (2j+1)\theta
\end{pmatrix}$ preserves its length and reflects it about the line through the origin making an angle $(j+1/2)\theta+\pi/4$. The symmetry of the first eigenfunction $u_1$ implies that $(f^{\alpha_k},u_1) = 0$. These observations imply that $W^{\alpha_k}$ is the unique solution of the problem
\[ a(W,v) = (f^{\alpha_k},v)_{H^{-1},H_0^1} ,\ \forall v \in H_0^1(\Pp) , \int_\Pp Wu_1=0.\]

Elementary computations show that
{\small 
\begin{align}\label{eq:sum-cosines}
\sum_{j=0}^{n-1} (\cos jk\theta+\cos(j+1)k\theta))^2 = n+n\cos (k\theta),& \sum_{j=0}^{n-1} (\cos jk\theta-\cos(j+1)k\theta))^2 = n-n\cos (k\theta).\\ \notag
\sum_{j=0}^{n-1} (\sin jk\theta+\sin(j+1)k\theta))^2 = n+n\cos (k\theta),& \sum_{j=0}^{n-1} (\sin jk\theta-\sin(j+1)k\theta))^2 = n-n\cos (k\theta).
\end{align}}
Therefore, a straightforward estimate using \eqref{eq:sum-cosines}, the symmetry of the eigenfunction $u_1$ and $\|v\|_{L^2} \leq \frac{1}{\sqrt{1+\lambda_1}}\|v\|_{H^1}$ shows that
\[ \|f^{\alpha_k}\|_{H^{-1}} \leq \lambda_1 \sqrt{\frac{1+\cos(k\theta)}{1+\lambda_1}} +\frac{\sqrt{\lambda_1(1-\cos(k\theta))}}{\sin \theta}.\]

For simplicity denote $K_j^{\alpha_k} = \frac{\cos(j+1)k\theta-\cos jk\theta}{\sin \theta} \begin{pmatrix}
-\sin (2j+1)\theta & \cos (2j+1) \theta \\
\cos(2j+1)\theta & \sin (2j+1)\theta
\end{pmatrix}$
Then we have
\begin{align*}
\sum_{j=0}^{n-1} \int_{T_j} K_j^{\alpha_k}\nabla u_1 \cdot \nabla v & = \sum_{j=0}^{n-1} \int_{T_j} -\di (K_j^{\alpha_k}\nabla u_1) v + \sum_{j=1}^{n-1} \int_{\partial T_j} (K_j^{\alpha_k}\nabla u_1 \cdot n) v \\
\end{align*}
Since $u_1 \in H^2(\Pp)$ the first term is regular. Let us investigate the second term. Denote with $S_j, S_{j+1}$ the two rays associated to the triangle $T_j$, $j=0,...,n-1$ (with notation modulo $n$). Denote with $N_j=\begin{pmatrix}
-\sin j\theta \\ \cos j\theta
\end{pmatrix}$ the normal to $S_j$ in the trigonometric sense.
The symmetry of the eigenfunction (see Remark \ref{rem:symm-eigenfunction}) implies that $(\nabla u_1)_{S_j} = \partial_r u_1 \begin{pmatrix}
	\cos j\theta \\ \sin j\theta
	\end{pmatrix}$, which implies
\[ \begin{pmatrix}
-\sin(2j+1)\theta & \cos(2j+1)\theta \\
\cos(2j+1)\theta & \sin (2j+1)\theta 
\end{pmatrix}
\nabla u \cdot N_j = \partial_r u_1 \cos \theta.\]
 We obtain for $v \in H_0^1(\Pp)$
\begin{align*}
\sum_{j=0}^{n-1} \int_{\partial T_j} (K_j^{\alpha_k}\nabla u_1 \cdot n) v & = \sum_{j=0}^{n-1}\left(- \int_{S_j} (K_j^{\alpha_k}\nabla u_1 \cdot N_j) v
+\int_{S_{j+1}} (K_j^{\alpha_k}\nabla u_1 \cdot N_{j+1})v
\right)\\
& = \sum_{j=0}^{n-1} \int_{S_j}( (K_{j-1}^{\alpha_k}-K_j^{\alpha_k})\nabla u_1 \cdot N_j) v\\
& =  -\sum_{j=0}^{n-1}\int_{S_j} \frac{\cos \theta}{\sin \theta} (\cos (j+1)k\theta+\cos (j-1)k\theta - 2 \cos jk\theta) \partial_r u_1 v.\\
& = -\sum_{j=0}^{n-1}\int_{S_j} \frac{\cos \theta}{\sin \theta} 2 \cos jk\theta(1-\cos k \theta) \partial_r u_1 v
\end{align*}
Finally 
$$(f_{\sing}^{\alpha_k}, v)_{H^{-1}\times H^1_0} =  -\sum_{j=0}^{n-1}\int_{S_j} \frac{\cos \theta}{\sin \theta} 2 \cos jk\theta(1-\cos k \theta) \partial_r u_1 v.
$$
which, using Remark \ref{rem:c-gamma}, gives
$$\|f_{\sing}^{\alpha_k}\|_{H^{-\frac 12-\gamma}}\le 2\sum_{j=0}^{n-1}   \frac{\cos \theta}{\sin \theta}  | \cos jk\theta(1-\cos k \theta)| \sqrt{\lb_1} \|u_1\|_{L^2(S_0)} C_\gamma.$$

For the regular part, we have
$$(f_{\reg}^{\alpha_k}, v)_{H^{-1}\times H^1_0}= \sum_{j=0}^{n-1} (\cos (j+1)k\theta+\cos jk\theta) \lb_1 \int_{T_j} u_1v  $$
$$ - \sum_{j=0}^{n-1} \frac{\cos(j+1)k\theta-\cos jk\theta}{\sin \theta}    \int_{T_j} \big (-\sin (2j+1)\theta \partial^2_{xx}u_1 + 2 \cos (2j+1)\theta \partial ^2_{xy} u_1+ \sin(2j+1)\theta \partial ^2_{yy} u_1\big )v $$
and using the fact that $\|D^2 u_1\|_{L^2} = \lambda_1$ we obtain
$$\|f^{\alpha_k}_{\reg}\|_{L^2}\le \sqrt{1+\cos (k \theta)}  \lb_1   +\frac{\sqrt{2}}{\sin  \theta} \sqrt{1- \cos (k \theta)}\lb_1 .$$

\begin{rem} 
	\rm Let us introduce the vectors
$$v_j=(\cos((j+1/2)\theta), \sin ((j+1/2)\theta)), \overline v_j= (-\sin((j+1/2)\theta), \cos ((j+1/2)\theta)).$$
Expressing the derivatives of $u_1$ in  the $(v_j, \overline v_j$) basis, by direct computation one gets
$$(f_{\reg}^{\alpha_k}, v)_{H^{-1}\times H^1_0}= \sum_{j=0}^{n-1} (\cos (j+1)k\theta+\cos jk\theta) \lb_1 \int_{T_j} u_1v-2 \frac{\cos(j+1)k\theta-\cos jk\theta}{\sin \theta} \sum_{j=0}^{n-1} \int_{T_j} \partial ^2_{v_j\overline v_j}u_1 v.$$
\label{rem:reg-part-change-basis}
\end{rem}

Consider now the discrete version of $f^{\alpha_k}$, replacing $u_1$ and $\lambda_1$ by their discrete approximations
	\begin{align*}
	(f_h^{\alpha_k},v)_{H^{-1},H_0^1} & = \sum_{j=0}^{n-1} (\cos (j+1)k\theta+\cos jk\theta) \lambda_{1,h} \int_{T_j} u_{1,h}v\\
	& + \sum_{j=0}^{n-1} \frac{\cos(j+1)k\theta-\cos jk\theta}{\sin \theta} \int_{T_j}\begin{pmatrix}
	-\sin (2j+1)\theta & \cos (2j+1) \theta \\
	\cos(2j+1)\theta & \sin (2j+1)\theta
	\end{pmatrix}\nabla u_{1,h} \cdot \nabla v
	\end{align*}
Working under the hypothesis that the mesh $\mathcal T^h$ has the symmetries of the regular polygon and that the triangles $T_j$ are meshed exactly, we have $(f_h^{\alpha_k},u_{1,h})_{H^{-1},H_0^1} = 0$. By direct computation we obtain
{\small
\[
(f^{\alpha_k}-f^{\alpha_k}_h,v)  = \sum_{j=0}^{n-1} (\cos (j+1)k\theta+\cos jk\theta) \int_{T_j}(\lambda_1 u_1-\lambda_{1,h}u_{1,h})v + \sum_{j=0}^{n-1}  \int_{T_j}K_j^{\alpha_k}(\nabla u - \nabla u_h)\cdot \nabla v 
\]
}
which implies
\begin{align*}
\|f^{\alpha_k}-f_h^{\alpha_k} \|_{H^{-1}} & \leq \sqrt{\frac{1+\cos k\theta}{1+\lambda_1}}(|\lambda_1-\lambda_{1,h}|+\lambda_{1,h}\|u_1-u_{1,h}\|_{L^2}) \\
& + \frac{\sqrt{1-\cos k\theta}}{\sin \theta} \|\nabla u-\nabla u_h\|_{L^2}
\end{align*}

{\bf The term $\beta_k$.} Denote by 
\[W^{\beta_k} = \sum_{j=0}^{n-1} \cos(jk\theta) U_0^2\circ \bo R_{j\theta}^T\]
so that $a(U_0^2,W^{\beta_k})$ allows us to express $\beta_k$ (see Proposition \ref{prop:detailed-computations}). The orthogonality of $U_0^2$ on $u_1$ implies that $\int_\Pp W^{\beta_k}u_1=0$. Denote $f^{\beta_k} \in H^{-1}$ the distribution
\begin{align*}
(f^{\beta_k},v)_{H^{-1},H_0^1} 
&= \frac{\cos\theta}{\sin\theta}\sum_{j=0}^{n-1} (\cos(j+1)k\theta-\cos jk\theta) \lambda_1\int_{T_j} u_1 v\\ 
&+ \sum_{j=0}^{n-1} \frac{\cos(j+1)k\theta-\cos jk\theta}{\sin \theta} \int_{T_j}\begin{pmatrix}
-\cos (2j+1)\theta & -\sin (2j+1) \theta \\
-\sin(2j+1)\theta & \cos (2j+1)\theta
\end{pmatrix}\nabla u_1 \cdot \nabla v. 
\end{align*}
 Same as before, the symmetry of the first eigenfunction $u_1$ implies that $(f^{\beta_k},u_1) = 0$. These observations imply that $W^{\beta_k}$ is the unique solution of the problem
\[ a(W,v) = (f^{\beta_k},v)_{H^{-1},H_0^1} ,\ \forall v \in H_0^1(\Pp) , \int_\Pp Wu_1=0.\]

A straightforward estimate shows that
\[ \|f^{\beta_k}\|_{H^{-1}} \leq \lambda_1\frac{\cos\theta}{\sin \theta} \sqrt{\frac{1-\cos(k\theta)}{1+\lambda_1}} +\frac{\sqrt{\lambda_1(1-\cos(k\theta))}}{\sin \theta}.\]
Similar computations as before give 
$$(f_{\sing}^{\beta_k}, v)_{H^{-1}\times H^1_0} =  \sum_{j=0}^{n-1}\int_{S_j} 2 \cos jk\theta(1-\cos k \theta) \partial_r u_1 v.
$$
which, using Remark \ref{rem:c-gamma}, gives
$$\|f_{\sing}^{\beta_k}\|_{H^{-\frac 12-\gamma}}\le 2\sum_{j=0}^{n-1}  | \cos jk\theta(1-\cos k \theta)| \sqrt{\lb_1} \|u_1\|_{L^2(S_0)} C_\gamma.$$

For the regular part, we have
$$(f_{\reg}^{\beta_k}, v)_{H^{-1}\times H^1_0}= \frac{\cos \theta}{\sin \theta}\sum_{j=0}^{n-1} (\cos (j+1)k\theta-\cos jk\theta) \lb_1 \int_{T_j} u_1v  $$
$$ - \sum_{j=0}^{n-1} \frac{\cos(j+1)k\theta-\cos jk\theta}{\sin \theta}    \int_{T_j} \big (-\cos (2j+1)\theta \partial^2_{xx}u_1 -2 \sin (2j+1)\theta \partial ^2_{xy} u_1+ \cos(2j+1)\theta \partial ^2_{yy} u_1\big )v $$
and using the fact that $\|D^2 u_1\|_{L^2} = \lambda_1$ we obtain
$$\|f_{\reg}^{\beta_k}\|_{L^2}\le \frac{\cos\theta}{\sin \theta}\sqrt{1-\cos (k \theta)}  \lb_1   +\frac{\sqrt{2}}{\sin  \theta} \sqrt{1- \cos (k \theta)}\lb_1 .$$

Consider now the discrete version of $f^{\beta_k}$, replacing $u_1$ and $\lambda_1$ by their discrete approximations
\begin{align*}
(f_h^{\beta_k},v)_{H^{-1},H_0^1} & = \frac{\cos \theta}{\sin \theta}\sum_{j=0}^{n-1} (\cos(j+1)k\theta-\cos jk\theta) \lambda_{1,h} \int_{T_j} u_{1,h}v\\
& + \sum_{j=0}^{n-1} \frac{\cos(j+1)k\theta-\cos jk\theta}{\sin \theta} \int_{T_j}\begin{pmatrix}
-\cos (2j+1)\theta & -\sin (2j+1) \theta \\
-\sin (2j+1)\theta & \cos (2j+1)\theta
\end{pmatrix}\nabla u_{1,h} \cdot \nabla v
\end{align*}
Working under the hypothesis that the mesh $\mathcal T^h$ has the symmetries of the regular polygon and that the triangles $T_j$ are meshed exactly, we have $(f_h^{\beta_k},u_{1,h})_{H^{-1},H_0^1} = 0$. Below we use the notation $K_j^{\beta_k} = \frac{\cos(j+1)k\theta-\cos jk\theta}{\sin \theta} \begin{pmatrix}
-\cos (2j+1)\theta & -\sin (2j+1) \theta \\
-\sin(2j+1)\theta & \cos (2j+1)\theta
\end{pmatrix}$ By direct computation we obtain
{\small
	\[
	(f^{\beta_k}-f^{\beta_k}_h,v)  = \frac{\cos \theta}{\sin \theta}\sum_{j=0}^{n-1} (\cos (j+1)k\theta-\cos jk\theta) \int_{T_j}(\lambda_1 u_1-\lambda_{1,h}u_{1,h})v + \sum_{j=0}^{n-1}  \int_{T_j}K_j^{\beta_k}(\nabla u - \nabla u_h)\cdot \nabla v 
	\]
}
which implies
\begin{align*}
\|f^{\beta_k}-f_h^{\beta_k} \|_{H^{-1}} & \leq \frac{\cos \theta}{\sin \theta} \sqrt{\frac{1-\cos k\theta}{1+\lambda_1}}(|\lambda_1-\lambda_{1,h}|+\lambda_{1,h}\|u_1-u_{1,h}\|_{L^2}) \\
& + \frac{\sqrt{1-\cos k\theta}}{\sin \theta} \|\nabla u-\nabla u_h\|_{L^2}
\end{align*}

{\bf The term $\gamma_k$.} In this case we have two possible formulae. We provide the details for both of them. Denote by 
\[W^{\gamma_k,1} = \sum_{j=0}^{n-1} \sin(jk\theta) U_0^2\circ \bo R_{j\theta}^T\]
so that $a(U_0^1,W^{\gamma_k,1})$ allows us to express $\gamma_k$ (see Proposition \ref{prop:detailed-computations}). The orthogonality of $U_0^2$ on $u_1$ implies that $\int_\Pp W^{\gamma_k,1}u_1=0$. Denote $f^{\gamma_k,1} \in H^{-1}$ the distribution
\begin{align*}
(f^{\gamma_k,1},v)_{H^{-1},H_0^1} 
&= \frac{\cos\theta}{\sin\theta}\sum_{j=0}^{n-1} (\sin(j+1)k\theta-\sin jk\theta) \lambda_1\int_{T_j} u_1 v\\ 
&+ \sum_{j=0}^{n-1} \frac{\sin(j+1)k\theta-\sin jk\theta}{\sin \theta} \int_{T_j}\begin{pmatrix}
-\cos (2j+1)\theta & -\sin (2j+1) \theta \\
-\sin(2j+1)\theta & \cos (2j+1)\theta
\end{pmatrix}\nabla u_1 \cdot \nabla v. 
\end{align*}
Same as before, the symmetry of the first eigenfunction $u_1$ implies that $(f^{\gamma_k,1},u_1) = 0$. These observations imply that $W^{\gamma_k,1}$ is the unique solution of the problem
\[ a(W,v) = (f^{\gamma_k,1},v)_{H^{-1},H_0^1} ,\ \forall v \in H_0^1(\Pp) , \int_\Pp Wu_1=0.\]

A straightforward estimate shows that
\[ \|f^{\gamma_k,1}\|_{H^{-1}} \leq \lambda_1\frac{\cos\theta}{\sin \theta} \sqrt{\frac{1-\cos(k\theta)}{1+\lambda_1}} +\frac{\sqrt{\lambda_1(1-\cos(k\theta))}}{\sin \theta}.\]
Similar computations as before give 
$$(f_{\sing}^{\gamma_k,1}, v)_{H^{-1}\times H^1_0} = -  \sum_{j=0}^{n-1}\int_{S_j} 2 \sin jk\theta(1-\cos k \theta) \partial_r u_1 v.
$$
which, using Remark \ref{rem:c-gamma}, gives
$$\|f_{\sing}^{\gamma_k,1}\|_{H^{-\frac 12-\gamma}}\le 2\sum_{j=0}^{n-1}  | \sin jk\theta(1-\cos k \theta)| \sqrt{\lb_1} \|u_1\|_{L^2(S_0)} C_\gamma.$$

For the regular part, we have
$$(f_{\reg}^{\gamma_k,1}, v)_{H^{-1}\times H^1_0}= \frac{\cos \theta}{\sin \theta}\sum_{j=0}^{n-1} (\sin (j+1)k\theta-\sin jk\theta) \lb_1 \int_{T_j} u_1v  $$
$$ - \sum_{j=0}^{n-1} \frac{\sin(j+1)k\theta-\sin jk\theta}{\sin \theta}    \int_{T_j} \big (-\cos (2j+1)\theta \partial^2_{xx}u_1 -2 \sin (2j+1)\theta \partial ^2_{xy} u_1+ \cos(2j+1)\theta \partial ^2_{yy} u_1\big )v $$
and using the fact that $\|D^2 u_1\|_{L^2} = \lambda_1$ we obtain
$$\|f_{\reg}^{\gamma_k,1}\|_{L^2}\le \frac{\cos\theta}{\sin \theta}\sqrt{1-\cos (k \theta)}  \lb_1   +\frac{\sqrt{2}}{\sin  \theta} \sqrt{1- \cos (k \theta)}\lb_1 .$$

Consider now the discrete version of $f^{\gamma_k,1}$, replacing $u_1$ and $\lambda_1$ by their discrete approximations
\begin{align*}
(f_h^{\gamma_k,1},v)_{H^{-1},H_0^1} & = \frac{\cos \theta}{\sin \theta}\sum_{j=0}^{n-1} (\sin(j+1)k\theta-\sin jk\theta) \lambda_{1,h} \int_{T_j} u_{1,h}v\\
& + \sum_{j=0}^{n-1} \frac{\sin(j+1)k\theta-\sin jk\theta}{\sin \theta} \int_{T_j}\begin{pmatrix}
-\cos (2j+1)\theta & -\sin (2j+1) \theta \\
-\sin (2j+1)\theta & \cos (2j+1)\theta
\end{pmatrix}\nabla u_{1,h} \cdot \nabla v
\end{align*}
Working under the hypothesis that the mesh $\mathcal T^h$ has the symmetries of the regular polygon and that the triangles $T_j$ are meshed exactly, we have $(f_h^{\gamma_k,1},u_{1,h})_{H^{-1},H_0^1} = 0$. Below we use the notation $K_j^{\gamma_k,1} = \frac{\sin(j+1)k\theta-\sin jk\theta}{\sin \theta} \begin{pmatrix}
-\cos (2j+1)\theta & -\sin (2j+1) \theta \\
-\sin(2j+1)\theta & \cos (2j+1)\theta
\end{pmatrix}$ By direct computation we obtain
{\small
	\[
	(f^{\gamma_k,1}-f^{\gamma_k,1}_h,v)  = \frac{\cos \theta}{\sin \theta}\sum_{j=0}^{n-1} (\sin (j+1)k\theta-\sin jk\theta) \int_{T_j}(\lambda_1 u_1-\lambda_{1,h}u_{1,h})v + \sum_{j=0}^{n-1}  \int_{T_j}K_j^{\gamma_k,1}(\nabla u - \nabla u_h)\cdot \nabla v 
	\]
}
which implies
\begin{align*}
\|f^{\gamma_k,1}-f_h^{\gamma_k,1} \|_{H^{-1}} & \leq \frac{\cos \theta}{\sin \theta} \sqrt{\frac{1-\cos k\theta}{1+\lambda_1}}(|\lambda_1-\lambda_{1,h}|+\lambda_{1,h}\|u_1-u_{1,h}\|_{L^2}) \\
& + \frac{\sqrt{1-\cos k\theta}}{\sin \theta} \|\nabla u-\nabla u_h\|_{L^2}.
\end{align*}

For the second formula for $\gamma_k$, denote by 
\[W^{\gamma_k,2} = \sum_{j=0}^{n-1} \sin(jk\theta) U_0^1\circ \bo R_{j\theta}^T\]
so that $a(U_0^2,W^{\gamma_k,2})$ allows us to express $\gamma_k$ (see Proposition \ref{prop:detailed-computations}). The orthogonality of $U_0^1$ on $u_1$ implies that $\int_\Pp W^{\gamma_k,2}u_1=0$. Denote $f^{\gamma_k,2} \in H^{-1}$ the distribution
\begin{align*}
(f^{\gamma_k,2},v)_{H^{-1},H_0^1} 
&= \sum_{j=0}^{n-1} (\sin(j+1)k\theta+\sin jk\theta) \lambda_1\int_{T_j} u_1 v\\ 
&+ \sum_{j=0}^{n-1} \frac{\sin(j+1)k\theta-\sin jk\theta}{\sin \theta} \int_{T_j}\begin{pmatrix}
-\sin (2j+1)\theta & \cos (2j+1) \theta \\
\cos(2j+1)\theta & \sin (2j+1)\theta
\end{pmatrix}\nabla u_1 \cdot \nabla v. 
\end{align*}
Same as before, the symmetry of the first eigenfunction $u_1$ implies that $(f^{\gamma_k,2},u_1) = 0$. These observations imply that $W^{\gamma_k,2}$ is the unique solution of the problem
\[ a(W,v) = (f^{\gamma_k,2},v)_{H^{-1},H_0^1} ,\ \forall v \in H_0^1(\Pp) , \int_\Pp Wu_1=0.\]

A straightforward estimate shows that
\[ \|f^{\gamma_k,2}\|_{H^{-1}} \leq \lambda_1 \sqrt{\frac{1+\cos(k\theta)}{1+\lambda_1}} +\frac{\sqrt{\lambda_1(1-\cos(k\theta))}}{\sin \theta}.\]
Similar computations as before give 
$$(f_{\sing}^{\gamma_k,2}, v)_{H^{-1}\times H^1_0} = - \frac{\cos \theta}{\sin \theta} \sum_{j=0}^{n-1}\int_{S_j} 2 \sin jk\theta(1-\cos k \theta) \partial_r u_1 v.
$$
which, using Remark \ref{rem:c-gamma}, gives
$$\|f_{\sing}^{\gamma_k,2}\|_{H^{-\frac 12-\gamma}}\le 2\frac{\cos \theta}{\sin \theta}\sum_{j=0}^{n-1}  | \sin jk\theta(1-\cos k \theta)| \sqrt{\lb_1} \|u_1\|_{L^2(S_0)} C_\gamma.$$

For the regular part, we have
$$(f_{\reg}^{\gamma_k,2}, v)_{H^{-1}\times H^1_0}= \sum_{j=0}^{n-1} (\sin (j+1)k\theta+\sin jk\theta) \lb_1 \int_{T_j} u_1v  $$
$$ - \sum_{j=0}^{n-1} \frac{\sin(j+1)k\theta-\sin jk\theta}{\sin \theta}    \int_{T_j} \big (-\sin (2j+1)\theta \partial^2_{xx}u_1 +2 \cos (2j+1)\theta \partial ^2_{xy} u_1+ \sin(2j+1)\theta \partial ^2_{yy} u_1\big )v $$
and using the fact that $\|D^2 u_1\|_{L^2} = \lambda_1$ we obtain
$$\|f_{\reg}^{\gamma_k,2}\|_{L^2}\le \sqrt{1+\cos (k \theta)}  \lb_1   +\frac{\sqrt{2}}{\sin  \theta} \sqrt{1- \cos (k \theta)}\lb_1 .$$

Consider now the discrete version of $f^{\gamma_k,2}$, replacing $u_1$ and $\lambda_1$ by their discrete approximations
\begin{align*}
(f_h^{\gamma_k,2},v)_{H^{-1},H_0^1} & = \sum_{j=0}^{n-1} (\sin(j+1)k\theta+\sin jk\theta) \lambda_{1,h} \int_{T_j} u_{1,h}v\\
& + \sum_{j=0}^{n-1} \frac{\sin(j+1)k\theta-\sin jk\theta}{\sin \theta} \int_{T_j}\begin{pmatrix}
-\sin (2j+1)\theta & \cos (2j+1) \theta \\
\cos (2j+1)\theta & \sin (2j+1)\theta
\end{pmatrix}\nabla u_{1,h} \cdot \nabla v
\end{align*}
Working under the hypothesis that the mesh $\mathcal T^h$ has the symmetries of the regular polygon and that the triangles $T_j$ are meshed exactly, we have $(f_h^{\gamma_k,2},u_{1,h})_{H^{-1},H_0^1} = 0$. Below we use the notation $K_j^{\gamma_k,2} = \frac{\sin(j+1)k\theta-\sin jk\theta}{\sin \theta} \begin{pmatrix}
-\sin (2j+1)\theta & \cos (2j+1) \theta \\
\cos(2j+1)\theta & \sin (2j+1)\theta
\end{pmatrix}$ By direct computation we obtain
{\small
	\[
	(f^{\gamma_k,2}-f^{\gamma_k,2}_h,v)  = \sum_{j=0}^{n-1} (\sin (j+1)k\theta+\sin jk\theta) \int_{T_j}(\lambda_1 u_1-\lambda_{1,h}u_{1,h})v + \sum_{j=0}^{n-1}  \int_{T_j}K_j^{\gamma_k,2}(\nabla u - \nabla u_h)\cdot \nabla v 
	\]
}
which implies
\begin{align*}
\|f^{\gamma_k,2}-f_h^{\gamma_k,2} \|_{H^{-1}} & \leq \sqrt{\frac{1+\cos k\theta}{1+\lambda_1}}(|\lambda_1-\lambda_{1,h}|+\lambda_{1,h}\|u_1-u_{1,h}\|_{L^2}) \\
& + \frac{\sqrt{1-\cos k\theta}}{\sin \theta} \|\nabla u-\nabla u_h\|_{L^2}.
\end{align*}

We conclude this section with the following result summarizing the error estimates obtained.

\begin{thm}
	The terms $\alpha_k,\beta_k,\gamma_k$ in Theorem \ref{thm:explicit-eigenvalues} admit an error estimate of order $O(h^{1-2\gamma})$ for every $\gamma \in (0,1/2)$ when the first eigenfunction $u_1$ and the function $\bo U_0=(U_0^1,U_0^2)$ are approximated using $\bf P_1$ finite elements.
	\label{thm:estimate-abc-P1}
\end{thm}

\emph{Proof:} Recall that the estimates given in Section \ref{sec:eig-estimates} allow us to obtain explicit bounds for $\int_{T_0}(\partial_x u_1)^2$ and $\int_{T_0} (\partial_y u_1)^2$ of order $O(h)$. Denoting $q_k = 2n(1-\cos(k\theta))/\sin \theta$ we have the following. 
\begin{itemize}
	\item For $\alpha_k = q_k\int_{T_0}(\partial_x)^2-2|\Pp| a(U_0^1,W^{\alpha_k})$ we apply \eqref{eq:estimate-a(U,V)} with $U^a = U_0^1,U^b = W^{\alpha_k}$.
	\item For $\beta_k = q_k\int_{T_0}(\partial_y)^2-2|\Pp| a(U_0^2,W^{\beta_k})$ we apply \eqref{eq:estimate-a(U,V)} with $U^a = U_0^2,U^b = W^{\beta_k}$. We note that $(f_{\sing}^{\beta_k},v)_{H^{-1},H_0^1} = 0$ for every function $v$ that is odd with respect to $y$. Since $U_0^1$ and its numerical approximation verify this hypothesis as soon as $\mathcal T_h$ is symmetric with respect to the $x$ axis we may apply Remark \ref{rem:zero-singular-part} and obtain a better error estimate.
	\item For $\gamma_k = -2|\Pp|a(U_0^1,W^{\gamma_k,1})$ we apply \eqref{eq:estimate-a(U,V)} with $U^a = U_0^1, U^b = W^{\gamma_k,1}$. We note that $(f_{\sing}^{\gamma_k,1},v)_{H^{-1},H_0^1} = 0$ for every function $v$ that is even with respect to $y$. Since $U_0^1$ and its numerical approximation verify this hypothesis as soon as $\mathcal T_h$ is symmetric with respect to the $x$ axis we may apply Remark \ref{rem:zero-singular-part} and obtain a better error estimate.
	\item For $\gamma_k = 2|\Pp| a(U_0^2,W^{\gamma_k,2})$ we apply \eqref{eq:estimate-a(U,V)} with $U^a = U_0^2, U^b = W^{\gamma_k,2}$.
\end{itemize}
In conclusion, the terms $\alpha_k,\beta_k,\gamma_k$ admit quantified approximations of order $O(h^{1-2\gamma})$ for every $\gamma \in (0,1/2)$. \hfill $\square$

\section{Numerical simulations}\label{bobu200.s6}

\subsection{Local minimality.} 
Given the regular polygon $\Pp$ with $n$ sides inscribed in the unit circle with a vertex at $(1,0)$, we divide it into $n$ equal slices used in the definition of $\varphi_i$, like in Figure \ref{fig:def-triangles}. Then we give an integer $m\geq 1$ and for each one of the triangles $T_j$, $j=0,...,n-1$ we construct a mesh $\mathcal T^h$ consisting of congruent triangles similar to $\frac{1}{m}T_j$. In this way we obtain a mesh with median length $h=1/m$. Examples are given in Figure \ref{fig:example-regular-meshes}.
\begin{figure}
	\centering
	\includegraphics[height=0.3\textwidth]{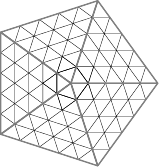}
	\includegraphics[height=0.3\textwidth]{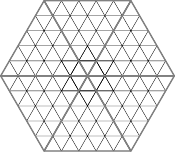}
	\includegraphics[height=0.3\textwidth]{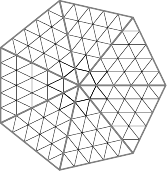}
	\caption{Examples of symmetric meshes for regular polygons used in the computations.}
	\label{fig:example-regular-meshes}
\end{figure}
With this definition of mesh $\mathcal T^h$ all triangles in the mesh are similar and the constant $C_1$ defined in the beginning of Section \ref{sec:eig-estimates} can be explicitly identified in terms of $n$. Given the mesh $\mathcal T^h$ we compute using $\bf P_1$ finite elements:
\begin{itemize}
	\item the first two eigenvalues $\lambda_{1,h},\lambda_{2,h}$ and the first eigenfunction $u_{1,h}$ of the discrete Dirichlet-Laplace eigenproblem \eqref{eq:eigenvalue-finite-elements}.
	\item the solutions $\bo U_h^j = (U_{j,h}^1, U_{j,h}^2)$ of 
	\[ \int_{\Pp} D\bo U_h \nabla v-\lambda_{1,h}\int_\Pp \bo U_h v = (\bo f_h,v)_{H^{-1},H_0^1}\]
	 for the discrete distributions $\bo f_h^j$, $j=0,...,n-1$ given by
	\begin{equation*}
	(\bo f_{h}^j,v)_{H^{-1},H^1} = \int_{\Pp}(\nabla \varphi_j \cdot \nabla u_{1,h})\nabla v +\int_{\Pp}( \nabla \varphi_j\cdot \nabla v)\nabla u_{1,h}+s_{1,h}^\lambda \int_\Pp u_{1,h}v.
	\end{equation*}
	 using the normalization $\int_\Pp \bo U_h^j u_{1,h} = 0$.
	\item for $1\leq k \leq n-1$ approximations of $W^{\alpha_k},W^{\beta_k},W^{\gamma_k,1},W^{\gamma_k,2}$ are constructed from $(U_{j,h}^1, U_{j,h}^2)$. Therefore we obtain the approximations of $\alpha_k,\beta_k,\gamma_k$ from Theorem \ref{thm:explicit-eigenvalues} that are of order $O(h^{1-2\gamma})$ for $\gamma \in (0,1)$, with explicit error bounds given in the previous section.
\end{itemize}
The procedure described above provides for each $k=1,...,n-1$ intervals $I_{\alpha_k}, I_{\beta_k}, I_{\gamma_k}$ for which we have the guarantee that $\alpha_k \in I_{\alpha_k},\beta_k \in I_{\beta_k}, \gamma_k \in I_{\gamma_k}$. Using the interval arithmetic toolbox Intlab \cite{intlab} we find intervals $I_j (h,\gamma)$ containing the eigenvalues $\mu_j$, $0 \leq j \leq 2n-1$ of $\bo M^{\lambda}$ described in Theorem \ref{thm:explicit-eigenvalues}. Given a value of $h$ and the associated numerical approximations we obtain a whole range of intervals $I_j(h,\gamma)$ for $\gamma\in (0,0.5)$. Note that changing $\gamma$ at fixed $h$ is not a difficulty since this parameter appears only in the choice of constants and exponents. When $\gamma$ is close to zero we obtain a weak estimate in Theorem \ref{thm:estimate-singular} while for $\gamma$ close to $0.5$ the constants in the estimates from Remark \ref{rem:c-gamma} become very large. An appropriate choice for $\gamma$ is made using a simple grid search. If among the intervals $I_j(h,\gamma)$ we obtain only two that contain zero then we conclude, based on Proposition \ref{prop:justification-2n-4}, that the regular polygon $\Pp$ is a local minimum for $P\mapsto |P|\lambda_1(P)$. If this is not the case we decrease $h$ and we repeat the procedure. 

\begin{rem}
    Numerical algorithms employed in scientific computing use floating point arithmetic. As a consequence there is a difference between the {\bf exact discrete solution} of the finite element problem and the one given by the numerical algorithm. The sources of error are as follows:
	\begin{itemize}
		\item the numerical mesh is a slight perturbation of the exact mesh, leading to perturbations in the mass and rigidity matrices.
		\item the linear systems are solved using iterative methods with a stopping criterion related to the residual vector.
	\end{itemize}
	In general, it is admitted that errors coming from the above considerations are smaller than the theoretical error estimates shown in Theorem \ref{thm:estimate-abc-P1}.  The condition number of the linear systems involved is of order $O(h^{-2})$, therefore, we expect that for $h\geq 10^{-4}$ the machine errors do not dominate in the estimation of $\lambda_1$. Moreover, for the gradient terms and for $\bo U_0$, which have an even weaker convergence rate, the discretization error is expected to dominate machine errors. We also make this assumption in the following.
\end{rem}

Formulas given in Theorem \ref{thm:explicit-eigenvalues} and Proposition \ref{prop:detailed-computations} allow us to compute the eigenvalues of the Hessian matrix in knowing the first eigenfunction $u_1$ on $\Bbb{P}_n$ and the pair $(U_0^1,U_0^2)$ solution of \eqref{eq:material-eig-decomposed}. Using ${\bf P_1}$ finite elements it is straightforward to approximate the first eigenpair. Given a mesh $\mathcal T_h$ with $N_v$ vertices, and denoting by $(\phi_i)_{i=1}^{N_v}$ the ${\bf P_1}$ basis functions, the rigidity and mass matrices are defined by
\[ \bo A = \left(\int_{\Bbb P_n} \nabla \phi_i \cdot \nabla \phi_j\right)_{1\leq i,j\leq n}, \bo B = \left(\int_{\Bbb P_n}  \phi_i  \phi_j\right)_{1\leq i,j\leq n}. \]
The first eigenpair and the second eigenvalue are approximated by solving the generalized eigenvalue problem $\bo A \bo x = \lambda \bo B \bo x$. Denote by $\bo x_1$ the eigenvector associated to the first eigenvalue. Then \eqref{eq:material-eig-decomposed.h} is solved by considering embedding the orthogonality on $u_{1,h}$ in the linear system:
\[ \begin{pmatrix}
 \bo A - \lambda_{1,h} \bo B & \bo c \\
 \bo c^T & 0
\end{pmatrix}
\begin{pmatrix}
\bo x \\ l
\end{pmatrix}
= \begin{pmatrix}
\bo f \\
0
\end{pmatrix}.\]
The constraint vector $\bo c$ is given by $\bo c = \bo x_1^T \bo M$ and the right hand side $\bo f$ is computed by evaluating $(f_0^{1,2},\phi_i)_{H^{-1},H^1}$ for every $\phi_i$ in the finite element basis. 

In order to have an error estimate small enough such that the interval around the eigenvalue does not contain zero rather small values of $h$ need to be considered, leading to large computational problems. The value of $h$ and the number of degrees of freedom (d.o.f) for the computational problems are listed in Table \ref{table:size-comp}. Therefore, in order to be able to solve these problems the software FreeFEM \cite{freefem} is used in its parallel version together with the libraries PETSc \cite{petsc}, SLEPc \cite{slepc}, Hypre \cite{hypre}. The computations use $200$ processors and are run on the cluster Cholesky from the IDCS Mesocenter at Ecole Polytechnique. The error estimates allow us to obtain sufficiently small intervals for $h=10^{-4}$ for $n \in \{5,6,7,8\}$. The  resulting eigenvalues and quantities needed are given to the interval arithmetic library Intlab \cite{intlab}. The library is then used to compute the interval enclosures for the eigenvalues. The non-zero eigenvalues and the corresponding enclosures are given in Table \ref{tab:eig-bounds}. The results shown in Table \ref{tab:eig-bounds} indicate that the regular polygon is a local minimizer for problem \eqref{eq:minimization-lamk} for $n \in \{5,6,7,8\}$. In Table \ref{table:optimal-size-comp} we estimate the largest mesh size $h$ for which the certified numerical computations validate the local minimality of the corresponding regular polygon. Exploiting the symmetry of the eigenfunction and of the functions $U_0^1, U_0^2$ the size of the problems can be further reduced in half. 

\begin{table}
	\begin{tabular}{|c|c|c|}
		\hline 
		 & $h$ & d.o.f. \\ \hline 
		Pentagon & $10^{-4}$ & 250\ 025\ 001\\ \hline 
		Hexagon & $10^{-4}$ & 300\ 030\ 001\\ \hline 
		Heptagon & $10^{-4}$ & 350\ 035\ 001\\ \hline 
		Octagon & $10^{-4}$ & 400\ 040\ 001\\ \hline 
	\end{tabular}
	\caption{Size of the computational problems for the finite element computations.}
	\label{table:size-comp}
\end{table}

\begin{table}
	\begin{tabular}{|c|c|c|c|}
		\hline 
		\multicolumn{4}{|c|}{Pentagon} \\
		\hline 
		Eig. & l.b. & u.b. & mult. \\ \hline 
		2.568803 & 2.359297 & 2.784816 & 2\\ \hline 
		8.015038 & 7.558395 & 8.460722 & 2\\ \hline 
		13.458443 & 13.012758 & 13.915086 & 2\\ \hline 
	\end{tabular}
	\begin{tabular}{|c|c|c|c|}
		\hline 
		\multicolumn{4}{|c|}{Hexagon} \\
		\hline 
		Eig. & l.b. & u.b. & mult. \\ \hline 
		1.323826 & 1.040291 & 1.629895 & 2\\ \hline 
		3.916803 & 3.112218 & 4.719205 & 2\\ \hline 
		12.990672 & 12.188270 & 13.795257 & 2\\ \hline 
		7.566593 & 6.326083 &  8.803012 & 1 \\ \hline
		11.540733 & 10.304314 & 12.781243 & 1 \\ \hline 
	\end{tabular}
	\begin{tabular}{|c|c|c|c|}
		\hline 
		\multicolumn{4}{|c|}{Heptagon} \\
		\hline 
		Eig. & l.b. & u.b. & mult. \\ \hline 
		0.747352 & 0.446026 & 1.096876 & 2\\ \hline 
		2.056766 & 0.963449 & 3.148214 & 2\\ \hline 
		4.655979 & 3.078862 & 6.228621 & 2\\ \hline 
		12.292485 & 10.719843 & 13.869602 &2 \\ \hline 
		12.582047 & 11.490599 & 13.675364 & 2 \\ \hline 
	\end{tabular}
	\begin{tabular}{|c|c|c|c|}
		\hline 
		\multicolumn{4}{|c|}{Octagon} \\
		\hline 
		Eig. & l.b. & u.b. & mult. \\ \hline 
		0.452095 & 0.182855 & 0.774247 & 2\\ \hline 
		1.171933 & 0.309482 & 2.034382 & 2\\ \hline 
		2.772135 & 1.273803 & 4.268064 & 2\\ \hline 
		12.049631 & 11.187182 & 12.912082 & 2 \\ \hline 
		13.037208 & 11.541279 & 14.535540 & 2 \\ \hline 
		3.999568 & 1.460555 & 6.536411 & 1\\ \hline 
		11.740713 & 9.203870 & 14.279726 & 1 \\ \hline 
	\end{tabular}
	\caption{Numerical approximations of the $2n-4$ non-zero eigenvalues of the Hessian matrix for $n\in \{5,6,7,8\}$ together with intervals given by the error estimate in Theorem \ref{thm:estimate-abc-P1}}
	\label{tab:eig-bounds}
\end{table}

\begin{table}
	\begin{tabular}{|c|c|c|}
		\hline 
		& Mesh size & deg. freedom \\ \hline 
		Pentagon & 9.8e-4 & $\approx$ 2.6 million\\ \hline 
		Hexagon & 4.2e-4 & $\approx$ 17 million\\ \hline 
		Heptagon & 1.9e-4  & $\approx$ 97 million\\ \hline 
		Octagon & 1.35e-4 & $\approx$ 220 million\\ \hline 
	\end{tabular}
	\caption{Approximately optimal mesh sizes and number of degrees of freedom for which currently known {\it a priori} estimates allow to certify the local minimality.}
	\label{table:optimal-size-comp}
\end{table}

\begin{rem}
	The results shown in this section prove the local minimality of the regular polygon when neglecting errors coming from floating point computations. Most algorithms are designed such that these errors are minimized and therefore it is generally agreed that these errors are smaller than the errors between the continuous solution and the exact discrete one. However, guaranteeing that the floating point errors are small enough it is a non-trivial matter that needs to be addressed in future works. Ideally, the whole computation of the finite element problems should be handled using an interval arithmetic library like Intlab \cite{intlab}, which is a non-trivial task in view of the minimal size of the problems listed in Table \ref{table:optimal-size-comp}.
\end{rem}

It is possible to compute the eigenvalues of the Hessian matrix for higher $n$, without guarantee that the numerical eigenvalues are precise enough. Nevertheless, it is well established that \emph{a priori} estimates are rather pessimistic and the following results might precise enough. In Table \ref{tab:eig-numeric} we present the non-zero eigenvalues of the Hessian matrix for $h=10^{-3}$ for $9\leq n \leq 15$. These eigenvalues are positive, suggesting that the regular polygon is still a local minimzier in these cases.

\begin{table}
	\centering
	\begin{tabular}{|c|c|}
		\hline 
		$n = 9$ & mult.\\ \hline 
		0.2888 & 2\\
		0.7145 & 2\\
		1.7104 & 2\\
		2.8667 & 2\\
		11.4506 & 2\\
		12.1695 & 2\\
		13.4392 & 2\\
		\hline 
	\end{tabular}
\begin{tabular}{|c|c|}
	\hline 
	$n = 10$ & mult.\\ \hline 
	0.1927 & 2\\
	0.4601 & 2\\
	1.1017 & 2\\
	1.9625 & 2\\
	2.4640 & 1\\
	10.8361 & 2\\
	11.9253 & 1\\
	12.7814 & 2\\
	13.5487 & 2\\
	\hline 
\end{tabular}
\begin{tabular}{|c|c|}
	\hline 
	$n = 11$ & mult.\\ \hline 
	0.1334 & 2\\
	0.3096& 2\\
	0.7386& 2\\
	1.3501& 2\\
	1.9129& 2\\
	10.2373& 2\\
	12.1968& 2\\
	13.2741& 2\\
	13.4475& 2\\
	\hline 
\end{tabular}
\begin{tabular}{|c|c|}
	\hline 
	$n = 12$ & mult.\\ \hline 
	0.0952& 2\\
	0.2160& 2\\
	0.5128& 2\\
	0.9473& 2\\
	1.4287& 2\\
	1.6659& 1\\
	9.6701& 2\\
	12.0620& 1\\
	12.6398& 2\\
	13.2059& 2\\
	13.5861& 2\\
	\hline 
\end{tabular}

\begin{tabular}{|c|c|}
	\hline 
	$n = 13$ & mult.\\ \hline 
	0.0697&2 \\
	0.1554&2 \\
	0.3669&2 \\
	0.6801&2 \\
	1.0598&2 \\
	1.3586&2 \\
	9.1413&2 \\
	12.2461&2 \\
	12.8768&2 \\
	13.0664&2 \\
	13.7288&2 \\
	\hline 
\end{tabular}
\begin{tabular}{|c|c|}
	\hline 
	$n = 14$ & mult.\\ \hline 
	0.0521&2 \\
	0.1146&2 \\
	0.2694&2 \\
	0.4994&2 \\
	0.7918&2 \\
	1.0742&2 \\
	1.1995&1 \\
	8.6527&2 \\
	12.1611& 1 \\
	12.4975&2 \\
	12.5693&2 \\
	13.4024&2 \\
	13.7331&2 \\
	\hline 
\end{tabular}
\begin{tabular}{|c|c|}
	\hline 
	$n = 15$ & mult.\\ \hline 
	0.0397&2 \\
	0.0864&2 \\
	0.2022&2 \\
	0.3744&2 \\
	0.5989&2 \\
	0.8406&2 \\
	1.0115&2 \\
	8.2033&2 \\
	12.0933&2 \\
	12.2933&2 \\
	12.9147&2 \\
	13.6292&2 \\
	13.6320&2 \\
	\hline 
\end{tabular}
\caption{Numerically computed non-zero eigenvalues of the Hessian matrix for larger $9\leq n \leq 15$ on meshes of size $h=10^{-3}$.}
\label{tab:eig-numeric}
\end{table}

\subsection{General gradient descent simulations}

The gradient of the first eigenvalue with respect to the coordinates of the vertices is given in Theorem \ref{thm:grad-eig}. Using these formulas is straightforward to implement a gradient descent algorithm starting from random initial polygons. 

Simulations were preformed for the minimization of the first eigenvalue for $n \in [5,15]$ and in every case the result of the optimization was a polygon very close to being regular. In order to see how close to being regular is the polygon $\omega_n$ given by the simulation the following information is given in Table \ref{tab:result-optim}: the optimal numerical first eigenvalue, the difference between the maximal and minimal edge lengths, the difference between the maximal and minimal angles (in radians), the difference between the optimal numerical eigenvalue and the precise first eigenvalue of the regular polygons $\omega_n^*$ given on the following web page: \url{http://hbelabs.com/regularpolygon/index.html} (based on the article \cite{Jones-precision}). Repeating the simulation starting from random initial polygon always gives similar results. 
\begin{table}
	\centering
	\begin{tabular}{c|c|c|c|c}
		$n$ & $J(\omega_n)$ & diff. sides & diff. angles & $J(\omega_n)-J(\omega_n^*)$ \\ \hline
		$5$ & $18.919104$ & \texttt{1.3e-5} &  \texttt{2.3e-5} &  \texttt{3.4e-9} \\ \hline 
		$6$ & $18.590116$ & \texttt{5.1e-5} & \texttt{7.7e-5} & \texttt{3.2e-8} \\ \hline 
		$7$ & $18.429994$ & \texttt{8.4e-5} & \texttt{1.8.1e-4}& \texttt{1.1e-7} \\ \hline
		$8$ & $18.342161$ & \texttt{9.2e-5} & \texttt{2.1e-4}& \texttt{1.6e-7} \\ \hline 
		$9$ & $18.289808$ & \texttt{3.8e-4} & \texttt{3.7e-4} & \texttt{2.6e-7} \\ \hline 
		$10$ & $18.256613$ & \texttt{3.1e-4} & \texttt{6.1e-4} & \texttt{5e-7} \\ \hline 
		$11$ & $18.234528$ & \texttt{3.3e-4} & \texttt{4.1e-4} & \texttt{3.3e-7} \\ \hline 
		$12$ & $18.219257$ & \texttt{3.3e-4} & \texttt{5e-4} & \texttt{2.9e-7} \\ \hline 
		$13$ & $18.208358$ & \texttt{6.5e-4} & \texttt{1.3e-3} & \texttt{4.8e-7} \\ \hline 
		$14$ & $18.200368$ & \texttt{7.5e-4} & \texttt{2.1e-3} & \texttt{6.6e-7} \\ \hline 
		$15$ & $18.194378$ & \texttt{1.5e-3} & \texttt{3.1e-3} & \texttt{1.7e-6} \\ \hline  
	\end{tabular}
	\caption{Results of the gradient descent optimization algorithm with random initial polygons.}
	\label{tab:result-optim}
\end{table}
The results shown in Table \ref{tab:result-optim} indicate that the optimal numerical polygons $\omega_n$ found by the numerical algorithm are close to being regular. Furthermore, the value of the objective function is as close to the precise value given for the actual regular polygon $\omega_n^*$, as the precision of the numerical computations allows. These computations further suggest that the regular polygon is indeed the global minimizer for \eqref{eq:minimization-lamk}.

\section{Reduction of the proof of the conjecture to a finite number of numerical computations}\label{bobu200} 
 
In this section we provide a strategy for proving the conjecture using a finite number of computations for a given number of sides. This strategy works under the implicit assumption that the conjecture is true!

In order to justify that for every $n$ the conjecture can be reduced to a finite number of numerical computations, we begin with some theoretical analysis. Assuming the area of a polygon with $n$ sides is fixed (say $\pi$), we shall find a value $D_{max}$ such that if the diameter of the polygon exceeds $D_{max}$ then the polygon cannot be optimal for \eqref{eq:minimization-lamk}. As well, we shall find a minimal value for the length of the edges $e_{\min}$ and for the inradius $r_{\min}$ of an optimal polygon. All these results (which depend on $n$), produce a compact set of polygons (seen as subset of $\R^{2n-4}$) outside which any polygon cannot be optimal for $|P|\lb_1(P)$. 

We denote by $ \ov  { \mathcal P}_n $ the closure of the class of simple polygons with at most $n$ edges for the Hausdorff distance of the complements. A polygon belonging to this class may be degenerate in the sense that one vertex can belong to a different edge. Depending on how this occurs, this may lead to a disconnection,  i.e. a union of two polygons. However, as soon as a polygon is optimal, disconnection can not occur.

Let us denote for every $n \ge 3$ the minimal value for the scale invariant formulation by
$$l_n^*= \min \{|P|\lambda_1(P) : P \in  \ov  { \mathcal P}_n \}.$$
It is known that $l_n^* < l_{n-1}^*$ (see \cite[Section 3.3]{henroteigs}).
\begin{thm}\label{bobu47} 
Let $n\ge 3$. There exists a value $D_{max} >0$ such that if $P \in \ov {\mathcal P}_n$, $|P|=\pi$ and $\diam (P) > D_{max}$ then 
$$\pi \lambda_1(P) > l_{n}^*.$$
\end{thm}
In other words, when searching the minimizer in the class of $n$-gons of area $\pi$, it is enough to restrict to polygons with diameter less than or equal to $D_{max}$. This information is crucial in order to limit the number of numerical computation and leads to a formal, inductive, proof of the conjecture. The value of $D_{max}$ can be computed and depends on $l^*_{n-1}$ and  $\lb_1(\mathbb P_n)$.

{\it Proof:} The proof is inspired by the surgery argument of \cite{BuMa15}, where the authors propose a precise way to estimate the diameter of an optimal set in relationship with the first eigenvalue. The key idea is that if the diameter of an optimal set  is too large, one can cut the set with a strip of positive width in order to produce a better one. The main difficulty in our case is that cutting a polygon having $n$ edges with a strip may produce a union of polygons, some of which may potentially have more than $n$ edges, making them non-admissible. In order to handle this situation, further analysis is necessary. 

\medskip\noindent{\bf Setting the constants.} 
 Denoting $\Lambda=  l_n^*/\pi^2 $, we consider 
  the unconstrained problem 
\begin{equation}\label{bobu203}
\min \{ \lb_1(P) + \Lambda  |P|: P \in  {\mathcal P_n}\}.
 \end{equation}
Then, the solution of this problem is the same as the solution of the constrained problem with area $\pi$ set in \eqref{eq:minimization-lamk}. Let us denote by $Q_n$ an optimal polygon, having area $\pi$. Let $K \ge l_n^*/\pi$ be fixed. For instance, $K$ may be obtained using a numerical approximation from above of $\lb_1(\Bbb P_n)$. 
 
\medskip

\noindent{\bf Surgery.}
In order to get the bound on the diameter, we shall use the surgery results of \cite{BuMa15}. Let us set the following constant
$$c= \frac{1}{2\pi(8+12 \log 2)e^{\frac{1}{4\pi}} K^2},$$
which plays the crucial role in \cite[Lemma 3.1]{BuMa15}. We can use  \cite[Lemma 4.2]{BuMa15} with the constant $c$ from above, which (in the notations of \cite[Lemma 4.2]{BuMa15})  leads to suitable values $(r_0, C_0)$. For instance, we can choose $C_0(C_0+1) \le c$ and $r_0=C_0$. 

\smallskip
\noindent{\bf Step 1.} (Use of  \cite[Lemma 3.1]{BuMa15}) In view of the choice of $c$, the polygon $Q_n$ is a subsolution for the torsion energy 
$$P \to E(P)+c |P|,$$
in the class ${\mathcal P}_n$. We recall that the torsion energy of $P$ is defined by
$$E(P)= \min _{u \in H^1_0(P)} \frac 12 \int_P |\nabla u|^2 dx -\int_Pu dx.$$
Indeed, if for some $P \in {\mathcal P}_n$, $P\sq Q_n$  we have
$$E(P)+c|P| <  E(Q_n)+c |Q_n|,$$
then from \cite[Lemma 3.1]{BuMa15})  we would get
$$|P|\lb_1(P) < |Q_n| \lb_1(Q_n),$$
in contradiction to the optimality of $Q_n$. 

\smallskip
\noindent{\bf Step 2.} (Use of  \cite[Corollary 4.3]{BuMa15}) Let $w$ be the torsion function of $Q_n$. Assume $a \in \R$ and denote by 
$$S_{r}(a) = \{(x,y) : r-a <x<r+a\}$$ 
an open strip in $\R^2$. Assume that the interior of the strip  intersects $Q_n$ and does not contain any vertex. In this case, the intersection of the strip with $Q_n$ is a union of trapezes $\{T_j\}_{j \in J}$. When removing any of these trapezes, one splits the polygon $Q_n$ in two (or more, if a vertex is on the boundary of the strip) polygons.

Following \cite[Corollary 4.3]{BuMa15}, using the constants $(r_0, C_0)$ defined above, we know that 
if $\max_{S_{2r}(a) } w <C_0^2$ then
\begin{equation}\label{bb62}
 E(Q_n\sm \ov S_{r}(a) )+c |Q_n\sm \ov S_{r}(a)|<  E(Q_n)+c |Q_n|.
 \end{equation}
 In fact, taking a closer look to the argument  of  \cite[Corollary 4.3]{BuMa15}, leads as well to 
 \begin{equation}\label{bb62.5}
 E(Q_n\sm \ov T_j)+c |Q_n\sm \ov T_j|<  E(Q_n)+c |Q_n|.
 \end{equation}
As a consequence of \cite[Lemma 3.1]{BuMa15} this implies
$$|Q_n\sm \ov T_j|\lb_1(Q_n\sm \ov T_j) < |Q_n| \lb_1(Q_n) \text{ for every } j \in J.$$
This last inequality leads to a contradiction of  the optimality of $Q_n$ only if the open set $Q_n\sm \ov T_j$ consists in a union of polygons, each one with at most $n$ edges. In this case, it is enough to pick the one with minimal first eigenvalue and contradict the optimality of $Q_n$. Of course, it may happen that one of the connected components of $Q_n\sm \ov  T_j$ is a polygon with more than $n$ edges, as   new edges could be produced by the surgery procedure.  We shall prove that if the diameter is larger than some computable constant, then there exists some suitable strip  $S_{r}(a)$  and a suitable trapeze $T_j$  such that  each connected component of $Q_n\sm \ov T_j$  is a polygon with at most $n$ edges. This contradicts the optimality of $Q_n$.

\smallskip
\noindent{\bf Step 3.} (Preparatory facts) We know from the Saint-Venant inequality that
$$\int_{Q_n} w dx \le \frac{\pi}{8}.$$
The following results is, for instance,  contained in \cite[Lemma 2.2]{BuMa15}: 
$$\text{if } w(x_0) \ge \eta >0, \text{ then }  \int _{B_\delta (x_0)} w dx\ge \frac{\eta \pi}{2} \delta ^2,$$
where $\delta = 2 \sqrt{\eta}$.

Consequently, if we consider a strip $S_{2r_0}(a)$ such that 
$$\max _{S_{2r_0}(a)} w >C_0^2,$$
then, recalling that $C_0=r_0$,
$$\int_{B_{2C_0} (x_0)} w dx \ge 2\pi C_0^4,$$
where $x_0$ is a maximum point of $w$ in $S_{2r_0}(a)$. In particular
$$\int_{S_{4r_0}(a)} w dx \ge 2\pi C_0^4.$$
Let us introduce the natural number ($\lfloor \cdot \rfloor$ denotes the integer part)
$$k= \left\lfloor\frac{\frac {\pi}{8}}{2\pi C_0^4}\right \rfloor +1= \left \lfloor\frac{1}{16C_0^4}\right \rfloor +1.$$

Clearly, if the diameter of $Q_n$ is larger than $8C_0k$, then taking the $x$-axis along the diameter, there will be at least one strip of width $4 C_0$ where the mass of $w$ is less than $C_0^2$.

We recall now the following inequality, for which we refer to \cite{BoVdB99}. Let $\Om$ be a bounded, open simply connected set in $\R^2$. 
Let $w_\Om$ be the torsion function in $\Omega$.
We have
    $w_\Om(x) = \int_\Omega G_\Omega(x,y)dy$, where $G_\Omega(x,y)$ is the Green function for the Dirichlet-Laplace operator on $\Omega$. From the Cauchy-Schwarz inequality we have
    \[ |w_\Om(x)| \leq |\Omega|^{1/2} \left(\int_\Omega G_\Omega (x,y)^2dy\right)^{1/2}.\]
    
    In  \cite[Proof of Theorem 1.5, inequality (5.16)]{BoVdB99}  it is shown that if $\pi R_0^2=|\Omega|$ then 
    \[ \int_\Omega G_\Omega^2(x,y)dy \leq \frac{8 d(x)R_0}{\pi},\]
    where $d(x)$ is the distance from $x$ to $\partial \Omega$. This leads to the estimate
 \begin{equation}\label{bb64.5}
  |w_\Om(x)|\leq |\Omega|^{3/4}\frac{8^{1/2}d(x)^{1/2}}{\pi^{3/4}}.
  \end{equation}
  We use this inequality for $\Om=Q_n$, so that $|Q_n|=\pi$, getting the bound $w_{Q_n}(x) \le 2 \sqrt{2} d(x)^{1/2}$.
 
  We introduce now $e^*, d^*$ such that 
\begin{equation}\label{bb64.6}
2\sqrt{2} (e^*) ^\frac 12< C_0^2\quad \mbox{and} \quad d^*=\frac{\pi}{e^*}.
  \end{equation}

\begin{lemma}\label{bb64}
 The diameter of $Q_n$ can not be larger than $2 d^* + (k+n-2)8 C_0$.
\end{lemma}
\begin{proof}  
Assume for contradiction that there are two vertices $\bo a_0$, $\bo a_m$ such that the diameter of $Q_n$ is the segment $[\bo a_0, \bo a_m]$ and that its length is larger than $2 d^* + { (k+n-2)8 C_0} $.
 {Around the midpoint of $[\bo a_0 \bo a_m]$ we build $k+n-2$ adjacent strips of width $8 C_0$. Outside the strips there are two sub-segments of $[\bo a_0 \bo a_m]$, each having length at least $d^*$ (see Figure \ref{fig:moving-trapezes}}). We remove at most $n-2$ strips having a vertex in their interior and among the remaining $k$ strips there is one, say $S_{4C_0}(a)$ such that 
$$\max _{S_{2C_0}(a)} w <C_0^2.$$
 From the choice of the strip, the set $\ov S_{C_0}(a)$ does not contain any vertex of the polygon $Q_n$, so that an edge either crosses the strip from one side to the other, or it stays on the same side. In particular, this implies that $Q_n\cap S_{C_0}(a)$ is a union of open trapezes  $\{T_j\}_j$. Moreover,  we get for each such trapeze
$$|Q_n\sm \ov T_j|\lb_1(Q_n\sm \ov T_j) < |Q_n| \lb_1(Q_n).$$
 
Assume we remove  {one trapeze, say $T_j$,}  from $Q_n\cap S_{C_0}(a)$ and get  two polygons   $P^l_{T_j} $ and $P_{T_j}^r$,  which together have $n+4$ edges.  There are two possibilities.

\begin{enumerate} 
\item Both polygons $P^l_{T_j} $ and $P_{T_j}^r$ have no more than $n$ edges. This situation contradicts the optimality of $Q_n$. 
\item One of  $P^l_{T_j} $ and $P_{T_j}^r$ has $n+1$ edges and the other one has $3$ edges. 
\end{enumerate}
In the following, we suppose that the second situation above occurs for each trapeze $T_j$, otherwise we contradict optimality. We claim that on one side of the strip there are only triangles. 

If there is only one trapeze, there is nothing to prove. Assume for contradiction that there are two trapezes, which when removed generate triangles on both sides of the strip. From simple connectedness, there is a continuous curve contained inside the polygon, joining the interiors of the two triangles. See Figure \ref{fig:moving-trapezes} (a). This curve crosses the strip at least one more time, implying the presence of at least another trapeze, which cannot leave a triangle on either side when removed without disconnecting the polygon. Therefore, removing this trapeze, we split the polygon in two polygons with less than $n$ edges contradicting optimality. 

In conclusion, removing any one of the trapezes $T_j$ generates triangles, all situated on one side of the strip. Assume this occurs on the left. 
Now, we choose the triangle containing the vertex $\bo a_0$ on the left, which is at distance at least  $d^*$ from the strip and the trapeze which isolates it in a triangle. We continuously move the strip $S_{C_0}(a)$ to the right (and the trapeze with it) up to the moment when the strip touches a first vertex. This vertex can be a neighbor of $\bo a_0$ (Figure \ref{fig:moving-trapezes} (a)) or a different vertex $\bo a_k$ (Figure \ref{fig:moving-trapezes} (c)). In any case, the trapeze will split the polygon in either two or three polygons and the number of edges for each polygon is at most $n$. 

\begin{figure}
	\centering 
	\begin{tabular}{ccc}
	\includegraphics[width=0.15\textwidth]{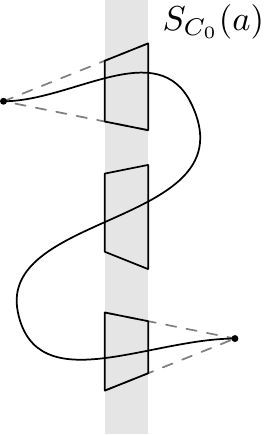}&
	\includegraphics[width=0.39\textwidth]{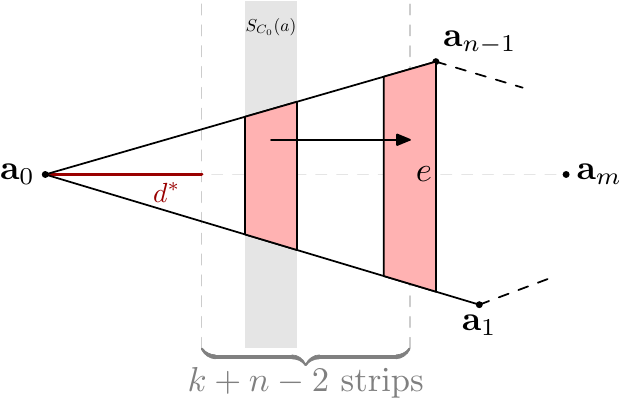}&
	\includegraphics[width=0.39\textwidth]{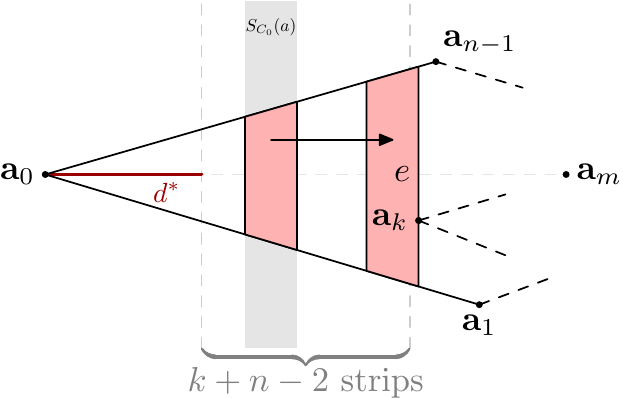}\\
	(a) & (b) & (c)
	\end{tabular}
	\caption{(a) Continuous curve linking two triangles on opposite sides of the strip. Moving the trapezes in the proof of Lemma \ref{bb64}: the trapeze meets a neighbor of $\bo a_0$ (b) or another vertex $\bo a_k$ (c).}
	\label{fig:moving-trapezes}
\end{figure}

Moreover, one polygon is the triangle with a vertex in $\bo a_0$. The area of this triangle together with the trapeze is at least $d^* e/2$, where $e$ is the length of the longest vertical edge of the trapeze, on the right side of the strip. This set is fully contained in the polygon, so has area at most $\pi$, meaning that $\frac e2 \le e^*$. Using inequalities  \eqref{bb64.5}-\eqref{bb64.6} we get that the maximum of $w_{Q_n}$ on the trapeze is below $C_0^2$. This contradicts the optimality of the polygon. 
\end{proof}

\begin{thm}\label{bobu220}
Assume $P=[\bo a_0... \bo a_{n-1}] \in   \ov {\mathcal P}_n$ is such that $|P|=\pi$, $\diam (P)\le D_{max}$. There exists $\delta _0 >0$ such that if  $|\bo a_0\bo a_1|\le \delta\le \delta_0$ then
\begin{equation}\label{bobu221}
\pi \lb_1(P)\ge l_{n-1}^*-C\delta^\frac12,
\end{equation}
where $C $ depends only on $n$. 
\end{thm}
In other words, an optimal polygon of area $\pi$ in  $\overline {\mathcal P}_n$ can not have an edge smaller than a certain threshold. To observe this fact, it is enough to choose $\delta_0 $ such that
\begin{equation}
 l_{n-1}^*-C\delta_0^\frac 12> l_n^*.
 \label{eq:delta-edge-min}
\end{equation}
 The following type of result has been proved by Davies in \cite{Da93} and refined by Pang in \cite{Pa97}. We give a short proof below, based on  the comparison with the torsion function. 
\begin{lemma}\label{bobu223}
Assume $P=[\bo a_0...\bo a_{n-1}] \in \ov {\mathcal P}_n$ is such that $|P|=\pi$, $\diam (P)\le D_{\max}$. Let $Q \in \ov {\mathcal P}_n$, $Q=[\bo b_0...\bo b_{n-1}] $ such that for every $i=0, \dots, n-1$, $|\bo a_i\bo b_i|\le \delta$.
Then
$$|\lb_1(Q)-\lb_1(P)| \le 4 \sqrt{2}\pi e^{\frac{1}{4\pi}}(\max \{\lb_1(P), \lb_1(Q)\})^2\lb_1(P\cap Q) \delta ^\frac 12.$$
\end{lemma}
\begin{proof}
 Assume in a first step  that $Q \sq P$ and {denote $w_Q, w_P$ the associated torsion functions}.
The inequality is a consequence of \cite[Inequality (2.6)]{BuMa15} which gives
$$0\le \lb_1(Q)-\lb_1(P) \le 2e^{\frac{1}{4\pi}}\lb_1(P)^2\lb_1(Q)\int_P (w_P-w_Q) dx$$
and of the estimate
$$\int_P (w_P-w_Q) dx \le 2\sqrt{2} \pi \delta ^\frac 12$$ 
which is a consequence of \eqref{bb64.5} applied to $w_P$ and of the harmonicity of $w_P-w_Q$ on $Q$. 

In general, if $Q \not \sq P$, we use the previous argument and compare both $\lb_1(P), \lb_1(Q)$ with $\lb_1(P\cap Q)$. 
\end{proof}
\begin{proof} (of Theorem \ref{bobu220})
Assume $P=[\bo a_0...\bo a_{n-1}] \in {\mathcal P}_n$ is such that $|P|=\pi$, $\diam (P)\le D_{\max}$ and $|\bo a_0\bo a_1|\le \delta$. If 
$\pi \lb_1(P) \ge l^*_{n-1}$, inequality \eqref{bobu221} is proved. Assume that $\pi \lb_1(P) < l^*_{n-1}$. We shall build a polygon $Q\in {\mathcal P}_{n-1}$ having almost the same eigenvalue and area. 

Assume at least one of the angles $\widehat{\bo{a}_0}, \widehat{\bo{a}_1}$ is convex, for example $\widehat{\bo{a}_0}$. Then we move the point $\bo a_0$ towards $\bo a_1$ continuously, denoting it $\bo a_0^t= (1-t) \bo a_0+t\bo a_1$. If the segment $[\bo a_{n-1}\bo a_0^t]$ does not meet any other vertex of the polygon for any $t\in(0,1)$, then we denote $Q$ the new polygon obtained for $t=1$. Clearly, $Q\in  {\mathcal P}_{n-1}$ and Lemma \ref{bobu223} can be applied to get
$$ \lb_1(Q)-\lb_1(P) \le 4 \sqrt{2}\pi e^{\frac{1}{4\pi}}\lb_1(Q)^3  \delta ^\frac12.$$
From Makai's inequality  \cite{Ma65} we know that $\lb_1(P) \ge \frac {1}{4 \rho_P^2}$, where $\rho_P$ is the inradius. Since $\pi \lb_1(P)< l^*_{n-1}$, we get
$$\frac {\pi}{4l^*_{n-1} }< \rho_P^2.$$
On the other hand, $\rho_Q\ge \rho_P- \delta$, hence 
$$\lb_1(Q) \le  \frac{1}{\rho_Q^2} \lb_1(B_1)\le \frac{1} {\big( \pi/(4l^*_{n-1} )\big) ^\frac 12 -2 \delta} \lb_1(B_1).$$
 Finally we get
\begin{multline*}\frac{l_{n-1}^ *}{\pi - D_{\max}\delta}-\lb_1(P) \le \frac{l_{n-1}^ *}{|Q|}-\lb_1(P) \le \lb_1(Q)-\lb_1(P) \\
\le 4 \sqrt{2}\pi e^{\frac{1}{4\pi}} \left(\frac{1} {\big( \pi/(4l^*_{n-1} )\big)^\frac 12 -2 \delta} \lb_1(B_1) \right)^3  \delta ^\frac12,
\end{multline*}
and  we conclude this case.

If for some $t\in (0,1)$ the edge $[\bo a_{n-1}\bo a_0^t]$ meets a vertex. Then the inequality above is still true, and the polygon $P^t$ is split in two polygons, each one with at most $n-1$ edges. See Figure \ref{fig:rem-edge} (left). We choose the one which has the lowest eigenvalue and repeat the previous argument.
\begin{figure}\centering 
	\includegraphics[height=0.4\textwidth]{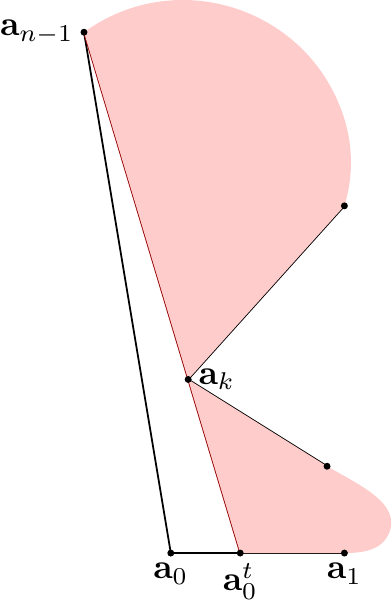}
    \includegraphics[height=0.4\textwidth]{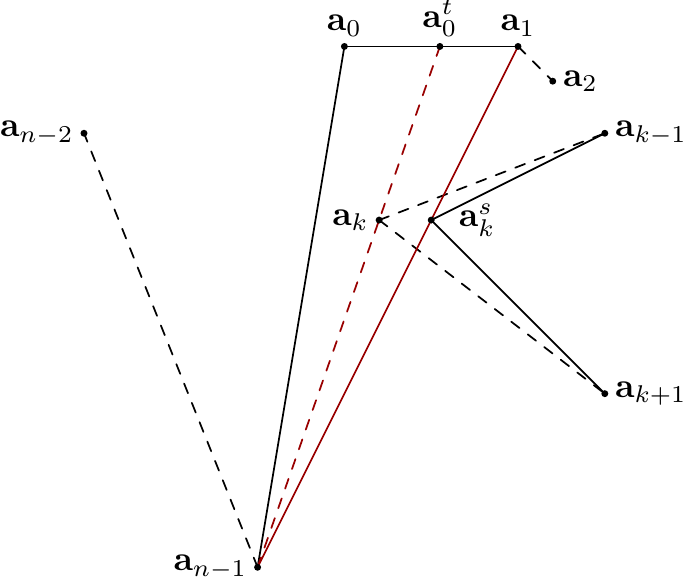}
    \caption{Modification of the polygon for removing a small edge: the case of a convex angle $\widehat{\bo a_0}$ (left), the case of two concave angles $\widehat{\bo a_0}, \widehat{\bo a_1}$ (right)}
    \label{fig:rem-edge}
\end{figure}

If both angles  $\widehat{\bo a_0}, \widehat{\bo a_1}$  are concave, we consider the same type of movement $\bo a_0^t= (1-t) \bo a_0+t\bo a_1$. If the segment $[\bo a_{n-1}\bo a_0^t]$ does not meet any other vertex of the polygon for any $t\in(0,1)$, then we denote $Q$ the new polygon obtained for $t=1$ and follow the previous argument. The difference occurs if $[\bo a_{n-1}\bo a_0^t]$ meets a vertex, say $\bo a_k$. In this case the angle $\widehat{\bo a_k}$ is convex. However, there is no splitting in this case. We continue the movement, moving at the same time $\bo a_k$  {parallel to $[\bo a_{n-1}\bo a_0]$, denoted $\bo a_k^s$ See Figure \ref{fig:rem-edge} (right)}.
If the movement finishes at $t=1$, we apply previous argument. If the movement blocks because the  {segments $[\bo a_{k-1}\bo a_k^s]$, $[\bo a_k^s, \bo a_{k+1}]$ touch another vertex} then the polygon  {splits, since we move a convex angle towards the interior of the polygon, and we stop following the same argument as in the previous case}. If it blocks because $[\bo a_{n-1}\bo a_0^t]$ meets another vertex, we treat it the same way as $\bo a_k$ and continue the movement. The blocking can occur at most $n-3$ times. 
\end{proof}

\medskip
Below we show that a strategy to prove the conjecture by a finite number of numerical computations can be followed, provided of course that the conjecture is true. However, from a practical point of view, this strategy is far from being optimal.

\begin{thm}
Provided the conjecture is true, for every $n \ge 5$ its proof can be reduced to a finite number of numerical computations.
\end{thm}
 
\begin{proof}

We know the inequality is true for $n=4$. Assume now that the inequality is true for polygons with up to $n-1$ edges. Then recalling the notation $\Bbb{P}_n=[\bo a_0^*...\bo a_{n-1}^*]$ for the regular polygon inscribed in the unit circle with one vertex at $\bo a_0^*=(1,0)$ we have  $l_{n-1}^* = \lambda_1(\Bbb P_{n-1})|\Bbb P_{n-1}|$.
We have a certified estimate from above and from below for this value. In order to prove the conjecture for $n$ edges, we shall analyze problem \eqref{bobu203}  in the following steps. 
\smallskip

 \noindent  {\bf Step 1.} Compute a certified approximation of the first eigenpair $(\lb_1, u_1)$ on $\Bbb P_{n}$. The certified approximation of the eigenfunction $u_1$ holds in $H^1_0(\Bbb P_n)$.

\noindent {\bf Step 2.} For the regular polygon $\Bbb{P}_n$ inscribed in the unit circle, having the vertex $\bo a_0 = (1,0)$ we compute the spectrum of the shape Hessian of $\lambda_1(\Bbb P_n)|\Bbb P_n|$, and we certify the positivity for $2n-4$ of its eigenvalues using results from Sections \ref{sec:eig-hessian}, \ref{sec:estimates-coeffs}. This concludes that the regular polygon is a local minimum.

From now on, we identify polygons $P=[\bo a_0\bo a_1\dots \bo a_{n-1}]$, with $\bo a_i=(x_i,y_i)$, by a point in $ \R^{2n-4}$ having coordinates $(x_2,y_2, \dots, x_{n-1}, y_{n-1})$. We consider the first two points fixed: $\bo a_0=\bo a_0^*, \bo a_1=\bo a_1^*$. Without restricting generality, we can assume that the edge $[\bo a_0\bo a_1]$ is the longest edge in the polygon $P$. Let us denote $\mathcal P$ the family of such polygons of $\ov {\mathcal P}_n$, identified as a compact subset in $\R^{2n-4}$. 

\noindent {\bf Step 3.} Compute, using Theorem \ref{bobu43}, a neigbourhood of $\Bbb P_{n}$ in $\R^{2n-4}$, where 
$$|P|\lb_1(P) \ge |\Bbb P_n|\lb_1(\Bbb P_n).$$
 Precisely, for a value $\vps_0>0$ we have  $|P| \lb_1(P) \ge |\Bbb{P}_n| \lb_1(\Bbb P_n)$  for every $P=[\bo a_0\dots \bo a_{n-1}]$, with   $\bo a_0=\bo a_0^*$, $\bo a_1=\bo a_1^*$,   such that $\forall i=2, n-1$, $ |\bo a_i\bo a_i^*| \le \vps _0$. Of course, in order to obtain $\vps_0$, the availability of the constants $C$ and $ \vartheta$ in Theorem \ref{bobu43} is assumed.  Let us denote $\mathcal L_n$ this neighbourhood, which is a closed set.

\noindent {\bf Step 4.} Using Theorem \ref{bobu47} find an estimate for the minimal measure of an optimal polygon in the class $\mathcal P$. Here we use the fact that the maximal length of an edge is precisely $[\bo a_0\bo a_1]$. Then, we get 
$$\sqrt{\frac{|\Bbb P_n|}{|P|}}  
|\bo a_0^* \bo a_1^*|\le D_{max}.$$
Using Makai's inequality we get a lower bound for the inradius  of an optimal $n$-gon (called $\rho_{\min}$ in the sequel), since
$$\frac{1}{\rho_P^2}|P|\le |P|\lb_1(P).$$
 In particular, if $\rho_P^2 \geq |P|/l_n^*$ then $P$ is cannot be optimal.
Using Theorem \ref{bobu220}, we obtain a lower bound on the shortest edge, $e_{\min}$. 

All these three geometric constraints: measure, inradius and shortest edge  generate a smaller compact set $\mathcal P' $, defined by purely geometric constraints, such that   $\mathcal P' \sq \overline {\mathcal   P}$  in $\R^{2n-4}$. In particular, the lower bound on the inradii of such polygons, makes that the inequality in Lemma \ref{bobu223} becomes uniform.  Below we work with $\delta \le  \frac{\rho_{\min}}{4}<1$. Indeed, for every $P, Q$ in the class $\mathcal P' $ the value $ \lb_1(B_{\rho_{\min}- \delta})$ is an upper bound for 
$\lb_1(P), \lb_1(Q),  \lb_1(P\cap Q)$:  for $P,Q \in \mathcal P'$, there exists a universal constant $K$ (with explicit value, issued from Lemma \ref{bobu223})  such that if the distance between the respective vertices is at most $\delta$ then $|\lambda_1(P)-\lambda_1(Q)|\leq K\delta^\frac 12$. 

The variation of the area $||P|-|Q||$ is also controlled by a term of the form $K'\delta $, with  $K'= nD_{max}+ n\pi$. There exist universal upper bounds for the first eigenvalue and for the area in $\mathcal P'$. Therefore, there exists an explicit constant $K''$ such that if the distance between the respective vertices of $P, Q\in \mathcal P'$ is at most $\delta$ then 
\begin{equation}
|\lambda_1(P)|P|-\lambda_1(Q)|Q||\leq 
 |P|(\lambda_1(P)-\lambda_1(Q))+\lambda_1(Q)||P|-|Q|| \leq  K''\delta^\frac 12.
\label{eq:estimate-delta-prod}
\end{equation}

\noindent  {\bf Step 5.} {Suppose the conjecture is true and $\Bbb{P}_n$ is the only minimizer for $P \mapsto |P|\lambda_1(P)$ in $\mathcal P'$. Then, in view of {\bf Step 2.} above, where a local minimality neighborhood was identified around $\Bbb{P}_n$, there exists $\varepsilon_1>0$ such that outside the neigbourhood $\mathcal L_n$ of $\Bbb{P}_n$ in $\Bbb{R}^{2n-4}$ we have $|P|\lambda_1(P)>l_n^*+\varepsilon_1$. If such an $\varepsilon_1$ cannot be found, then a minimizing sequence which does not converge to $\Bbb{P}_n$ could be constructed, contradicting the hypothesis that the conjecture is valid.}

  Consider $\delta>0$ such that $K''(2\delta)^\frac 12<\varepsilon_1/4$, with $K''$ from \eqref{eq:estimate-delta-prod}. Moreover, suppose that $2\delta\leq \delta_0$ with $\delta_0$ from \eqref{eq:delta-edge-min}. We cover the compact set $\mathcal P'\setminus \mathcal L_n$ with at most $c_{2n-4} \Big( \frac{D_{max}}{\delta}\Big) ^{2n-4}$ balls $(B_j)_{j \in J}$ of radius $\delta$, where $c_{2n-4}$ is a dimensional constant. Several estimates of  $c_{2n-4}$ are available, a non optimal one being $\Big (\frac{\sqrt{2n-4}D_{max}}{2 \delta}\Big) ^{2n-4}$. 

Choose one of the balls $B_j$ enumerated above. Take an admissible polygon $P\in \mathcal P'\sm \mathcal L_n$ having coordinates $(x_2,y_2,...,x_{n-1},y_{n-1})$ in the ball $B_j$. If such a polygon does not exist, there is nothing to be done and we move to the next ball. We evaluate $|P|\lb_1(P)$ numerically, obtaining a certified estimate interval of length at most $\varepsilon_1/4$. If this certified computation gives
\begin{equation}\label{bobuec}
|P|\lb_1(P)\ge l_n^*+ \frac{\vps_1}{2},
\end{equation}
then $P$ is not optimal and, in view of the choice of the constant $\delta$, no other optimal polygon exists having coordinates in the same ball. If \eqref{bobuec} holds for every ball $B_j$ containing an admissible polygon then the conjecture is solved.

However, the value of $\vps_1$ is not known. For this reason, we start with a value $\vps_1=1$ and perform the computations enumerated above. If inequality \eqref{bobuec} holds every time there is an admissible polygon in one of the balls then the proof succeeded and we stop. If for some polygon the inequality fails, we divide $\vps_1$ by $2$ and restart the computation, etc. This procedure stops in a finite number of steps. Note that from practical point of view, this procedure is completely inefficient, but formally leads to the conclusion.  
\end{proof}

\begin{rem}[Polygonal Saint-Venant inequality]
\label{bobu40}
\rm
Another variational energy of interest is  torsional rigidity. It is denoted by
\begin{equation}
T(\Omega) =   \int_\Omega  wdx, \text{ where } w \text{ verifies } \left\{\begin{array}{rcll}
-\Delta w & = & 1 & \text{ in }\Omega, \\
w & = & 0  &\text{ on } \partial \Omega, 
\end{array}\right.
\label{eq:dirichlet-energy}
\end{equation}
and the problem to consider
\begin{equation}
\max_{P \in \Pol n, |\Omega| = \pi} T(P).
\label{eq:max-dir-energy}
\end{equation}
The Saint-Venant inequality states that the maximum of the torsional rigidity among {\it all} sets of area $\pi$ is achieved on the disc. 
P\' olya  and Szeg\"o have also conjectured in  1951 (see \cite[page 158]{PoSz51}) the following.

\medskip
\noindent {\bf Conjecture. }{\it  
The unique solution to problem \eqref{eq:max-dir-energy} is the regular polygon with $n$ sides and area $\pi$.
	}
\medskip

All the results we have obtained for  the eigenvalue transfer similarly to the conjecture above. However, this conjecture is computationally less challenging than the eigenvalue. In particular, there is no additional normalization and orthogonality constraints for $w$ and for the associated material derivatives. The proof of the local maximality goes through the computation of the Hessian matrix of \eqref{eq:dirichlet-energy} on the regular polygon. The expression of its coefficients was obtained by Laurain in \cite{laurain2ndDeriv}. 
Recalling that the functions $\varphi_i$ are constructed in  \eqref{def:phi},
one introduces the functions $\bo U_i \in H_0^1(P,\Bbb{R}^2),\ i=0,...,n-1$
\begin{equation}
\int_P D\bo U_i \nabla v = \int_P -(\nabla \varphi_i \otimes \nabla w)\nabla v + 2 (\nabla w \odot \nabla v)\nabla \varphi_i + \int_P v \nabla \varphi_i,\ \text{ for every } v \in H_0^1(P).
\label{eq:material-dirichlet-decomposed}
\end{equation}
The  following result is proved by Laurain in \cite[Proposition 14]{laurain2ndDeriv}: the Hessian matrix $\bo T \in \Bbb{R}^{2n\times 2n}$ of the torsional rigidity \eqref{eq:dirichlet-energy} with respect to the coordinates of the $n$-gon is given by the following $n\times n$ block matrix 
	\[ \bo T = (\bo T_{ij})_{0\leq i,j\leq n-1}\]
	where the $2\times 2$ blocks are given by
	\begin{align}
	\bo T_{ij} & = \int_P D\bo U_i D\bo U_j^T + \nabla \varphi_i \otimes \bo S_1^D \nabla \varphi_j + \bo S_1^D \nabla \varphi_i \otimes \nabla \varphi_j  \notag \\
	& +\int_P \left(\frac{1}{2}|\nabla w|^2 - w\right) (2 \nabla \varphi_i \odot \nabla \varphi_j)  \notag \\
	& +\int_P-(\nabla \varphi_j\cdot \nabla w)(\nabla \varphi_i \otimes \nabla w)-(\nabla \varphi_i \cdot \nabla w)(\nabla w \otimes \nabla \varphi_j)-(\nabla \varphi_i \cdot \nabla \varphi_j)(\nabla w \otimes \nabla w)
	\label{eq:Hessian-Dir-block}
	\end{align}
	where $\bo U_i, i=0,...,n-1$ are solutions of \eqref{eq:material-dirichlet-decomposed} and $ \bo S_1^D = (-1/2|\nabla w|^2 +w)\Id +\nabla w \otimes \nabla w$.
	\label{thm:hessian-torsion}

\end{rem}

\smallskip
\begin{ack}
The authors  have been supported by the ANR Shapo (ANR-18-CE40-0013) programme. The first author wishes to thank Pierre Jolivet for valuable advice regarding the large scale computations in FreeFEM.  The second author wishes to thank M. Van den Berg for useful suggestions concerning the bound of the torsion function in Theorem \ref{bobu47}. The parallel computations were performed on the Choleski server at the IDCS mesocentre within the Institut Polytechnique de Paris.
\end{ack}

\bibliographystyle{abbrv}
\bibliography{./biblio.bib}

\appendix

\section{Proof of Proposition \ref{prop:detailed-computations}}
\label{appendix:computations}

Recall that functions $\varphi_j$ are associated to a symmetric triangulation $T_j$, $0 \leq j \leq n-1$. The gradients of $\varphi_j$ are expressed in \eqref{eq:grad-phi}. 

First diagonal term from the real part:
\begin{align*}
a(U_0^1,\sum_{j=0}^{n-1} \cos(jk\theta) U_0^1 \circ R_{j\theta}^T) &  = \int_\Omega \sum_{j=0}^{n-1} \cos(jk\theta)(\nabla \varphi_j\cdot \nabla U_0^1)(\cos(j\theta) \partial_x u_1+\sin (j\theta)\partial_y u_1) \\
&+ \int_\Omega \sum_{j=0}^{n-1}\cos(jk\theta) (\nabla \varphi_j\cdot \nabla u_1)(\cos(j\theta) \partial_x U_0^1 +\sin(j\theta) \partial_y U_0^1) \\
& = \int_\Omega \sum_{j=0}^{n-1} (2 \cos(jk\theta)\cos(j\theta)\partial_x \varphi_j) \partial_x u_1 \partial_x U_0^1 \\
& + \int_\Omega \sum_{j=0}^{n-1}( 2\cos(jk\theta)\sin(j\theta)\partial_y \varphi_j) \partial_y u_1 \partial_y U_0^1 \\
& + \int_\Omega \sum_{j=0}^{n-1} (\cos(jk\theta)\cos(j\theta)\partial_y \varphi_j+\cos(jk\theta)\sin(j\theta)\partial_x \varphi_j) \partial_y u_1 \partial_x U_0^1 \\
& + \int_\Omega \sum_{j=0}^{n-1} (\cos(jk\theta)\cos(j\theta)\partial_y \varphi_j+\cos(jk\theta)\sin(j\theta)\partial_x \varphi_j) \partial_x u_1 \partial_y U_0^1\\
\end{align*}
\[=\frac{1}{\sin \theta}\sum_{j=0}^{n-1}2 (\cos jk\theta \cos j \theta \sin(j+1)\theta -\cos(j+1)k\theta \cos (j+1)\theta \sin j\theta) \int_{T_j} \partial_x u_1 \partial_x U_0^1\]
\[ +\frac{1}{\sin \theta}\sum_{j=0}^{n-1} 2(\cos(j+1)k\theta \sin(j+1)\theta \cos j \theta-\cos jk\theta \sin j \theta \cos(j+1)\theta)\int_{T_j} \partial_y u_1 \partial_y U_0^1 \]
\[
+\frac{1}{\sin \theta}\sum_{j=0}^{n-1} (\cos(j+1)k\theta-\cos jk\theta)\cos(2j+1)\theta 
\int_{T_j} ( \partial_x u_1 \partial_y U_0^1 + \partial_y u_1 \partial_x U_0^1)
\]
\[
=\frac{2}{\sin \theta} \sum_{j=0}^{n-1} 
\frac{1}{2} \left(\sin(2j+1)\theta (\cos jk\theta-\cos(j+1)k\theta)+\sin\theta (\cos jk\theta+\cos (j+1)k\theta) \right) \int_{T_j} \partial_x u_1 \partial_x U_0^1
\]
\[+
\frac{2}{\sin \theta} \sum_{j=0}^{n-1} 
\frac{1}{2} \left(\sin(2j+1)\theta (\cos (j+1)k\theta-\cos j k\theta)
+\sin\theta (\cos jk\theta+\cos (j+1)k\theta)
\right)
\int_{T_j} \partial_y u_1 \partial_y U_0^1
\]
\[
+\frac{1}{\sin \theta}\sum_{j=0}^{n-1} (\cos(j+1)k\theta-\cos jk\theta)\cos(2j+1)\theta 
\int_{T_j} ( \partial_x u_1 \partial_y U_0^1 + \partial_y u_1 \partial_x U_0^1)
\]
\[ = \sum_{j=0}^{n-1} (\cos jk\theta+\cos (j+1)k\theta) \int_{T_j}\nabla u_1 \cdot \nabla U_0^1 
\]
\[
+ \sum_{j=0}^{n-1} \frac{\cos(j+1)k\theta-\cos jk\theta}{\sin \theta} \int_{T_j}\begin{pmatrix}
-\sin (2j+1)\theta & \cos (2j+1) \theta \\
\cos(2j+1)\theta & \sin (2j+1)\theta
\end{pmatrix}\nabla u_1 \cdot \nabla U_0^1 
\]

Second diagonal term from the real part:
\begin{align*}
a(U_0^2,\sum_{j=0}^{n-1} \cos(jk\theta) U_0^2 \circ R_{j\theta}^T) &  = \int_\Omega \sum_{j=0}^{n-1} \cos(jk\theta)(\nabla \varphi_j\cdot \nabla U_0^2)(-\sin(j\theta) \partial_x u_1+\cos (j\theta)\partial_y u_1) \\
&+ \int_\Omega \sum_{j=0}^{n-1}\cos(jk\theta) (\nabla \varphi_j\cdot \nabla u_1)(-\sin(j\theta) \partial_x U_0^2 +\cos(j\theta) \partial_y U_0^2) \\
& = \int_\Omega \sum_{j=0}^{n-1} (-2 \cos(jk\theta)\sin(j\theta)\partial_x \varphi_j) \partial_x u_1 \partial_x U_0^2 \\
& + \int_\Omega \sum_{j=0}^{n-1}( 2\cos(jk\theta)\cos(j\theta)\partial_y \varphi_j) \partial_y u_1 \partial_y U_0^2 \\
& + \int_\Omega \sum_{j=0}^{n-1} (-\cos(jk\theta)\sin(j\theta)\partial_y \varphi_j+\cos(jk\theta)\cos(j\theta)\partial_x \varphi_j) \partial_y u_1 \partial_x U_0^2 \\
& + \int_\Omega \sum_{j=0}^{n-1} (-\cos(jk\theta)\sin(j\theta)\partial_y \varphi_j+\cos(jk\theta)\cos(j\theta)\partial_x \varphi_j) \partial_x u_1 \partial_y U_0^2\\
\end{align*}
\[=\frac{1}{\sin \theta}\sum_{j=0}^{n-1}2 (\cos(j+1)k\theta-\cos jk\theta)\sin(j+1)\theta\sin j\theta \int_{T_j} \partial_x u_1 \partial_x U_0^2\]
\[ +\frac{1}{\sin \theta}\sum_{j=0}^{n-1} 2(\cos(j+1)k\theta-\cos jk\theta)\cos j \theta\cos (j+1)\theta \int_{T_j} \partial_y u_1 \partial_y U_0^2 \]
\[
-\frac{1}{\sin \theta}\sum_{j=0}^{n-1} (\cos(j+1)k\theta-\cos jk\theta)\sin(2j+1)\theta 
\int_{T_j} ( \partial_x u_1 \partial_y U_0^2 + \partial_y u_1 \partial_x U_0^2)
\]
\[=
\frac{2}{\sin \theta}\sum_{j=0}^{n-1} (\cos(j+1)k\theta - \cos jk\theta ) \frac{\cos \theta-\cos(2j+1)\theta}{2} \int_{T_j}\partial_x u_1 \partial_x U_0^2
\]
\[+
\frac{2}{\sin \theta}\sum_{j=0}^{n-1} (\cos(j+1)k\theta - \cos jk\theta ) \frac{\cos \theta+\cos(2j+1)\theta}{2} \int_{T_j}\partial_y u \partial_y U_0^2
\]
\[
-\frac{1}{\sin \theta}\sum_{j=0}^{n-1} (\cos(j+1)k\theta-\cos jk\theta)\sin(2j+1)\theta 
\int_{T_j} ( \partial_x u_1 \partial_y U_0^2 + \partial_y u \partial_x U_0^2)
\]
\[
=\frac{\cos \theta}{\sin \theta}\sum_{j=0}^{n-1} (\cos(j+1)k\theta-\cos jk\theta) \int_{T_j} \nabla u_1 \cdot \nabla U_0^2
\]
\[
+ \sum_{j=0}^{n-1} \frac{\cos(j+1)k\theta-\cos jk\theta}{\sin \theta} \int_{T_j}
\begin{pmatrix}
-\cos(2j+1)\theta & -\sin(2j+1)\theta \\ 
-\sin(2j+1)\theta & \cos(2j+1)\theta 
\end{pmatrix}\nabla u_1 \cdot \nabla U_0^2
\]

Term on position $(1,2)$ from the imaginary part:
\begin{align*}
a(U_0^1,\sum_{j=0}^{n-1} \sin(jk\theta) U_0^2 \circ R_{j\theta}^T) &  = \int_\Omega \sum_{j=0}^{n-1} \sin(jk\theta)(\nabla \varphi_j\cdot \nabla U_0^1)(-\sin(j\theta) \partial_x u_1+\cos (j\theta)\partial_y u_1) \\
&+ \int_\Omega \sum_{j=0}^{n-1}\sin(jk\theta) (\nabla \varphi_j\cdot \nabla u_1)(-\sin(j\theta) \partial_x U_0^1 +\cos(j\theta) \partial_y U_1^1) \\
& = \int_\Omega \sum_{j=0}^{n-1} (-2 \sin(jk\theta)\sin(j\theta)\partial_x \varphi_j) \partial_x u_1 \partial_x U_0^1 \\
& + \int_\Omega \sum_{j=0}^{n-1}( 2\sin(jk\theta)\cos(j\theta)\partial_y \varphi_j) \partial_y u_1 \partial_y U_0^1 \\
& + \int_\Omega \sum_{j=0}^{n-1} (-\sin(jk\theta)\sin(j\theta)\partial_y \varphi_j+\sin(jk\theta)\cos(j\theta)\partial_x \varphi_j) \partial_y u_1 \partial_x U_0^1 \\
& + \int_\Omega \sum_{j=0}^{n-1} (-\sin(jk\theta)\sin(j\theta)\partial_y \varphi_j+\sin(jk\theta)\cos(j\theta)\partial_x \varphi_j) \partial_x u_1 \partial_y U_0^1\\
\end{align*}
\[=\frac{1}{\sin \theta}\sum_{j=0}^{n-1}2 (\sin(j+1)k\theta-\sin jk\theta)\sin(j+1)\theta\sin j\theta \int_{T_j} \partial_x u_1 \partial_x U_0^1\]
\[ +\frac{1}{\sin \theta}\sum_{j=0}^{n-1} 2(\sin(j+1)k\theta-\sin jk\theta)\cos j \theta\cos (j+1)\theta \int_{T_j} \partial_y u_1 \partial_y U_0^1 \]
\[
-\frac{1}{\sin \theta}\sum_{j=0}^{n-1} (\sin(j+1)k\theta-\sin jk\theta)\sin(2j+1)\theta 
\int_{T_j} ( \partial_x u_1 \partial_y U_0^1 + \partial_y u_1 \partial_x U_0^1)
\]

$$=\frac{1}{\sin \theta}\sum_{j=0}^{n-1}2 (\sin(j+1)k\theta-\sin jk\theta)\frac{\cos\theta-\cos (2j+1) \theta}{2}
\int_{T_j} \partial_x u_1 \partial_x U_0^1$$
$$ +\frac{1}{\sin \theta}\sum_{j=0}^{n-1} 2(\sin(j+1)k\theta-\sin jk\theta)\frac{\cos\theta+\cos (2j+1) \theta}{2} \int_{T_j} \partial_y u_1 \partial_y U_0^1 $$
$$-\frac{1}{\sin \theta}\sum_{j=0}^{n-1} (\sin(j+1)k\theta-\sin jk\theta)\sin(2j+1)\theta 
\int_{T_j} ( \partial_x u_1 \partial_y U_1^1 + \partial_y u_1 \partial_x U_0^1)
$$
$$= \frac{\cos \theta}{\sin \theta}\sum_{j=0}^{n-1} (\sin(j+1)k\theta-\sin jk\theta)  \int_{T_j} \nabla u_1 \nabla U_0^1$$
$$+ \sum_{j=0}^{n-1}\dfrac{\sin(j+1)k\theta-\sin jk\theta}{\sin \theta} \int _{T_j} \begin{pmatrix}
-\cos (2j+1)\theta & -\sin (2j+1) \theta \\
-\sin (2j+1)\theta & \cos (2j+1)\theta
\end{pmatrix}\nabla u_1 \cdot \nabla U_0^1.$$

Term on position $(2,1)$ form imaginary part:
\begin{align*}
a(U_0^2,\sum_{j=0}^{n-1} \sin(jk\theta) U_0^1 \circ R_{j\theta}^T) &  = \int_\Omega \sum_{j=0}^{n-1} \sin(jk\theta)(\nabla \varphi_j\cdot \nabla U_0^2)(\cos(j\theta) \partial_x u_1+\sin (j\theta)\partial_y u_1) \\
&+ \int_\Omega \sum_{j=0}^{n-1}\sin(jk\theta) (\nabla \varphi_j\cdot \nabla u)(\cos(j\theta) \partial_x U_0^2 +\sin(j\theta) \partial_y U_0^2) \\
& = \int_\Omega \sum_{j=0}^{n-1} (2 \sin(jk\theta)\cos(j\theta)\partial_x \varphi_j) \partial_x u_1 \partial_x U_0^2 \\
& + \int_\Omega \sum_{j=0}^{n-1}( 2\sin(jk\theta)\sin(j\theta)\partial_y \varphi_j) \partial_y u_1 \partial_y U_0^2 \\
& + \int_\Omega \sum_{j=0}^{n-1} (\sin(jk\theta)\cos(j\theta)\partial_y \varphi_j+\sin(jk\theta)\sin(j\theta)\partial_x \varphi_j) \partial_y u_1 \partial_x U_0^2 \\
& + \int_\Omega \sum_{j=0}^{n-1} (\sin(jk\theta)\cos(j\theta)\partial_y \varphi_j+\sin(jk\theta)\sin(j\theta)\partial_x \varphi_j) \partial_x u_1 \partial_y U_0^2\\
\end{align*}
\[=\frac{1}{\sin \theta}\sum_{j=0}^{n-1}2 (\sin jk\theta \cos j \theta \sin(j+1)\theta -\sin(j+1)k\theta \cos (j+1)\theta \sin j\theta) \int_{T_j} \partial_x u_1 \partial_x U_0^2\]
\[ +\frac{1}{\sin \theta}\sum_{j=0}^{n-1} 2(\sin(j+1)k\theta \sin(j+1)\theta \cos j \theta-\sin jk\theta \sin j \theta \cos(j+1)\theta)\int_{T_j} \partial_y u_1 \partial_y U_0^2 \]
\[
+\frac{1}{\sin \theta}\sum_{j=0}^{n-1} (\sin(j+1)k\theta-\sin jk\theta)\cos(2j+1)\theta 
\int_{T_j} ( \partial_x u_1 \partial_y U_0^2 + \partial_y u_1 \partial_x U_0^2)
\]
$$=\frac{1}{\sin \theta}\sum_{j=0}^{n-1}[ (\sin jk\theta  + \sin(j+1)k\theta)\sin \theta + \sin (2j+1)\theta ( \sin jk\theta  -\sin(j+1)k\theta) ]\int_{T_j} \partial_x u_1 \partial_x U_0^2 $$
$$+ \frac{1}{\sin \theta}\sum_{j=0}^{n-1}[ (\sin jk\theta  + \sin(j+1)k\theta)\sin \theta + \sin (2j+1)\theta ( -\sin jk\theta  +\sin(j+1)k\theta) ]\int_{T_j} \partial_
y u_1 \partial_y U_0^2  $$
$$ +\frac{1}{\sin \theta}\sum_{j=0}^{n-1} (\sin(j+1)k\theta-\sin jk\theta)\cos(2j+1)\theta 
\int_{T_j} ( \partial_x u_1 \partial_y U_0^2 + \partial_y u_1 \partial_x U_0^2)$$
$$= \sum_{j=0}^{n-1}(\sin jk\theta  + \sin(j+1)k\theta) \int_{T_j} \nabla u_1 \nabla U_0^2 $$
$$+ \sum_{j=0}^{n-1}\dfrac{\sin(j+1)k\theta-\sin jk\theta}{\sin \theta} \int _{T_j} \begin{pmatrix}
-\sin (2j+1)\theta & \cos (2j+1) \theta \\
\cos (2j+1)\theta & \sin (2j+1)\theta
\end{pmatrix}\nabla u_1 \cdot \nabla U_0^2 $$
\end{document}